\documentclass[12pt, reqno]{amsart}

\usepackage[usenames,dvipsnames]{xcolor}
\usepackage{amssymb, amsmath, amsthm, amsfonts, mathrsfs, mathtools}
\usepackage[backref, colorlinks = true,  linkcolor = blue, citecolor = Green]{hyperref}
\usepackage[alphabetic,backrefs,lite]{amsrefs}
\usepackage{amscd}   % for commutative diagrams
\usepackage{fullpage}
\usepackage{tikz-cd} % for complicated commutative diagrams
\usepackage[all,cmtip]{xy}
\usepackage{enumitem}

\DeclareFontEncoding{OT2}{}{} % to enable usage of cyrillic fonts

\newtheorem{lemma}{Lemma}[section]
\newtheorem{theorem}[lemma]{Theorem}
\newtheorem{question}[lemma]{Question}

\newtheorem{prop}[lemma]{Proposition}
\newtheorem{cor}[lemma]{Corollary}

\newtheorem{claim*}{Claim}
\newtheorem{thm}[lemma]{Theorem}
\newtheorem{defn}[lemma]{Definition}
\newtheorem{example}[lemma]{Example}

\theoremstyle{remark}
\newtheorem{remark}[lemma]{Remark}

\newtheorem{rmk}[lemma]{Remark}
\newtheorem{rmks}[lemma]{Remarks}

\newcommand{\A}{{\mathbb A}}
\newcommand{\G}{{\mathbb G}}

\newcommand{\PP}{{\mathbb P}}
\newcommand{\C}{{\mathbb C}}
\newcommand{\F}{{\mathbb F}}
\newcommand{\Q}{{\mathbb Q}}

\newcommand{\Z}{{\mathbb Z}}

\newcommand{\Xbar}{{\overline{X}}}
\newcommand{\Qbar}{{\overline{\Q}}}

\newcommand{\kbar}{{\overline{k}}}

\newcommand{\Fbar}{{\overline{\F}}}

\newcommand{\Ybar}{{\overline{Y}}}

\newcommand{\kk}{{\mathbf k}}

\newcommand{\elltilde}{\widetilde{\ell}}

\newcommand{\calA}{{\mathcal A}}

\newcommand{\calC}{{\mathcal C}}

\newcommand{\calG}{{\mathcal G}}

\newcommand{\calQ}{{\mathcal Q}}

\newcommand{\OO}{{\mathcal O}}

\usepackage[mathscr]{euscript}

\newcommand{\scrS}{{\mathscr S}}
\newcommand{\scrT}{{\mathscr T}}

\newcommand{\scrX}{{\mathscr X}}

\DeclareMathOperator{\HH}{H}

\DeclareMathOperator{\Char}{char}
\DeclareMathOperator{\inv}{inv}

\DeclareMathOperator{\Aut}{Aut}
\DeclareMathOperator{\Gal}{Gal}

\DeclareMathOperator{\Cor}{Cor}
\DeclareMathOperator{\Res}{Res}
\DeclareMathOperator{\Norm}{Norm}
\DeclareMathOperator{\Br}{Br}

\DeclareMathOperator{\Sym}{Sym}

\DeclareMathOperator{\Pic}{Pic}
\DeclareMathOperator{\Jac}{Jac}

\DeclareMathOperator{\Spec}{Spec}

\DeclareMathOperator{\PGL}{PGL}

\DeclareMathOperator{\N}{N}

\DeclareMathOperator{\rank}{rank}

\DeclareMathOperator{\GL}{GL}

\DeclareMathOperator{\M}{M}

\DeclareMathOperator{\disc}{disc}

\DeclareMathOperator{\mult}{mult}

\DeclareMathOperator{\Clif}{Clif}

\DeclareMathOperator{\Bl}{Bl}

\newcommand{\isom}{\cong}

\newcommand{\eps}{\varepsilon}

\newcommand{\CQ}[1]{\calC_{#1}}

\numberwithin{equation}{section}
\numberwithin{table}{section}

\newcommand{\defi}[1]{\textsf{#1}} % for defined terms

\title{Quadratic points on intersections of two quadrics}
\author{Brendan Creutz}
\author{Bianca Viray}
\address{School of Mathematics and Statistics, University of Canterbury, Private Bag 4800, Christchurch 8140, New Zealand}
\email{brendan.creutz@canterbury.ac.nz}
\urladdr{http://www.math.canterbury.ac.nz/\~{}b.creutz}

\address{University of Washington, Department of Mathematics, Box 354350, Seattle, WA 98195,~USA}
\email{bviray@uw.edu}
\urladdr{http://math.washington.edu/\~{}bviray}

\begin{document}

\date{}

\begin{abstract}
We prove that a smooth complete intersection of two quadrics of dimension at least $2$ over a number field has index dividing $2$, i.e., that it possesses a rational $0$-cycle of degree $2$.
\end{abstract}
	
\maketitle
%\tableofcontents

%%%%%%%%%%%%%%%%%%%%%%%%%%%%%%%%%%%%%%%%%%%%%%%%%%%%%%%%%%%%%%%%%%%%%%%%%%%%
\section{Introduction}
%%%%%%%%%%%%%%%%%%%%%%%%%%%%%%%%%%%%%%%%%%%%%%%%%%%%%%%%%%%%%%%%%%%%%%%%%%%%

The \defi{index} of a variety over a field $k$ is the greatest common divisor of the degrees $[\kk(x):k]$ ranging over the residue fields $\kk(x)$ of the (zero-dimensional) closed points $x$ of the variety. Equivalently, the index is the smallest positive degree of a $k$-rational $0$-cycle. 

Let $X \subset \PP^n_k$ be a smooth complete intersection of two quadrics over a field $k$ of characteristic not equal to $2$. Then the index of $X$ necessarily divides $4$, because intersecting with a plane yields a $0$-cycle of degree $4$. In general, this is the best possible bound. Indeed, there are examples with index $4$ over local and global fields when $n = 3$ \cite{LangTate}*{Theorem 7} and over fields of characteristic $0$ when $n = 4$, as we show in Theorem~\ref{thm:ind4}.

Our main result is the following sharp bound on the index when $n\geq 4$ and $k$ is a number field or a local field.

\begin{theorem}\label{thm:MainIndThm}
	Let $X$ be a smooth complete intersection of two quadrics in $\PP^n_k$ with $n \ge 4$ and assume that $k$ is either a number field or a local field. Then the index of $X$ divides $2$.
\end{theorem}

This result allows us to complete the list of integers which occur as the index of a del Pezzo surface over a local field or a number field (See Section~\ref{sec:dP}). It also allows us to deduce nontrivial index bounds for other interesting classes of varieties. In particular, if $C/k$ is a genus $2$ curve over a number field with a rational Weierstrass point, then it follows from the result above that any torsor of period $2$ under the Jacobian of $C$ has index dividing $8$ (see Theorem~\ref{thm:PI}) and the corresponding Kummer variety, which is an intersection of $3$ quadrics in $\PP^5$, has index dividing $4$ (see Remark~\ref{rmk:PI}). Again, these results fail for arbitrary fields (see Remark~\ref{rmk:PI}). Theorem~\ref{thm:MainThm} below shows that Theorem~\ref{thm:MainIndThm} also holds for global function fields of odd characteristic when $n \ge 5$ and conditionally in a number of cases when $n \ge 4$.

Theorems of Amer, Brumer and Springer \cites{Amer,Brumer,Springer} show that, for $X$ as above, index $1$ is equivalent to the existence of a $k$-rational point. Analogously one can ask if index $2$ implies the existence of a closed point of degree $2$. Colliot-Th\'el\`ene has recently sketched an argument that if $X$ is a smooth complete intersection of two quadrics in $\PP^4$ over a field of characteristic $0$ and $X$ has index $2$, then $X$ has a closed point of degree $14, 6$ or $2$. {Our next result identifies conditions under which we can prove that a smooth intersection of two quadrics in $\PP^n$ has a closed point of degree $2$. In order to state it we introduce the following notation: We say that a global field $k$ satisfies ($\star$) if Brauer-Manin is the only obstruction to the Hasse principle for del Pezzo surfaces of degree $4$ over all quadratic extensions of $k$.}

\begin{theorem}\label{thm:MainThm}
    Let $n \ge 4$ and let $X \subset \PP^n_k$ be a smooth complete intersection of two quadrics over a field $k$. In any of the following cases there is a quadratic extension $K/k$ such that $X(K)\neq \emptyset$:
    \begin{enumerate}
        \item\label{case-local} $k$ is a local field and $n \ge 4$;
        \item \label{case-globalfunction-n5} $k$ is a global function field and $n \ge 5$;
        \item \label{case-globalchar2} $k$ is a global function field of characteristic $2$ and $n = 4$;
        \item \label{case-number-n5}$k$ is a number field that satisfies {Schinzel's hypothesis} and $n \ge 5$;
        \item\label{case-global-n4} $k$ is a global field that satisfies ($\star$) {or a number field that satisfies Schinzel's hypothesis}, $n = 4$ and the following holds: for any quadratic field extension $L/k$ and rank $4$ quadric $\calQ\subset \PP^4_L$ such that $X = \cap_{\sigma\in \Gal(L/k)}\sigma(\calQ)$ and  $\Norm_{L/k}\left(\disc(\calQ)\right)\in k^{\times2}$, we have that $\calQ$ fails to have smooth local points at an \emph{even} number of primes of $L$.
    \end{enumerate}
   \end{theorem}

	When $n=4$, there are exactly five rank $4$ quadrics in the pencil of quadrics containing $X$ (see Section~\ref{sec:Correspondences} for details). The condition in case~\eqref{case-global-n4} holds {for most intersections of quadrics} and can be easily checked. In particular, it is satisfied if there is no pair of Galois conjugate rank $4$ quadrics in the pencil or if $X$ has points everywhere locally {(for then any quadric containing $X$ will have points over all completions).} {In fact, if $X$ is assumed everywhere locally solvable, the proofs of our main results become much easier (See Corollary~\ref{cor:k1} and Remark~\ref{rmk:410}).}  For further details of the cases  covered (and not covered) in case~\eqref{case-global-n4}, see Remark~\ref{rmk:RemainingOpenCases} and Section~\ref{sec:opencases}. 
    
    Theorem~\ref{thm:MainThm}\eqref{case-global-n4} naturally raises the question of whether the parity condition is necessary.  We have constructed many examples that fail this parity condition, but in each we have found an ad hoc proof that (\(\star\)) implies the existence of a quadratic point.  Based on our results and this extensive numerical evidence, we expect the following question to have a positive answer. 
    \begin{question}\label{ques:1}
		Does every complete intersection of $2$ quadrics $X \subset \PP^4_k$ over a number field $k$ possess a $K$-rational point for some quadratic extension $K/k$?
	\end{question}

    One can also pose this question for other classes of fields, e.g., \(C_r\) fields.  Over \(C_3\) fields, the question has a negative answer (see Section~\ref{sec:ind4} for examples), but it is open for \(C_2\) fields.

    %%%%%%%%%%%%%%%%%%%%%%%%%%%%%%%%%%%%%%%%%%%%%%%%%%%%%%%%%%%%%%%%%%%%%%%%%%%%
    \subsection{Obstructions to index \texorpdfstring{$1$}{1} over local and global fields}
    %%%%%%%%%%%%%%%%%%%%%%%%%%%%%%%%%%%%%%%%%%%%%%%%%%%%%%%%%%%%%%%%%%%%%%%%%%%%
    Over local and global fields, necessary and sufficient conditions for an intersection of two quadrics to have index $1$ (equivalently, to have a rational point) have been well studied.  When $k$ is a local field and $n \le 7$ there are examples with $X(k) = \emptyset$ (which necessarily have index greater than $1$), while for $n \ge 8$ and $k$ a $p$-adic field, $X(k) \ne \emptyset$ \cite{Demyanov}. For $k$ a number field, Colliot-Th\'el\`ene, Sansuc and Swinnerton-Dyer conjecture that a smooth complete intersection of quadrics in $\PP^n_k$ satisfies the Hasse principle as soon as $n\ge 5$~\cite{CTSSDII}*{\S 16}. For $n\ge 8$, the conjecture is proven in~\cites{CTSSDI,CTSSDII} and this has been extended to $n\ge 7$ by Heath-Brown~\cite{HeathBrown}. The analogue of this conjecture over global function fields of odd characteristic has been established by Tian \cite{Tian}, allowing us to deduce case~\eqref{case-globalfunction-n5} from case~\eqref{case-local} of Theorem~\ref{thm:MainThm}.
    
    When $n = 4$ (in which case $X$ is a del Pezzo surface of degree $4$), the Hasse principle can fail~\cite{BSD}. Colliot-Th\'el\`ene and Sansuc have conjectured that this failure is always explained by the Brauer-Manin obstruction~\cite{CTS1}. {This conjecture implies that all number fields satisfy the condition $(\star)$ appearing in Theorem~\ref{thm:MainThm}\eqref{case-global-n4}.} Most cases of the $n=4$ conjecture have been proven conditionally on Schinzel's hypothesis and the finiteness of Tate-Shafarevich groups of elliptic curves by Wittenberg \cite{Wittenberg}. This also gives a conditional proof of the Hasse principle when $n \ge 5$ as this can be reduced to cases of the $n = 4$ conjecture which are covered by Wittenberg's result.

    %%%%%%%%%%%%%%%%%%%%%%%%%%%%%%%%%%%%%%%%%%%%%%%%%%%%%%%%%%%%%%%%%%%%%%%%%%%%
    \subsection{Outline of the proof of Theorems~\ref{thm:MainIndThm} and \ref{thm:MainThm}}
    %%%%%%%%%%%%%%%%%%%%%%%%%%%%%%%%%%%%%%%%%%%%%%%%%%%%%%%%%%%%%%%%%%%%%%%%%%%%
    Using an argument of Wittenberg \cite{Wittenberg} (which we review in Section~\ref{sec:ProofOfMainIndThm}), we can reduce to the case $n = 4$, when $X$ is a del Pezzo surface of degree $4$. 
    
    In Section~\ref{sec:LocalQuadPts} we prove that any del Pezzo surface of degree $4$ over a local field must have points over some quadratic extension, which proves Theorem~\ref{thm:MainThm}\eqref{case-local} and the local case of Theorem~\ref{thm:MainIndThm}. {Our approach uses the theorems of Amer, Brumer, and Springer to reduce to the case where no integral model of \(X\) has a special fiber that is split (i.e., contains a geometrically integral open subscheme) over a quadratic extension.  We then use semistable models of degree $4$ del Pezzo surfaces, introduced by Tian~\cite{Tian}, to directly show that the remaining types of degree $4$ del Pezzo surfaces obtain points over \emph{every} ramified quadratic extension of \(k\).} 
    
    In Section~\ref{sec:AlternateChar2}, we give an easy generalization of a result in \cite{DD}, showing that, for $k$ a field of characteristic $2$, any del Pezzo surface of degree $4$ obtains a point over \(k^{1/2}\).  For local and global fields of characteristic $2$ we have $[k^{1/2}:k] = 2$, so this proves Theorem~\ref{thm:MainThm}\eqref{case-globalchar2} and gives an alternate proof of Theorem~\ref{thm:MainThm}\eqref{case-local} in characteristic $2$.  Thus, for the remainder of the paper, it suffices to assume that \(k\) is of characteristic different from \(2\).

    % In Sections~\ref{sec:AlternateChar2} and~\ref{sec:Alternate}, we give alternate proofs of Theorem~\ref{thm:MainThm}\eqref{case-local} which work, respectively, in that case that $k$ has characteristic $2$ and in the case that $k$ has odd residue characteristic. The first, an easy generalization of a result in \cite{DD}, shows that \(X(k^{1/2}) \ne \emptyset\) for any del Pezzo surface of degree $4$ over a field $k$ of characteristic $2$. For local and global fields of characteristic $2$ we have $[k^{1/2}:k] = 2$, so this proves Theorem~\ref{thm:MainThm}\eqref{case-globalchar2} and gives an alternate proof of Theorem~\ref{thm:MainThm}\eqref{case-local} in characteristic $2$. The second uses a {period equals index} result of Lichtenbaum~from \cite{Lichtenbaum} which is a consequence of Tate's local duality theorem in the cohomology of elliptic curves. 

   Over a global field, the results of Section~\ref{sec:LocalQuadPts} show that after base change to a suitable quadratic extension $X$ becomes everywhere locally solvable. While it is also true that the Brauer group of $X$ becomes constant after a suitable quadratic extension (this can be deduced from the explicit calculation of $\Br(X)/\Br_0(X)$ in \cite{VAV}), one cannot deduce that Theorem~\ref{thm:MainThm} holds for fields $k$ satisfying $(\star)$ directly from case~\eqref{case-local} in this way because, in general, there is no quadratic extension $K/k$ for which $X_K$ is locally solvable and the Brauer group of $X_K$ is trivial modulo constant algebras (See Example~\ref{ex:ObsToWeakApproximation}).

    To obtain our results when $k$ is a global field of characteristic not equal to $2$ we study the arithmetic of the symmetric square of $X$, which is birational to the variety $\mathcal{G}$ parameterizing lines on the quadrics in the pencil of quadrics in $\PP^4_k$ containing $X$ (see Section~\ref{sec:Correspondences} for more details). In Section~\ref{sec:calG}, we develop the main tools for studying the arithmetic of $\calG$ over a global field. We determine explicit central simple algebras over the function field of $\calG$ representing the Brauer group of $\calG$ modulo constant algebras and then develop techniques to calculate the evaluation maps of these central simple algebras at several types of local points.

	{Theorem~\ref{thm:MainThm}\eqref{case-local} implies that $\calG$ is everywhere locally solvable. The results of Section~\ref{sec:calG}} are used in Section~\ref{sec:MainProofs} to show further that there is always an adelic $0$-cycle of degree $1$ on $\calG$ orthogonal to the Brauer group and, under the hypothesis of Theorem~\ref{thm:MainThm}\eqref{case-global-n4}, that there is an adelic point on $\calG$ orthogonal to the Brauer group. This is perhaps surprising given that in this case the Brauer group of $\calG$ can contain nonconstant algebras and in general can obstruct weak approximation on $\calG$ (see Corollary~\ref{cor:ObsToWeakApproximation} and Example~\ref{ex:ObsToWeakApproximation}). 
   
    The variety of lines on a smooth quadric $3$-fold is a Severi-Brauer $3$-fold, so the arithmetic of $\mathcal{G}$ is amenable to the fibration method, as first observed in~\cite{CTS-Schinzel}. Results of \cite{CTSD} show that, in the number field case, the vanishing of the Brauer-Manin obstruction on $\mathcal{G}$ implies the existence of a $0$-cycle of degree $1$ on $\mathcal{G}$ and, conditionally on Schinzel's hypothesis, a $k$-rational point on $\mathcal{G}$.  This yields a $0$-cycle of degree $2$ on $X$ and, {under the hypothesis of Theorem~\ref{thm:MainThm}\eqref{case-global-n4}}, a quadratic point on $X$ if we assume Schinzel's hypothesis. To the best of our knowledge the function field analogue of these results based on the fibration method have not been established. This prevents us from considering global function fields in the $n = 4$ case of Theorem~\ref{thm:MainIndThm}.
    
    One can ask whether $\textup{index}(\calG) = 1$ always implies that $\calG$ has a rational point (when $k$ is a global field this is equivalent to Question~\ref{ques:1}). Our results do not answer this question, but they do show that a stronger condition on $0$-cycles fails {over $p$-adic fields}. Namely, $\calG$ can contain $0$-cycles of degree $1$ that are \emph{not} rationally equivalent to a rational point (See Remark~\ref{rmk:CTP2}\eqref{rmkCTP2:1}).
    
    To deduce the results in case~\eqref{case-global-n4} of Theorem~\ref{thm:MainThm} assuming that $k$ satisfies $(\star)$ (without assuming Schinzel), we make use of Proposition~\ref{prop:sym2}, which may be of interest in its own right. It relates the Brauer-Manin obstruction on the symmetric square of a variety {that has finite Brauer group (modulo constant algebras)} to the Brauer-Manin obstruction over quadratic extensions. (More generally, in Section~\ref{sec:Extensions} we collect results relating the Brauer-Manin obstruction on a nice variety \(Y\) to the Brauer-Manin over an extension which may also be of independent interest.)  In a similar spirit, we answer a question posed in \cite{CTP} concerning Brauer-Manin obstructions over extensions (see Remarks~\ref{rmk:CTP2}\eqref{rmkCTP2:2}) and give an example of a del Pezzo surface of degree $4$ defined over $\Q$ which, for any finite extension $k/\Q$, has a Brauer-Manin obstruction to the existence of $k$-points if and only if $k$ is of odd degree over $\Q$ (See Section~\ref{sec:BSDexample}).

   %%%%%%%%%%%%%%%%%%%%%%%%%%%%%%%%%%%%%%%%%%%%%%%%%%%%%%%%%%%%%%%%%%%%%%%%%%%%
   \subsection*{Notation}\label{sec:GeneralNotation}
   %%%%%%%%%%%%%%%%%%%%%%%%%%%%%%%%%%%%%%%%%%%%%%%%%%%%%%%%%%%%%%%%%%%%%%%%%%%%

    For a field $k$ we use $\kbar$ to denote a separable closure and $G_k := \Gal(\kbar/k)$ to denote the absolute Galois group of $k$. In Sections~\ref{sec:LocalQuadPts} and~\ref{sec:Extensions}, we allow \(k\) of arbitrary characteristic; in the remainder of the paper we restrict to \(k\) of characteristic different from \(2\). For $k$-schemes $Y \to \Spec(k)$ and $S \to \Spec(k)$ we define $Y_S := Y \times_{\Spec(k)}S$ and $\Ybar = Y \times_{\Spec k}\Spec(\kbar)$. When $S = \Spec(A)$ is the spectrum of a $k$-algebra $A$, we use the notation $Y_A := Y_{\Spec(A)}$. A \defi{quadratic point} on $Y$ is a morphism of $k$-schemes $\Spec(K) \to Y$, where $K$ is an \'etale $k$-algebra of degree $2$. In particular, $K = k\times k$ is allowed in which case $Z_K \simeq Z \times Z$ for any $k$-subscheme $Z \subset Y$.

    The Brauer group of a scheme $Y$ is the \'etale cohomology group $\Br(Y) := \HH^2_\textup{\'et}(Y,\G_m)$; when $Y = \Spec(R)$ is the spectrum of a ring $R$ we define $\Br(R) := \Br(\Spec R)$. If $s_Y\colon Y \to \Spec(k)$ is a $k$-scheme, then $\Br_0(Y) \subset \Br(Y)$ is the image of the pullback map $s_Y^*:\Br(k) \to \Br(Y)$. We use $\Br_1(Y)$ to denote the kernel of the map $\Br(Y) \to \Br(\Ybar)$. We recall that there is a canonical injective map $\Br_1(Y)/\Br_0(Y) \to \HH^1(k,\Pic(\Ybar))$ coming from the Hochschild-Serre spectral sequence \cite{CTS-Brauer}*{Prop. 4.3.2} and that this map is an isomorphism if $\HH^3(k,\G_m) = 0$.
    
    An element $\beta \in \Br(Y)$ may be evaluated at a $k$-point $y :\Spec(k) \to Y$ by pulling back along $y$ to obtain $\beta(y) := y^*\beta \in \Br(k)$. For a finite locally free morphism of schemes $Y \to Z$ we use $\Cor_{Y/Z}\colon \Br(Y) \to \Br(Z)$ to denote the corestriction map. When $Y = \Spec(A)$ and $Z = \Spec(B)$ are affine schemes this is also denoted by $\Cor_{A/B} \colon \Br(A) \to \Br(B)$. 

    A variety over $k$ is a separated scheme of finite type over $k$. A variety is called \defi{nice} if it is smooth, projective and geometrically integral and is called \defi{split} if it contains an open subscheme that is geometrically integral.
    
    If $Y$ is an integral $k$-variety, $\kk(Y)$ denotes its function field. More generally, if $Y$ is a finite union of integral $k$-varieties $Y_i$, then $\kk(Y) := \prod \kk(Y_i)$ is the ring of global sections of the sheaf of total quotient rings. In particular, if a finite dimensional \'etale $k$-algebra $A$ decomposes as a product $A\simeq \prod k_j$ of finite field extensions of $k$ and $Y$ is a reduced $k$-variety, then $\kk(Y_A) \simeq \prod\kk(Y_{k_j})$, and $\Cor_{\kk(Y_A)/\kk(Y)} = \sum \Cor_{\kk(Y_{k_j})/\kk(Y)}$.

    For a global field $k$, we use $\Omega_k$ to denote the set of primes of $k$. For a prime $v \in \Omega_k$ we use $k_v$ to denote the corresponding completion and for a $k$-scheme $Y$ we set $Y_v := Y_{k_v}$. We use $\A_k$ to denote the adele ring of $k$. For a subgroup $B \subset \Br(Y)$, $Y(\A_k)^{B} \subset Y(\A_k)$ denotes the set of adelic points orthogonal to $B$, i.e.,
    \[
        Y(\A_k)^{B} = \{ (y_v) \in Y(\A_k) \;:\; \forall\,\beta \in B\,,\,\sum_{v\in \Omega_k} \inv_v(\beta(y_v)) = 0\,\}\,.
    \]
    We define $Y(\A_k)^{\Br} := Y(\A_k)^{\Br(Y)}$.

%%%%%%%%%%%%%%%%%%%%%%%%%%%%%%%%%%%%%%%%%%%%%%%%%%%%%%%%%%%%%%%%%%%%%%%%%%%%%%%%
%%%%%%%%%%%%%%%%%%%%%%%%%%%%%%%%%%%%%%%%%%%%%%%%%%%%%%%%%%%%%%%%%%%%%%%%%%%%%%%%
\section*{Acknowledgements}
%%%%%%%%%%%%%%%%%%%%%%%%%%%%%%%%%%%%%%%%%%%%%%%%%%%%%%%%%%%%%%%%%%%%%%%%%%%%%%%%
%%%%%%%%%%%%%%%%%%%%%%%%%%%%%%%%%%%%%%%%%%%%%%%%%%%%%%%%%%%%%%%%%%%%%%%%%%%%%%%%

The authors were supported by the Marsden Fund Council administered by the Royal Society of New Zealand, and the second author was also supported by NSF grant \#1553459. This project was initiated while the authors attended the trimester ``Reinventing Rational Points'' at the Institut Henri Poincar\'e (IHP) and the authors would like to thank the IHP and the organizers of the trimester for their support. The second author would also like to thank the UW ADVANCE Transitional Support Program, which made it possible for her to participate in the IHP program with young children. 

The authors thank John Ottem for outlining the construction given in Remark~\ref{rmk:PI}\eqref{item:Ottem},  Jean-Louis Colliot-Th\'el\`ene for a number of helpful comments and pointing out that~\cite{CTCoray}*{Theorem C} could be used to prove Lemma~\ref{lem:dp4ConstantEval}, Asher Auel for suggesting helpful references for the proof of Lemma~\ref{lem:relBr}, Yang Cao and Olivier Wittenberg for suggesting proofs of Lemma~\ref{lem:CorCommutesWithEval} and helpful comments on the exposition, and Aaron Landesman and Dori Bejleri for a discussion related to the proof of Proposition~\ref{prop:sym2}\eqref{it:p1}. The authors would also like to thank the referees for a number of thoughtful comments which have improved the exposition.

%%%%%%%%%%%%%%%%%%%%%%%%%%%%%%%%%%%%%%%%%%%%%%%%%%%%%%%%%%%%%%%%%%%%%%%%%%%%%%%%
%%%%%%%%%%%%%%%%%%%%%%%%%%%%%%%%%%%%%%%%%%%%%%%%%%%%%%%%%%%%%%%%%%%%%%%%%%%%%%%%
\section{Intersections of quadrics in $\PP^4$ over local fields}\label{sec:LocalQuadPts}
%%%%%%%%%%%%%%%%%%%%%%%%%%%%%%%%%%%%%%%%%%%%%%%%%%%%%%%%%%%%%%%%%%%%%%%%%%%%%%%%
%%%%%%%%%%%%%%%%%%%%%%%%%%%%%%%%%%%%%%%%%%%%%%%%%%%%%%%%%%%%%%%%%%%%%%%%%%%%%%%%

    \begin{theorem}\label{thm:localquadpts}
   	  Let $X \subset \PP^4_k$ be a smooth complete intersection of two quadrics over a local field $k$. There is a quadratic extension $K/k$ such that $X(K) \ne \emptyset$.
    \end{theorem}
    
    \subsubsection*{Outline of proof of Theorem~\ref{thm:localquadpts}} In Section~\ref{sec:split}, we prove that if there exists an integral model \(\scrX\subset\PP^4\) with split special fiber, then \(X(k)\neq\emptyset\).  We use this result to reduce to the case that the special fiber is a union of four planes permuted transitively by the Galois group.  We then use the geometric classification results in Section~\ref{sec:UnionOfPlanes} together with the existence of semistable models proved by Tian~\cite{Tian} (following Koll\'ar~\cite{Kollar}) to give explicit models of the remaining cases in Section~\ref{sec:Semistable}. Next, we study these explicit models and show directly that over every ramified quadratic extension there is a change of coordinates so that the model has split special fiber. Thus, by the results of Section~\ref{sec:split}, these models have points over every ramified quadratic extension.  The details of how the ingredients come together are in Section~\ref{sec:ProofOfLocalThm}.  

    \begin{remark}
        The methods of this proof are fairly flexible, but it does rely on two key properties of finite fields: 1) There is a unique quartic extension of any finite field, it is Galois, and the Galois group is cyclic; and 2) Every split variety over a finite field has index \(1\).  If \(k\) is a complete field with respect to a discrete valuation and its residue field satisfies the above two properties, then Theorem~\ref{thm:localquadpts} holds over \(k\).
    \end{remark}
    
As mentioned in the introduction, we also give alternate proofs of Theorem~\ref{thm:localquadpts} which work in the case that $k$ has odd residue characteristic (Section~\ref{sec:Alternate}) and in the case that $k$ has characteristic $2$ (Section~\ref{sec:AlternateChar2}); this latter proof also holds for global fields of characteristic \(2\).

%%%%%%%%%%%%%%%%%%%%%%%%%%%%%%%%%%%%%%%%%%%%%%%%%%%%%%%%%%%%%%%%%%%%%%%%%%%%%%%%
\subsection{Intersections of quadrics with split special fiber}\label{sec:split}
%%%%%%%%%%%%%%%%%%%%%%%%%%%%%%%%%%%%%%%%%%%%%%%%%%%%%%%%%%%%%%%%%%%%%%%%%%%%%%%%
    \begin{prop}\label{prop:SplitSpecialFiber}
        Let \(k\) be a nonarchimedean local field, let \(\OO\) denote the valuation ring of \(k\), and let \(X/k\) smooth complete intersection of quadrics in $\PP^4_k$.   Assume there exists an integral model \(\scrX/\OO\) such that the special fiber is \emph{split} (i.e., contains a geometrically integral open subscheme).  Then \(X(k)\neq\emptyset\).
    \end{prop}
    \begin{proof}
        Since the special fiber is split, it contains a geometrically integral open subscheme \(U^{\circ}/\F\).  By the Hasse-Weil bounds, \(U^{\circ}\) contains a smooth \(\F'\)-point for all extensions with sufficiently large cardinality.  In particular, there exists an extension \(\F'/\F\) of \emph{odd} degree where \(U^{\circ}\) has a smooth \(\F'\)-point. Thus, by Hensel's Lemma, $X$ has a $k'$-point for $k'/k$ an unramified extension of odd degree. {Since \(X\) is an intersection of two quadrics,} the theorems of Amer, Brumer and Springer \cites{Amer,Brumer,Springer} then imply that $X(k)\neq \emptyset$. (In characteristic \(2\), see~\cite{EKM-QuadraticForms}*{Cor. 18.5 and Thm. 17.14} for proofs of the Amer, Brumer and Springer theorems; the Amer and Brumer theorem in characteristic \(2\) is attributed to an unpublished preprint of Leep.)
    \end{proof}

%%%%%%%%%%%%%%%%%%%%%%%%%%%%%%%%%%%%%%%%%%%%%%%%%%%%%%%%%%%%%%%%%%%%%%%%%%%%%%%%
\subsection{Ranks of a quadratic forms in arbitrary characteristic}\label{sec:rank}
%%%%%%%%%%%%%%%%%%%%%%%%%%%%%%%%%%%%%%%%%%%%%%%%%%%%%%%%%%%%%%%%%%%%%%%%%%%%%%%%

    Let $q$ be a quadratic form on a vector space $V$ over a field $F$. Then (by definition) the mapping $B_q:V \times V \to F$ given by $B_q(x,y) = q(x+y) - q(x) - q(y)$ is bilinear. We say that $q$ is \defi{regular} if the set $\{ x \in V \;:\; q(x) = 0 \text{ and } \forall\, y \in V, \, B_q(x,y) = 0 \}$ contains only the zero vector in $V$. (If the characteristic of $F$ is not $2$, then the condition $q(x) = 0$ is superfluous.) We say that $q$ is \defi{geometrically regular} if its base change to the algebraic closure of $F$ is regular. Such forms are called \defi{nondegenerate} in \cite{EKM-QuadraticForms}*{Definition 7.17}. A quadratic form $q$ on a vector space of dimension at least $2$ is geometrically regular if and only if the quadric $\calQ$ in $\PP(V)$ defined by the vanishing of $q$ is geometrically regular or, equivalently, smooth (see \cite{EKM-QuadraticForms}*{Proposition 22.1}). 
    
    The \defi{rank} of a quadratic form $q$ is the largest integer $m$ such that there is a subspace $W \subset V$ of dimension $m$ such that the restriction of $q$ to $W$ is geometrically regular, i.e., such that the intersection of $\calQ$ with the linear space corresponding to $W$ is smooth. The rank of a quadric in $\PP^n$ is defined to be the rank of any quadratic form defining it. If \(F\) has characteristic different from \(2\), then the rank of \(q\) is the same as the rank of a symmetric matrix associated to \(B_q\).  
    
    If \(\Char(F) = 2\), then the rank of \(q\) is not necessarily equal to the rank of (a matrix associated to) \(B_q\), but the definition yields the lower bound \(\rank(B_q) \leq \rank(q)\). The possible discrepancy between these two ranks is due to the fact that \(B_q({x}, {x}) = q(2x) - 2q(x) = 0\) for all \({x}\in V\). Thus a matrix associated to \(B_q\) has zeros along the diagonal and so is skew-symmetric (and symmetric).  Skew symmetric matrices always have even rank, but quadratic forms can have odd rank (e.g., \(q = x^2\) has rank \(1\)). 
    
    Over an algebraically closed field a quadratic form \(q\) has rank \(2n\) if and only there is a change of coordinates such that \(q = x_1x_2 + x_3x_4 + \dots + x_{2n-1}x_{2n}\), and \(\rank(q) = 2n + 1\) if and only if there is a change of coordinates such that \(q = x_0^2 + x_1x_2 + x_3x_4 + \dots + x_{2n-1}x_{2n}\) (see \cite{EKM-QuadraticForms}*{Props. 7.29 and 7.31 and Ex. 7.34}).\footnote{This characterization shows that, in general, the rank of the symmetric bilinear form can only differ from the rank of the quadratic form by \(1\), namely that \(\rank(B_q) \leq \rank(q) \leq \rank(B_q) + 1.\)} It follows from this characterization that a quadric in $\PP^n$ of rank $1$ with $n\ge 1$ is not geometrically reduced and a quadric in $\PP^n$ of rank $2$ with $n \ge 2$ is not geometrically irreducible.
    
    It also follows that, for a quadratic form $q$ over an algebraically closed field, the rank is the smallest integer $r$ such that there exists a linear change of variables under which $q$ becomes a quadratic form in the variables $x_1,\dots,x_r$ alone. This is the definition of rank used in \cite{HeathBrown}. We will only require the equivalence of these definitions over algebraically closed fields, but we note that they are also equivalent if the field is not of characteristic $2$ (by the well known fact that $q$ can be diagonalized) or if the field is perfect of characteristic $2$ (as follows from \cite{EKM-QuadraticForms}*{Proposition  7.31} using that in this case $c_1x_1^2 + \cdots + c_sx_s^2 = (c_1^{1/2}x_1 + \cdots + c_s^{1/2}x_s)^2$). In general, the two notions differ as seen by considering the rank $1$ form $x_1^2 + tx_2^2 = (x_1+t^{1/2}x_2)^2$ over $\F_2(t)$ for which there is no $\F_2(t)$-linear change of variables writing it as a form in $1$ variable.

    \begin{lemma}\label{lem:ranks}
    	Suppose $q$ and $\tilde{q}$ are quadratic forms of rank $r(q)$ and $r(\tilde{q})$, respectively, over a field $F$. Then $r(q \perp \tilde{q}) = r(q) + r(\tilde{q})$ except when $\Char(F) = 2$ and $r(q)$ and $r(\tilde{q})$ are both odd, in which case $r(q \perp \tilde{q}) = r(q) + r(\tilde{q}) -1$.
    \end{lemma}
    \begin{proof}
	For $\Char(F) \ne 2$ see \cite{EKM-QuadraticForms}*{{Prop. 7.29}}. For $\Char(F) = 2$ this follows from \cite{EKM-QuadraticForms}*{{Proposition} 7.31 and Remark 7.21} and the fact that an orthogonal direct sum of rank $1$ forms has rank $1$ (cf. \cite{EKM-QuadraticForms}*{Remark 7.24}).
    \end{proof}
    
%%%%%%%%%%%%%%%%%%%%%%%%%%%%%%%%%%%%%%%%%%%%%%%%%%%%%%%%%%%%%%%%%%%%%%%%%%%%%%%%
\subsection{Intersections of two quadrics with many irreducible components}\label{sec:UnionOfPlanes}
%%%%%%%%%%%%%%%%%%%%%%%%%%%%%%%%%%%%%%%%%%%%%%%%%%%%%%%%%%%%%%%%%%%%%%%%%%%%%%%%

    \begin{lemma}\label{lem:NumberOfComponents}
        Let $X\subset\PP^4$ be a reduced complete intersection of two quadrics over an algebraically closed field. If $X$ is reducible, then \(X\) contains a \(2\)-plane or an irreducible quadric surface.  In addition:
        \begin{enumerate}
            \item if \(X\) contains an irreducible quadric surface, then \(X\) is the union of two quadric surfaces (with one possibly reducible) and \(X\) is contained in a rank 2 quadric; and \label{part:IrredQuadric}
            \item if \(X\) contains two distinct \(2\)-planes \(P_1, P_2\), then \(X\) is either the union of four distinct \(2\)-planes or the union of \(P_1\) and \(P_2\) with an irreducible quadric surface.
        \end{enumerate}
    \end{lemma}
    \begin{proof}
        The components of this proof can be found in~\cite{CTSSDI}*{Section 1} and~\cite{HeathBrown}*{Proof of Lemma 3.2}.  We repeat them here for the reader's convenience.

        The degrees of the irreducible components of \(X\) sum to \(4\), so we consider the partitions
        \[
          3 + 1, \; 2 + 2, \; 2 + 1 + 1, \; 1 + 1 + 1 + 1.  
        \]
        In any of these cases, \(X\) contains a surface of degree \(1\) (i.e., a \(2\)-plane) or a surface of degree \(2\) (i.e., a quadric surface).  To complete the proof, it remains to show that if \(X\) contains an irreducible quadric surface, then \(X\) is contained in a rank \(2\) quadric, and in the case of the partition \(2 + 1 + 1\), the union of the two planes is a quadric surface, i.e., is contained in a hyperplane.

        Assume that \(X\) contains an irreducible quadric surface, given by the vanishing of a quadratic form \(q\) and a linear form \(\ell\). The rank of $q$ cannot be $1$ because $X$ is reduced and the rank of $q$ cannot be $2$ because the quadric surface is irreducible. So $q$ must have rank at least $3$. Then the quadratic forms defining \(X\) must be of the form \(cq + \ell\ell'\), for some constant \(c\) and some linear form \(\ell'\).  There will be some linear combination of these where \(c=0\), and so \(X\) is cut out by the ideal 
        \[
            \langle \ell\ell', q + \ell\ell''\rangle = \langle \ell, q\rangle\cdot\langle \ell', q + \ell\ell''\rangle,
        \] 
        for some linear forms \(\ell', \ell''\).  The first factor gives our original quadric surface, the residual factor will give a (possibly reducible) quadric surface, and \(V(\ell\ell')\) is a rank 2 quadric hypersurface containing \(X\).
    \end{proof}

    \begin{lemma}\label{lem:FourPlanes}
        Let $X\subset\PP^4$ be a complete intersection of two quadrics over an algebraically closed field. If $X$ is the union of $4$ distinct planes, then \(X\) is a cone and $X$ is contained in a quadric hypersurface of rank $2$. If, in addition, $X$ has a unique cone point and there is cyclic subgroup of $\Aut(X)$ acting transitively on the irreducible components of $X$, then, up to an automorphism of \(\PP^4\), $X = V(x_0x_1, x_2x_3)\subset \PP^4$.
    \end{lemma}
    \begin{proof}
       After a change of coordinates, we may assume that one of the planes is \(V(x_0,x_1)\).  If all pairs of the planes meet in a line, then we may  assume that one of the other planes is \(V(x_0,x_2)\).  Thus, \(X\) must be defined by \(x_0\ell = x_0\elltilde + x_1x_2 = 0\) for some linear forms \(\ell, \elltilde\). Note that $x_0\ell$ has rank $2$. If \(x_0,x_1,x_2,\ell,\elltilde\) are linearly dependent, then \(X\) is a cone.  If \(x_0,x_1,x_2,\ell,\elltilde\) are linearly independent, then, without loss of generality, we may assume that \(\ell = x_3\) and \(\elltilde = x_4\), so 
        \[
            X = V(x_0x_3, x_0x_4 + x_1x_2) = V(x_0, x_1)\cup V(x_0, x_2)\cup V(x_3, x_0x_4 + x_1x_2).
        \]
        This is not a union of four planes, so we have a contradiction.  
        
        If any pair of the planes meet in a point (in which case any cone point would be unique), then we may instead assume that one of the other planes is \(V(x_2,x_3)\).  Under these assumptions \(X\) must be the intersection of \(V(a_ix_0x_2 + b_ix_0x_3 + c_ix_1x_2 + d_ix_1x_3)\) for \(i =0,1\) and some \(a_i,b_i,c_i,d_i\). In particular, \(X\) is a cone. In addition, if \((a_0,d_0), (a_1,d_1)\) are linearly independent, then one of the defining equations can be taken to be a rank \(2\) quadric divisible by \(x_i\), and similarly if \((b_0,c_0), (b_1,c_1)\) are linearly independent.  Thus, it remains to consider the case that \(X = V(ax_0x_2 + dx_1x_3, bx_0x_3 + cx_1x_2)\), with \(abcd\neq 0\).  Then 
        \[
          bc\left(ax_0x_2 + dx_1x_3\right) + \sqrt{abcd}\left(bx_0x_3 + cx_1x_2\right)
          = \left(\sqrt{a}bx_0 + \sqrt{bcd}x_1\right)\left(\sqrt{a}cx_2 + \sqrt{bcd}x_3\right),
        \]
        and so \(X\) is contained in a rank \(2\) quadric.
        
        It remains to show that if \(X\) has a unique cone point and admits a transitive cyclic action on its irreducible components, then, up to an automorphism of \(\PP^4\), $X = V(x_0x_1, x_2x_3)\subset \PP^4$. Without loss of generality, we may assume the cone point is \([0:0:0:0:1]\), and so \(X\) is a cone over an intersection of quadrics in $\PP^3$, which is a curve $Z$ of arithmetic genus $1$.  Since by assumption $X$ is a union of $4$ planes, $Z$ must be the union of $4$ lines.  Furthermore, since \(X\) has a unique cone point, the four lines of \(Z\) cannot all meet.  This combined with the transitive \(\Z/4\Z\)-action then implies that any triple of the lines cannot meet.  By enumerating the possible intersection configurations, one can check that the only arrangement of lines with a transitive \(\Z/4\Z\)-action, with no triple meeting, and whose union is a curve of genus \(1\) is a $4$-gon, i.e., a cycle of rational curves, where each curve meets exactly two of the others. After a change of coordinates, we may assume that the intersections are
        \[
          P_1\cap P_2 = V(x_0,x_1,x_2),  P_2\cap P_3 = V(x_0,x_1,x_3),  P_3\cap P_4 = V(x_0,x_2,x_3),  P_4\cap P_1 = V(x_1,x_2,x_3),
        \]
        so $X = V(x_0x_2, x_1x_3)$.
    \end{proof}

    \begin{cor}\label{cor:Nonsplit}
        Let \(X\subset \PP^4_k\) be a geometrically reduced complete intersection of two quadrics over a field $k$.  If \(X\) is nonsplit, then \(X\) is contained in a rank \(2\) quadric.
    \end{cor}
    \begin{proof}
        Assume \(X\) is nonsplit.  Since \(X\) is geometrically reduced and nonsplit, it must be geometrically reducible, and so reducible over a separable closure.  Thus the absolute Galois group of $k$ acts on the geometric components.  Since \(X\) is nonsplit, none of the components are fixed by Galois, and so, by Lemma~\ref{lem:NumberOfComponents}, \(X\) is geometrically either the union of two irreducible quadric surfaces or the union of four planes.  In the first case, Lemma~\ref{lem:NumberOfComponents}\eqref{part:IrredQuadric} gives the result, and in the second Lemma~\ref{lem:FourPlanes} does.
    \end{proof}

%%%%%%%%%%%%%%%%%%%%%%%%%%%%%%%%%%%%%%%%%%%%%%%%%%%%%%%%%%%%%%%%%%%%%%%%%%%%%%%%
\subsection{Semistable models}\label{sec:Semistable}
%%%%%%%%%%%%%%%%%%%%%%%%%%%%%%%%%%%%%%%%%%%%%%%%%%%%%%%%%%%%%%%%%%%%%%%%%%%%%%%%

Following work of Koll\'ar~\cite{Kollar} in the case of hypersurfaces, Tian~\cite{Tian}*{Section 2.1} has defined a notion of semistability for intersections of two quadrics over discrete valuation rings.  This notion of semistability allows one to find a model of \(X\) whose special fiber is fairly well controlled.  

Before stating our results, we first review some of the definitions from Tian's semistability machinery. Suppose $k$ is a nonarchimedean local field with ring of integers \(\OO\) and residue field \(\F\). We will use \(\pi\) to denote a uniformizer. Let \(\scrX\subset \PP^4_{\OO}\) be an intersection of two quadrics.  Given \(\scrX\), we can associate a \(2\times 15\) matrix \(A\) such that each row is the coefficient vector of the corresponding defining equation for \(\scrX\).  Note that changing the defining equations corresponds to multiplying \(A\) on the left by an element of \(\GL_2(\OO)\).  Thus, up to this \(\GL_2\)-action, we have a well-defined matrix \(A_{\scrX}\).
    
Given an nonnegative integer weight vector \(\mathbf{w}\in \mathbb{N}^5\), we define the change of coordinates \(f_{\mathbf{w}}\colon \PP^4_{\OO}\to \PP^4_{\OO}, \;x_i\mapsto \pi^{w_i}x_i\).
Then we define the \defi{multiplicity of \(\scrX\) with respect to \(\mathbf{w}\)} to be 
\[
    \mult_{\mathbf{w}}(\scrX)  := \min \{ v(m) \;:\; m \text{ is a } 2\times2\textup{ minor of }A_{f_{\mathbf{w}}^*\scrX}\}\,,
\]
where \(v\) denotes the valuation on \(\OO\).  Then \(\scrX\) is said to be \defi{semistable} if for all weight vectors \(\mathbf{w}\) and all automorphisms \(g\in \Aut(\PP^4_{\OO}) = \PGL_5(\OO)\), we have
\[
    \mult_{\mathbf{w}}(g(\scrX)) \leq \frac45\left(\sum_{i=0}^4w_i\right).
\]
By~\cite{Tian}*{Theorem 2.7}, any smooth intersection of two quadrics \(X\subset \PP^4_k\) has a semistable integral model.  For more details, see~\cite{Tian}*{Section 2.1 and 2.4}.

We will also make use of the following results from~\cite{Tian}.
\begin{lemma}%[\cite{Tian}*{Lemma 2.9 and 2.22(1), (2), and (4)}]
    \label{lem:ReducedSpecialFiber}
    Let $k$ be a nonarchimedean local field, let \(\OO\) denote the valuation ring of \(k\), let $\F$ denote the residue field of $k$, and let \(X \subset \PP^4_k\) be a smooth complete intersection of two quadrics.   Let \(\scrX\subset \PP^4_{\OO}\) be a semistable model of \(X\) (which exists by~\cite{Tian}*{Theorem 2.7}). Then:
    \begin{enumerate}
        \item{{\cite{Tian}*{Lemma 2.9}}} The special fiber of \(\scrX\)  is a complete intersection of two quadrics.
        \item \cite{Tian}*{Lemma 2.22(1)} The special fiber is not contained in a reducible quadric hypersurface defined over \(\F\).\label{part:NoReducibleOverF}
        \item \cite{Tian}*{Lemma 2.22(2)} The special fiber does not contain a plane defined over \(\F\).\label{part:NoHyperplane}
        \item \cite{Tian}*{Lemma 2.22(4)} The special fiber is reduced.\label{part:reduced}
    \end{enumerate}
\end{lemma}
\begin{remark}
    In~\cite{Tian}*{Sections 2.2 -- 2.4}, Tian works over local function fields, but as noted in~\cite{Tian}*{beginning of Section 2.2}, the proofs go through essentially verbatim for any nonarchimedean local field.  In~\cite{Tian}*{Section 2.4} (in which~\cite{Tian}*{Lemma 2.22} is stated and proved), Tian adds the hypothesis that the residue field has odd characteristic, and so freely interchanges smooth and nonsingular. However, no assumption on the residue characteristic is needed for the proofs of \cite{Tian}*{Lemma 2.22 (1), (2), and (4)}. For the sake of completeness, we repeat Tian's proof of Lemma~\ref{lem:ReducedSpecialFiber}\eqref{part:NoReducibleOverF}--\eqref{part:reduced}.
\end{remark}

\begin{proof}
    If the special fiber is contained in a reducible quadric hypersurface defined over \(\F\), then, after possibly changing variables, one of the quadrics defining \(\scrX\) must be of the form
    \(x_0x_1 + \pi \tilde{q}\), in which case  \(\textup{mult}_{(1,0,0,0,0)}(\scrX) \geq 1\).  However, since \(\scrX\) is assumed to be semistable we must have \(\textup{mult}_{(1,0,0,0,0)}(\scrX) \leq \frac{4\cdot(1)}{5}\), resulting in a contradiction.  This proves~\eqref{part:NoReducibleOverF}.  Similarly, if the special fiber contains a linear subspace of dimension \(2\) defined over \(\F\), {which we may assume is $V(x_0,x_1)$}, then \(\textup{mult}_{(1,1,0,0,0)}(\scrX) \geq 2\). However, the semistability hypothesis implies that \(\textup{mult}_{(1,1,0,0,0)}(\scrX) \leq \frac{4\cdot(1 + 1)}{5} = \frac85\), giving a contradication. Thus, we conclude~\eqref{part:NoHyperplane}.

    Now we prove~\eqref{part:reduced}.  By~\cite{Tian}*{Lemma 2.9}, the special fiber is a complete intersection, so the special fiber is reduced if and only if all geometric irreducible components are reduced. Assume that the special fiber has a nonreduced geometric irreducible component.  Since the special fiber has degree \(4\) and contains no plane defined over \(\F\), the only possibilities are:
    \begin{enumerate}[label=(\alph*)]
        \item\label{case:doublequadric} a quadric surface of multiplicity \(2\), or 
        \item\label{case:twoconj} a union of two conjugate planes, each with multiplicity \(2\).
    \end{enumerate}
    Note that in case~\ref{case:twoconj}, the two planes must meet in a line, as otherwise a general hyperplane section would be the union of two skew double lines, which is not possible.  Thus, case~\ref{case:twoconj} is subsumed by case~\ref{case:doublequadric}, and so the reduced special fiber is given by the vanishing of a linear form \(\ell\) and a quadratic form \(q\).  Hence, the special fiber is defined by quadratic forms of the form \(\ell\ell_1, \ell\ell_2 + q\) for some linear forms \(\ell_1, \ell_2\), which contradicts~\eqref{part:NoReducibleOverF}.
\end{proof}

    \begin{prop}\label{prop:ReductionFourPlanes}
        Let $k$ be a nonarchimedean local field, let \(\OO\) denote the valuation ring of \(k\), let $\F$ denote the residue field of $k$, and
        let \(X\subset \PP^4_k\) be smooth complete intersection of two quadrics. Let \(\scrX\subset \PP^4_{\OO}\) be a semistable model of \(X\) (which exists by~\cite{Tian}*{Theorem 2.7}).  Assume that the special fiber of \(\scrX/\OO\) is  geometrically the union of four \(2\)-planes and that the Galois group acts transitively on the four \(2\)-planes.
        Then, for any choice of uniformizer \(\pi\), \(X\) must be given by the vanishing of quadratic forms of the shape:
        \begin{equation}
            q(x_0,\dots,x_3) + \pi^mx_4\ell(x_0,\dots,x_3), \textup{ and }
            \tilde{q}(x_0,\dots,x_3) + \pi x_4^2 + \pi^nx_4\elltilde(x_0,\dots,x_3),  \label{eq:ConeOverI4}
        \end{equation}
        (with \(m, n\) positive integers, and \(q,\tilde{q}\) quadratic forms such that every \(\Fbar\)-linear combination of $q$ and $\tilde{q}$ modulo $\pi$ has rank at least \(2\)); or
        \begin{align}
            \begin{split}
            &g(x_0,x_1,x_2) + \pi h(x_3,x_4) + \pi^ax_3\ell_3(x_0,x_1,x_2) + \pi^bx_4\ell_4(x_0,x_1,x_2), \textup{ and }\\ 
           &\tilde{g}(x_0,x_1,x_2) + \pi \tilde{h}(x_3,x_4) + \pi^cx_3\elltilde_3(x_0,x_1,x_2) + \pi^dx_4\elltilde_4(x_0,x_1,x_2),
            \end{split}\label{eq:ConeOver4Pts}
        \end{align}
        (with \(a,b,c,d\) positive integers, $\ell_i,\tilde{\ell}_i$ linear forms and \(g,\tilde{g},h, \tilde{h} \) quadratic forms such that every \(\Fbar\)-linear combination of $g$ and $\tilde{g}$ modulo $\pi$ has rank at least \(2\) and every $\Fbar$-linear combination of $h$ and $\tilde{h}$ modulo \(\pi\) has rank at least $1$).
    \end{prop}
    % Before giving the proof of this proposition, 

    \begin{proof}%[Proof of Proposition~\ref{prop:ReductionFourPlanes}]
        By Lemma~\ref{lem:FourPlanes}, the special fiber must be isomorphic (over \(\Fbar\)) to \(V(x_0x_1,x_2x_3)\) or a cone over a complete intersection of two quadrics in \(\PP^2\) (i.e., a complete intersection of two conics).

        Let us first assume that the special fiber is geometrically isomorphic to $V(x_0x_1, x_2x_3)$.  Note that this variety has a unique singular point, the cone point, so it must be defined over \(\F\).  After a change of coordinates, we may assume the cone point reduces to \(V(x_0,x_1,x_2,x_3)\) and hence $X$ is given by
        \[
            q(x_0,\dots,x_3) + \pi^mx_4\ell(x_0,\dots,x_4), \textup{ and }
            \tilde{q}(x_0,\dots,x_3) + \pi^nx_4\elltilde(x_0,\dots,x_4),
        \]
        for some integers \(m, n\geq 1\), quadratic forms \(q, \tilde{q}\) and linear forms \(\ell, \elltilde\) that are nonzero modulo \(\pi\).  We will first use the semistability of \(\scrX\) for the weight vector \(\mathbf{w}:= (1,1,1,1,0)\) to show that one of \(\pi^m\ell\) or \(\pi^n\elltilde\) must evaluate to a uniformizer at \([0:0:0:0:1]\).  Note that \(A_{f_{\mathbf{w}}^*\scrX}\) has the following form
        \[
            \begin{pmatrix}
                \pi^2\textup{coefs}(q) 
                    & \pi^{m+1}\ell_0
                    & \pi^{m+1}\ell_1
                    & \pi^{m+1}\ell_2
                    & \pi^{m+1}\ell_3
                    & \pi^{m}  \ell_4\\
                \pi^2\textup{coefs}(\tilde{q}) 
                    & \pi^{n+1}\elltilde_0 
                    & \pi^{n+1}\elltilde_1
                    & \pi^{n+1}\elltilde_2
                    & \pi^{n+1}\elltilde_3
                    & \pi^{n}  \elltilde_4
            \end{pmatrix},
        \]
        where \(\ell = \sum_i\ell_ix_i, \elltilde = \sum_i\elltilde_ix_i\) and \(\textup{coefs}(q) , \textup{coefs}(\tilde{q})\) denote the coefficient vectors of \(q, \tilde{q}\) respectively. Hence, using the strong triangle equality and the definition of multiplicity, one can compute that \(\textup{mult}_{\mathbf{w}}(\scrX) \geq \min(4, 2 + m + v(\ell_4),2 + n + v(\elltilde_4))\).  However, the semistability assumption implies that \(\textup{mult}_{\mathbf{w}}(\scrX) \leq \frac{4\cdot(1 + 1 + 1 + 1)}{5} = \frac{16}{5}\), and so  \(\min(m + v(\ell_4),n + v(\elltilde_4)) = 1\). Thus, after renaming \(q, \tilde{q}\) and \(\ell, \tilde{\ell}\) and possibly scaling the equations, we may assume the equations are of the form:
        \[
            q(x_0,\dots,x_3) + \pi^mx_4\ell(x_0,\dots,x_3), \textup{ and }
            \tilde{q}(x_0,\dots,x_3) + \pi^nx_4\elltilde(x_0,\dots,x_3) + \pi x_4^2.
        \]
    	To see that every \(\Fbar\)-linear combination of \(q\) and \(\tilde{q}\) modulo $\pi$ is rank at least $2$, recall that the variety defined by \(q\) and \(\tilde{q}\) modulo \(\pi\) is geometrically isomorphic to \(V(x_0x_1,x_2x_3)\) and note that \(a x_0x_1 + b x_2x_3\) has rank \(4\) for all \(a,b \ne 0\).

        Now assume that the special fiber is a cone over a complete intersection of two quadrics in \(\PP^2\). Then, up to a change of variables, $\scrX$ must be given by quadratic forms of the shape
        \begin{align*}
          &g(x_0,x_1,x_2) + \pi^m h(x_3,x_4) + \pi^ax_3\ell_3(x_0,x_1,x_2) + \pi^bx_4\ell_4(x_0,x_1,x_2), \textup{ and }\\ 
          &\tilde{g}(x_0,x_1,x_2) + \pi^{\tilde{m}} \tilde{h}(x_3,x_4) + \pi^cx_3\elltilde_3(x_0,x_1,x_2) + \pi^dx_4\elltilde_4(x_0,x_1,x_2),        
        \end{align*}
        where $a,b,c,d, m, \tilde{m}$ are positive integers, \(g,\tilde{g}, h,\tilde{h}\) are quadratic forms, and $\ell_i,\tilde{\ell}_i$ are linear forms. Since, by assumption, the special fiber is reduced, the complete intersection in \(\PP^2_\Fbar\) defined by the vanishing of $g$ and $\tilde{g}$ modulo $\pi$ must also be reduced. This complete intersection is therefore, geometrically, a set of $4$ non-colinear points in $\PP^2_\Fbar$. These points are not contained in any quadric of rank $1$ so every \(\Fbar\)-linear combination of $g$ and $\tilde{g}$ modulo $\pi$ has rank at least \(2\).
         
        To complete the proof, we need to show that \(m = \tilde{m} = 1\) and that \(h, \tilde{h}\) are linearly independent modulo \(\pi\).  We will again use our semistability hypothesis. Consider the weight vector \(\mathbf{w} = (1,1,1,0,0)\). One can compute that 
        \(\textup{mult}_{\mathbf{w}}(\scrX)\) is at least \(\min \{ 4, m + \tilde{m}, 2 + m, 2 + \tilde{m}\}\) and, in addition, if \(h\) and \(\tilde{h}\) are linearly dependent modulo \(\pi\), then \(\textup{mult}_{\mathbf{w}}(\scrX)\geq\min \{ 4, m + \tilde{m} + 1, 2 + m, 2 + \tilde{m}\}\).  However, the semistability assumption implies that \(\textup{mult}_{\mathbf{w}}(\scrX) \leq \frac{4\cdot(1 + 1 + 1)}{5} = \frac{12}{5}\).  Thus, we must have that \(h\) and \(\tilde{h}\) are linearly independent modulo \(\pi\), and \(m + \tilde{m} =  2\), which implies that \(m = \tilde{m} = 1\).  
    \end{proof}

%%%%%%%%%%%%%%%%%%%%%%%%%%%%%%%%%%%%%%%%%%%%%%%%%%%%%%%%%%%%%%%%%%%%%%%%%%%%%%%%
\subsection{Proof of Theorem~\ref{thm:localquadpts}}\label{sec:ProofOfLocalThm}
%%%%%%%%%%%%%%%%%%%%%%%%%%%%%%%%%%%%%%%%%%%%%%%%%%%%%%%%%%%%%%%%%%%%%%%%%%%%%%%%
    If $k$ is archimedean, then $[\kbar:k]\leq 2$ so the result is immediate.  Henceforth we assume that $k$ is nonarchimedean, and we write $\OO$ for the valuation ring of $k$ and $\F$ for the residue field of $k$. 
    By~\cite{Tian}*{Theorem 2.7}, there is a linear change of coordinates on $\PP^4_k$ such that the resulting integral model  $\scrX\subset \PP^4_{\OO}$ of $X$ is semistable. In particular, by Lemma~\ref{lem:ReducedSpecialFiber}, the special fiber of $\scrX$ is a reduced complete intersection of quadrics.  

    If the special fiber of \(\scrX\) is split, then the desired result follows from Proposition~\ref{prop:SplitSpecialFiber}.  If the special fiber of \(\scrX\) is not split, but becomes split over the quadratic extension of $\F$, then we may apply Proposition~\ref{prop:SplitSpecialFiber} over \(k'\), the unique quadratic unramified extension of \(k\), and conclude that \(X(k') \neq \emptyset\).

    Thus, we have reduced to the case that the special fiber \(\scrX^{\circ}\) of \(\scrX\) is nonsplit and remains nonsplit over the unique quadratic extension \(\F'/\F\). Since \(\F\) is perfect and \(\scrX^{\circ}\) is reduced, \(\scrX^{\circ}\) must be geometrically reduced.  Therefore \(\scrX^{\circ}\) must be geometrically reducible and \(\Gal(\Fbar/\F')\) must act nontrivially on the components.  By Lemma~\ref{lem:NumberOfComponents}, this is possible only if \(\scrX^{\circ}_{\Fbar}\) is the union of four \(2\)-planes.  Furthermore, the current assumptions imply that \(\Gal(\Fbar/\F)\) must act transitively on the four \(2\)-planes.  Thus, by Proposition~\ref{prop:ReductionFourPlanes}, we may assume that \(X\) is given by quadrics as in~\eqref{eq:ConeOverI4} or~\eqref{eq:ConeOver4Pts}.

    Consider a ramified quadratic extension $k'/k$ and let $\varpi$ be a uniformizer of $k'$.  First assume that \(\scrX\) is given by equations of the form~\eqref{eq:ConeOverI4}. Over $k'$ we may absorb a $\varpi$ into $x_4$ and obtain the model $\scrX'/\OO'$ (where $\OO'$ is the valuation ring of $k'$):
    \begin{align}
    	\begin{split}
       &q(x_0,\dots,x_3) + u^r \varpi^{2r-1}x_4\ell(x_0,\dots,x_3), \textup{ and }\\
       &\tilde{q}(x_0,\dots,x_3) + u^n\varpi^{2n-1}x_4\elltilde(x_0,\dots,x_3) + ux_4^2,    
       	\end{split}\label{pencil1}  
    \end{align}
    where $u$ is the unit such that $u\varpi^2 = \pi$. %From~\eqref{eq:ConeOverI4}, $q$ and $\tilde{q}$ have rank $4$ modulo $\varpi$, so Lemma~\ref{lem:ranks} implies that $\tilde{q} + ux_4^2$ has rank $5$ modulo $\varpi$. 
    Every $\Fbar$-linear combination of the forms in~\eqref{pencil1} modulo $\varpi$ is an orthogonal sum of an $\Fbar$-linear combination of $q$ and $\tilde{q}$ modulo $\varpi$ (which has rank at least $2$ by ~\eqref{eq:ConeOverI4}) with a quadratic form of  rank $1$. It follows from Lemma~\ref{lem:ranks} that every $\Fbar$-linear combination of the forms in~\eqref{pencil1} modulo $\varpi$ has rank at least $3$. Thus, by Corollary~\ref{cor:Nonsplit}, the special fiber of $\scrX'$ is split, so, by Proposition~\ref{prop:SplitSpecialFiber}, $\scrX'$ has a $k'$-point.

    Now assume that \(X\) is given by equations of the form~\eqref{eq:ConeOver4Pts}.  Then, we may absorb a $\varpi$ into $x_3$ and $x_4$ and obtain the model $\scrX'/\OO$ given by
        \begin{align}
        	\begin{split}
            &g(x_0,x_1,x_2) + uh(x_3,x_4) + u^a\varpi^{2a-1}x_3\ell_3(x_0,x_1,x_2) + u^b\varpi^{2b-1}x_4\ell_4(x_0,x_1,x_2), \textup{ and }  \\
            &\tilde{g}(x_0,x_1,x_2) + {u}\tilde{h}(x_3,x_4) + u^c\varpi^{2c-1}x_3\elltilde_3(x_0,x_1,x_2) + u^d\varpi^{2d-1}x_4\elltilde_4(x_0,x_1,x_2),    
            \end{split}\label{pencil2}   
          \end{align}
    where $u$ is the unit such that $u\varpi^2 = \pi$. Then every $\Fbar$-linear combination of the forms in~\eqref{pencil2} modulo $\varpi$ is an orthogonal direct sum of an $\Fbar$-linear combination of $g$ and $\tilde{g}$ modulo $\varpi$ (which is a form of rank $2$ or $3$) with an $\Fbar$-linear combination of $h$ and $\tilde{h}$ modulo $\varpi$ (which is a form of rank $1$ or $2$). Thus, by Lemma~\ref{lem:ranks}, every $\Fbar$-linear combination of the forms in~\eqref{pencil2} modulo $\varpi$ has rank at least $3$. Hence, by Corollary~\ref{cor:Nonsplit}, the special fiber of $\scrX'$ is split, and so $X$ has a $k'$-point, by Proposition~\ref{prop:SplitSpecialFiber}. \qed
        
    \subsection{Alternate proof in characteristic $2$}\label{sec:AlternateChar2}
     
     The following is a slight generalization of \cite{DD}*{Theorem 4.4}.
     
    \begin{prop}\label{prop:char2}
    	Suppose $k$ is a field of characteristic $2$ and $X \subset \PP^4_k$ is smooth complete intersection of two quadrics. Then $X(k^{1/2}) \ne \emptyset$. In particular, if $k$ is a local or global field of characteristic $2$, then $X$ contains a point defined over the quadratic extension $k^{1/2}$ of $k$.
    \end{prop}
    
    \begin{proof}
    	By \cite[Theorem 1.1]{DD}, $X$ can be defined by the vanishing of quadratic forms of the form
    	\begin{align*}
    		&a_0x_0^2 + a_1x_1^2 + a_2x_2^2 +x_3\ell_1 + x_4\ell_2\,, \text{and}\\
    		&b_0x_0^2 + b_1x_1^2 + b_2x_2^2 + x_3\ell_3 + x_4\ell_4
    	\end{align*}
    	where $a_i,b_i \in k$ and $\ell_i \in k[x_0,\dots,x_4]$ are linear forms. In particular, the intersection of $X$ with the plane $V(x_3,x_4)$ is an intersection of two conics in $\PP^2$ neither of which is geometrically reduced. The reduced subschemes of the base changes of these conics to the algebraic closure are the lines $V(a_0^{1/2}x_0 + a_1^{1/2}x_1 + a_2^{1/2}x_2)$ and $V(b_0^{1/2}x_0 + b_1^{1/2}x_1 + b_2^{1/2}x_2)$, which are defined over $k^{1/2}$. Their intersection yields a $k^{1/2}$-point on $X$.
    	
	It remains to show that $[k^{1/2}:k] = 2$ when $k$ is a local or global field of characteristic $2$. If $k$ is local, then $k = \F((t))$ with $\F$ a finite field of characteristic $2$ and $k^{1/2} = \F((t^{1/2}))$ which is clearly an extension of degree $2$. Similarly, if $k = \F(t)$ is a global function field of genus $0$ and characteristic $2$, then $k^{1/2} = \F(t^{1/2})$ is clearly a degree $2$ extension. For a general global field $k$ of characteristic $2$, which is necessarily a finite extension of $k_0 = \F(t)$ with $\F$ finite characteristic $2$, we may reduce to the genus $0$ case as follows (cf. \cite[Theorem 3]{BeckerMacLane}). Frobenius gives an isomorphism $F: k^{1/2} \to k$ which restricts to an isomorphism $k_0^{1/2} \to k_0$, and so $[k^{1/2}:k_0^{1/2}] = [k:k_0]$. Since $k$ and $k_0^{1/2}$ are both intermediate fields of the extension $k_0 \subset k^{1/2}$ we have
    	\[
    		[k^{1/2}:k][k:k_0] = [k^{1/2}:k_0^{1/2}][k_0^{1/2}:k_0]\,.
    	\]
    	Taken together these observations show that $[k^{1/2}:k] = [k_0^{1/2}:k_0]$.
    \end{proof}

%%%%%%%%%%%%%%%%%%%%%%%%%%%%%%%%%%%%%%%%%%%%%%%%%%%%%%%%%%%%%%%%%%%%%%%%%%%%%%%%
%%%%%%%%%%%%%%%%%%%%%%%%%%%%%%%%%%%%%%%%%%%%%%%%%%%%%%%%%%%%%%%%%%%%%%%%%%%%%%%%
\section{Brauer-Manin obstructions over extensions}\label{sec:Extensions}
%%%%%%%%%%%%%%%%%%%%%%%%%%%%%%%%%%%%%%%%%%%%%%%%%%%%%%%%%%%%%%%%%%%%%%%%%%%%%%%%
%%%%%%%%%%%%%%%%%%%%%%%%%%%%%%%%%%%%%%%%%%%%%%%%%%%%%%%%%%%%%%%%%%%%%%%%%%%%%%%%

    In this section, we prove some general results relating the Brauer-Manin obstruction on a nice variety $Y$ to the Brauer-Manin obstruction over an extension.  Moreover, for quadratic extensions, we relate the Brauer-Manin obstruction on (a desingularization of) the symmetric square to the Brauer-Manin obstruction over quadratic extensions.

    \begin{lemma}\label{lem:CorCommutesWithEval}
        Let $Y/k$ be a nice variety over a global field $k$, let $K/k$ be a finite extension, and let $B$ be a subset of $\Br(Y_K)$.  Then $Y(\A_k)^{\Cor_{K/k}(B)} \subset Y_{{K}}(\A_K)^{B}$.  In particular,
        \begin{enumerate}
            \item if $Y(\A_k)^{\Br} \neq \emptyset$, then $Y_K(\A_K)^{\Br}\neq \emptyset$, and\label{part:NewBrCannotObsOldPts}
            \item for any $d\mid [K:k]$, $Y(\A_k) \subset Y_K(\A_K)^{\Res_{K/k} \Br (Y)[d]}$.\label{part:ConstEvalLeadsToNoObstructions}
        \end{enumerate}
    \end{lemma}
    \begin{proof}
        By~\cite{CTS-Brauer}*{Prop. 3.8.1}, for any $\alpha\in \Br(Y_K)$ and for any local point $P_v\in Y(k_v)$, we have $(\Cor_{Y_K/Y}(\alpha))(P_v) = \Cor_{K_v/k_v}(\alpha(P_v))$, where $K_v = K \otimes_k k_v$.  Thus, for $(P_v)\in Y(\A_k)$, 
        \[
          \sum_{v\in \Omega_k}  \inv_v\left(\Cor_{Y_K/Y}(\alpha)(P_v)\right) = \sum_{v\in \Omega_k} \inv_v\left(\Cor_{K_v/k_v}(\alpha(P_v))\right)
          = \sum_{v\in \Omega_k} \sum_{w\in \Omega_K, w|v} \inv_w(\alpha(P_v))
        \]
        (where the last equality follows from the equality of maps $\inv_w = \inv_v\circ\Cor_{K_w/k_v}$ for any prime $w|v$), and so $Y(\A_k)^{\Cor_{K/k}(\alpha)} \subset Y(\A_K)^{\alpha}$.  The general statement follows by considering the intersection of $Y(\A_K)^{\alpha}$ for all $\alpha\in B$.

        It remains to prove statements~\eqref{part:NewBrCannotObsOldPts} and~\eqref{part:ConstEvalLeadsToNoObstructions}.  The first follows from taking $B = \Br(Y_K)$ and observing that $Y(\A_k)^{\Br(Y)}\subset Y(\A_k)^{\Cor_{K/k}(\Br(Y_K))}$, and the second follows from taking $B = \Res_{K/k} \Br (Y)[d]$ and using that $\Cor_{K/k}\circ\Res_{K/k} = [K:k]$.
    \end{proof}

	\begin{rmk}
		Yang Cao has given an alternative proof of Lemma~\ref{lem:CorCommutesWithEval}\eqref{part:NewBrCannotObsOldPts} which also yields a similar statement for the \'etale-Brauer obstruction. This will appear in forthcoming work of Yang Cao and Yongqi Liang \cite{CaoLiang}.
	\end{rmk}
	
    {The following lemma and corollary extend techniques of Kanevsky in the case of cubic surfaces~\cite{Kanevsky}.}

	\begin{lemma}\label{lem:Brconds}
		Let $Y$ be a nice variety over a field \(k\) such that \(\HH^3(k, \G_m) = 0\). Assume that: 
		\begin{enumerate}
			\item\label{ita} $\Pic(\Ybar)$ is finitely generated and torsion free,
			\item\label{itb} $\Br(\Ybar)$ is finite, and 
			\item \label{itc} $\Br(Y) \to \Br(\Ybar)^{G_k}$ is surjective.
		\end{enumerate}
		Then there is a finite Galois extension $k_1/k$ such that for all extensions $K/k$ linearly disjoint from $k_1$ the map $\Res_{K/k}\colon \Br(Y)/\Br_0(Y) \to \Br(Y_K)/\Br_0(Y_K)$ is surjective.
	\end{lemma}
	\begin{proof}
		The assumption \(\HH^3(k, \G_m) = 0\) implies that the injective map $\Br_1(Y)/\Br_0(Y) \to \HH^1(k,\Pic(\Ybar))$ coming from the Hochschild-Serre spectral sequence \cite{CTS-Brauer}*{Prop. 4.3.2} is an isomorphism. Assumption~\eqref{ita} implies that $\HH^1(k,\Pic(\Ybar)) \simeq \HH^1(k_0/k,\Pic(\Ybar))$ for some finite Galois extension $k_0/k$. By assumption~\eqref{itb}, there is a finite Galois extension $k_1/k_0$ such that $\Res_{\kbar/k_1}\colon \Br(Y_{k_1}) \to \Br(\Ybar)$ is surjective. Now suppose $K/k$ is linearly disjoint from $k_1$. In particular, $K$ is linearly disjoint from $k_0$, so $\Res_{K/k}\colon\Br_1(Y)/\Br_0(Y) \simeq \HH^1(k,\Pic(\Ybar)) \to \HH^1(K,\Pic(\Ybar)) \simeq \Br_1(Y_K)/\Br_0(Y_K)$ is an isomorphism. So it will suffice to show that $\Br(Y)$ and $\Br(Y_K)$ have the same image in $\Br(\Ybar)$. Since $\Br(Y_{k_1}) \to \Br(\Ybar)$ is surjective, the image of $\Br(Y_K) \to \Br(\Ybar)$ is contained in $\Br(\Ybar)^{G_K} \cap \Br(\Ybar)^{G_{k_1}}$, which is equal to  $\Br(\Ybar)^{G_k}$, since $k_1$ and $K$ are linearly disjoint.  Thus, by assumption~\eqref{itc}, $\Br(Y)$ and $\Br(Y_K)$ have the same image in $\Br(\Ybar)$.
	\end{proof}

    \begin{cor}\label{cor:k1}
    	If $Y$ is a nice variety over a global field $k$ such that $Y(\A_k)\ne\emptyset$ and $\Br(Y)/\Br_0(Y)$ is generated by the image of $\Br(Y)[d]$, then for any extension $K/k$ of degree $d$, $Y_K(\A_K)^{\Res_{K/k}(\Br(Y))} \ne \emptyset$. Moreover, if $Y$ satisfies the conditions of Lemma~\ref{lem:Brconds}, then there is a finite extension $k_1/k$ such that for any degree $d$ extension $K/k$ which is linearly disjoint from $k_1$ we have $Y_K(\A_K)^{\Br} \ne \emptyset$. \hfill\qed
    \end{cor}
    
    \begin{proof}
    	For a global field $k$ we have \(\HH^3(k, \G_m) = 0\). So the corollary follows immediately from Lemmas~\ref{lem:CorCommutesWithEval}\eqref{part:ConstEvalLeadsToNoObstructions} and~\ref{lem:Brconds}.
    \end{proof}

	\begin{rmk}
		If $Y \subset \PP^4_k$ is smooth complete intersection of two quadrics over a global field $k$ of characteristic not equal to $2$ and $Y$ is everywhere locally solvable, then the corollary applies with $d = 2$. This gives a proof of the $n = 4$ case of Theorem~\ref{thm:MainThm}\eqref{case-global-n4} under the additional hypothesis of local solubility. Note that local solubility is used here in two distinct ways. First it ensures that $\Br(Y)/\Br_0(Y)$ is generated by the image of $\Br(Y)[2]$ (which is not the case in general even though $\Br(Y)/\Br_0(Y)$ is $2$-torsion)~\cite{VAV}*{Thm. 3.4}. Second, it implies that the canonical maps $\Br(k) \to \Br_0(Y)$ are isomorphisms, locally and globally. This is used implicitly in the proof of Lemma~\ref{lem:CorCommutesWithEval}. In general, $\Br(k) \to \Br_0(Y)$ need not be injective (see Lemma~\ref{lem:relBr} for a description of the kernel when $Y$ is a del Pezzo surface of degree $4$) and so $\Res_{K/k}$ does not necessarily annihilate $[K:k]$-torsion elements of $\Br_0(Y)$. Consequently, the exact sequence
        \[
            0 \to \Br(k) \to \bigoplus \Br(k_v) \to \Q/\Z \to 0
        \]	
        of global class field theory has no analogue for $\Br_0(Y)$.
        \end{rmk}

	The following proposition relates the Brauer-Manin obstruction over quadratic extensions to the Brauer-Manin obstruction on the symmetric square. Note that while the symmetric square $\Sym^2(Y)$ is singular if $Y$ has dimension at least $2$, there exists a smooth projective model $Y^{(2)}$ over any field of characteristic different from \(2\) (see the proof of part~\eqref{it:p1} of the proposition for details).

    \begin{prop}\label{prop:sym2} Let \(k\) be a field of characteristic different from \(2\), let $Y/k$ be a nice variety of dimension at least $2$ with torsion free geometric Picard group, and let $Y^{(2)}$ be a smooth projective model of the symmetric square of $Y$ over $k$.  
    \begin{enumerate}
    \item\label{it:p1} {The rational map \(Y^2 \to \Sym^2(Y)\dashrightarrow Y^{(2)}\) induces a corestriction map \[\Cor_{Y^2/Y^{(2)}} \colon \Br(Y^2) \to \Br(Y^{(2)})\] on the Brauer groups of the varieties.}
    Furthermore, if $\pi_1$ denotes projection onto the first factor of $Y^2 = Y\times Y$, then the composition $\Cor_{Y^2/Y^{(2)}} \circ \pi_1^* : \Br(Y) \to \Br(Y^{(2)})$ induces an injective map
    \[
        \phi \colon \frac{\Br_1(Y)}{\Br_0(Y)} \hookrightarrow \frac{\Br_1(Y^{(2)})}{\Br_0(Y^{(2)})} \,.
    \]
    \item\label{it:p2} Let $\alpha \in \Br(Y)$ and let $\beta = \Cor_{Y^2/Y^{(2)}} \circ \pi_1^*(\alpha) \in \Br(Y^{(2)})$. {There exists a dense open \(U\subset Y^{(2)}\) such that for any \(y\in U\),} 
    $y$ corresponds to a quadratic point $\tilde{y}\colon \Spec(K) \to Y$ for some degree \(2\) \'etale $\kk(y)$-algebra $K$ and we have $\beta(y) = \Cor_{K/\kk(y)}\left(\alpha(\tilde{y})\right)$.
    \item\label{it:p3} Suppose $k$ is a global field, $\Br(Y)/\Br_0(Y)$ is finite and let $B \subset \Br(Y^{(2)})/\Br_0(Y^{(2)})$ denote the image of $\Cor_{Y^2/Y^{(2)}}\circ \pi_1^*$ modulo constant algebras. If there exists a quadratic extension $K/k$ such that $Y_K(\A_K)^{\Res_{K/k}(\Br(Y))} \ne \emptyset$, then $Y^{(2)}(\A_k)^B \ne \emptyset$.      
    \item\label{it:p4} Let $B \subset \Br(Y^{(2)})/\Br_0(Y^{(2)})$ denote the image of $\Cor_{Y^2/Y^{(2)}}\circ \pi_1^*$ modulo constant algebras. Suppose that $k$ is a global field, that $Y^{(2)}(\A_k)^B \ne \emptyset$, and that $Y$ satisfies the hypotheses of Lemma~\ref{lem:Brconds}. Then there exists a finite set $S \subset \Omega_k$, degree $2$ \'etale $k_v$-algebras $K_w/k_v$ for $v \in S$ and a finite extension $k_1/k$ such that for any quadratic extension $K/k$ that is linearly disjoint from $k_1$ and such that $K \otimes k_v \simeq K_w$ for $v \in S$ we have $Y_K(\A_K)^{\Br} \ne \emptyset$. In particular, there are infinitely many quadratic extensions $K/k$ such that $Y_K(\A_K)^{\Br} \ne \emptyset$.

    \end{enumerate}  
    \end{prop}
    
	\begin{proof}
		\textbf{\eqref{it:p1}:} 
        Let $\Delta = \{ (y,y) \;:\; y \in Y \}\subset Y^2$ denote the diagonal subscheme and let \(\Bl_\Delta Y^2\) denote the blow-up of $Y$ along $\Delta$.  Observe that the \(S_2\)-action on \(Y^2\) extends to an action on \(\Bl_\Delta Y^2\) whose fixed locus is the exceptional divisor \(E_{\Delta}\); we claim that the quotient \((\Bl_\Delta Y^2)/S_2\) is smooth (equivalently geometrically regular). Since \(\Bl_\Delta Y^2\) is smooth, the quotient \((\Bl_\Delta Y^2)/S_2\) is automatically smooth away from the branch locus.  Let \(y\in E_{\Delta}\).  Since \(E_{\Delta}\) is a divisor, the involution acts as a pseudo-reflection on the geometric tangent space of \(y\). Since the order of the group acting is not divisible by the characteristic of $k$, the Chevalley-Shephard-Todd theorem (see, e.g.,~\cite{Smith}) implies that the dimensions of the geometric tangent spaces of \(y\) and its image in the quotient are equal. Hence the quotient is smooth at the image of \(y\). 

        Consider the following commutative diagram.
        \[
            \xymatrix{
                \Bl_{\Delta} Y^2 \ar[r] \ar[d] & {(\Bl_{\Delta} Y^2)/S_2 } \ar[d]\\
                Y^2 \ar[r]  & \Sym^2Y\\
                }
        \]
        The left vertical map is birational by definition, and since \(Y^2 \to \Sym^2Y\) is {generically} degree \(2\), the right vertical map is also birational. The top horizontal map is flat of degree \(2\)~\cite{stacks-project}*{\href{https://stacks.math.columbia.edu/tag/00R4}{Tag 00R4}}, so we have a corestriction morphism \(\Br(\Bl_{\Delta} Y^2) \to \Br((\Bl_{\Delta} Y^2)/S_2)\) that extends to the corestriction map on function fields~\cite{CTS-Brauer}*{Section 3.8}.  Since the Brauer group of smooth projective varieties is a birational invariant (and pullback along any birational map gives an isomorphism)~\cite{CTS-Brauer}*{Corollary {6.2.11}}, this yields the first claim.

        It remains to prove injectivity of the induced map $\phi$ on the quotient \(\Br_1(Y)/\Br_0(Y)\). 
        Since $\kk(Y^2)$ is Galois over $\kk(Y^{(2)})$ with Galois group generated by the involution $\sigma$ interchanging the factors of $Y\times Y$, by \cite{GS-csa}*{Chapter 3, Exercise 3}, the composition 
        \[
            \Res_{\kk(Y^2)/\kk(Y^{(2)})} \circ \Cor_{\kk(Y^2)/\kk(Y^{(2)})}  \colon \Br(\kk(Y^{2})) \to \Br\left(\kk(Y^{2})\right)
        \]
        is given by $x \mapsto x + \sigma(x)$. We may then deduce that the same formula holds for the composition $\Res_{Y^2/Y^{(2)}} \circ \Cor_{Y^2/Y^{(2)}}   \colon \Br(Y^2) \to \Br(Y^2)$ by evaluating at generic points \cite{CTS-Brauer}*{Theorem 3.5.4}. Therefore, the composition \(\Res\circ\Cor\circ \pi_1^*\) is equal to the diagonal map $\Br(Y) \to \Br(Y)\oplus\Br(Y) \to \Br(Y^2)$ sending $\alpha$ to $\pi_1^*\alpha + \sigma(\pi_1^*\alpha) = \pi_1^*\alpha + \pi_2^*\alpha$. 
        
        If $\Pic(\Ybar)$ is torsion free, then $\Pic(\Ybar) \oplus \Pic(\Ybar) \simeq \Pic(\Ybar^2)$ (see \cite{SZproducts}*{Prop. 1.7}). So the diagonal map together with the Hochschild-Serre spectral sequence gives a commutative diagram
        \[
            \xymatrix{
            \HH^1(k,\Pic(\Ybar)) \ar[r] & \HH^1(k,\Pic(\Ybar))^{\oplus 2} \ar@{=}[r] & \HH^1(k,\Pic(\Ybar^2))\\
            \Br_1(Y)/\Br_0(Y) \ar@{^{(}->}[u] \ar[r] &  \left(\Br_1(Y)/\Br_0(Y)\right)^{\oplus 2} \ar@{^{(}->}[u] \ar[r] &  \Br_1(Y^2)/\Br_0(Y^2) \ar@{^{(}->}[u]\,.
            }
        \]
        As the composition along the top row is injective, the same must be true of the composition along the bottom row. This composition is induced by $\Res\circ\Cor\circ\pi_1^*$ and it factors through the map $\phi$ in the last statement of~\eqref{it:p1}, so $\phi$ must also be injective.

		\textbf{\eqref{it:p2}:} {Since \(Y^{(2)}\) is birational to \(\Sym^2Y\), there is an open set \(U\subset Y^{(2)}\) that is isomorphic to an open set of the regular locus of \(\Sym^2 Y\), i.e., the image of \(Y^2 - \Delta\).  For \(y\in U\), we obtain \(\tilde{y}\) by taking the preimage of \(y\) under \(Y^2\to \Sym^2 Y \dasharrow Y^{(2)}\).}  The points $y$ and $\tilde{y}$ fit into a commutative diagram displayed on the left below. This induces the diagram displayed on the right. Commutativity of the latter gives the result. 
        \[
        	\xymatrix{
        		\Spec(\kk(y)) \ar[r]^-y & Y^{(2)} && \Br(\kk(y)) & \Br(Y^{(2)}) \ar[l]_-{y^*} \\
        		f^{-1}(y) \ar@{=}[d]\ar[u]\ar[r] & Y\times Y \ar@{.>}[u]_f\ar[d]^{\pi_1} &\Rightarrow& \Br(f^{-1}(y)) \ar[u]^{\Cor}& \Br(Y\times Y) \ar[u]_\Cor\ar[l]\\
        		\Spec(K) \ar[r]^-{\tilde{y}} & Y && \Br(K) \ar@{=}[u] & \Br(Y) \ar[u]_{\pi_1^*} \ar[l]^-{\tilde{y}^*}
        	}
        \]
        
        \textbf{\eqref{it:p3}:}  Suppose that $K/k$ is a quadratic extension, $Y_K(\A_K)^{\Res_{K/k}(\Br(Y))} \ne \emptyset$ and that $\beta = \Cor_{Y^2/Y^{(2)}}(\pi_1^*(\alpha)) \in \Br(Y^{(2)})$ represents a class in $B$ that is the image of $\alpha \in \Br(Y)$. Since $\Br(Y)/\Br_0(Y)$ is finite, $Y_K(\A_K)^{\Res_{K/k}(\Br(Y))}$ is an open subset of $Y_K(\A_K)$ in the adelic topology which we have assumed is nonempty. So for any $v \in \Omega_k$ the image of the projection map $Y_K(\A_K)^{\Res_{K/k}(\Br(Y))} \to \prod_{w \mid v} Y_K(K_w) = Y(K\otimes k_v)$ is a nonempty open subset and therefore contains a quadratic point $\tilde{y}_v : \Spec(K\otimes k_v) \to Y$ corresponding to a point {\(y_v\in U(k_v)\), where \(U\) is the open set from~\eqref{it:p2}.} For the case that $v$ does not split in $K$ we are using the fact that $Y(K_w) \ne Y(k_v)$ since $k_v$ is a local field (see, e.g., \cite{LiuLorenzini}*{Proposition 8.3}). By~\eqref{it:p2} we have 
        \[ \sum_{v \in \Omega_k}\inv_v\beta(y_v) = \sum_{v \in \Omega_k}\inv_v\Cor_{K\otimes k_v/k_v}(\alpha(\tilde{y_v})) = 0\,. \]
        So the adelic point $y = (y_v) \in Y^{(2)}(\A_k)$ is orthogonal to $\beta$.
            
        \textbf{\eqref{it:p4}:}  
        Suppose $Y^{(2)}(\A_k)^B \ne \emptyset$. The hypothesis in Lemma~\ref{lem:Brconds} implies that $\Br(Y)/\Br_0(Y)$ and, hence, $B$ is finite. Thus, $Y^{(2)}(\A_k)^B$ is open and, arguing as in~\eqref{it:p3}, we see that, for each $v \in \Omega_k$, its image in $Y(k_v)$ contains a point \(y_v\in U(k_v)\) corresponding to a quadratic point $\tilde{y}_v \colon \Spec(K_v) \to Y$, where $K_v$ is an \'etale $k_v$-algebra of degree $2$. Moreover, by~\eqref{it:p2} if $\alpha \in \Br(Y)$ and $\beta = \Cor_{Y^2/Y^{(2)}}(\pi_1^*(\alpha))$, then $\beta(y_v) = \Cor_{K_v/k_v}(\alpha(\tilde{y}_v))$. By assumption $\sum_{v \in \Omega_k} \inv_v\beta(y_v) = 0$, so $(\tilde{y}_v)_{v \in \Omega_k}$ is an effective adelic $0$-cycle of degree $2$ on $Y$ which is orthogonal to the Brauer group of $Y$. Under the additional hypotheses of~\eqref{it:p4}, $\Br(Y)/\Br_0(Y)$ is finite and, by Lemma~\ref{lem:Brconds}, there is an extension $k_1/k$ such that for $K/k$ linearly disjoint from $k_1$, $\Res_{K/k}\colon \Br(Y) \to \Br(Y_K)/\Br_0(Y_K)$ is surjective. Moreover, for any set $\alpha_1,\dots,\alpha_n \in \Br(Y)$ of representatives for $\Br(Y)/\Br_0(Y)$, there  is a finite set $S \subset \Omega_k$ such that for all $i = 1,\dots,n$ and all $v \not\in S$ the evaluation maps $\inv_v\circ\alpha_i\colon Y(k_v) \to \Q/\Z$ are constant (see \cite{GoodReductionBM}*{Lemma 1.2 \& Theorem 3.1}). In particular $Y(k_v) \ne \emptyset$ for $v \not\in S$. Let $K/k$ be a quadratic extension linearly disjoint from $k_1$ and such that $K\otimes k_v \simeq K_w$ for $v \in S$. By {weak approximation on \(k^{\times}\)} the map $k^\times/k^{\times 2} \to \prod_{v\in S'}k_v^{\times}/k_v^{\times 2}$ is surjective for any finite set of primes $S' \subset \Omega_k$, so such extensions $K/k$ do in fact exist. Any adelic point $(x_w)_{w \in \Omega_K} \in Y_K(\A_K)$ such that $\tilde{y}_v = \sum_{w\mid v} x_w$ for $v \in S$ will be orthogonal to $\Br(Y_K)$.
	\end{proof}

%%%%%%%%%%%%%%%%%%%%%%%%%%%%%%%%%%%%%%%%%%%%%%%%%%%%%%%%%%%%%%%%%%%%%%%%%%%%%%%%
%%%%%%%%%%%%%%%%%%%%%%%%%%%%%%%%%%%%%%%%%%%%%%%%%%%%%%%%%%%%%%%%%%%%%%%%%%%%%%%%
\section{Pencils of quadrics in \texorpdfstring{\(\mathbb{P}^4\)}{P4} and associated objects}\label{sec:Correspondences}
%%%%%%%%%%%%%%%%%%%%%%%%%%%%%%%%%%%%%%%%%%%%%%%%%%%%%%%%%%%%%%%%%%%%%%%%%%%%%%%%
%%%%%%%%%%%%%%%%%%%%%%%%%%%%%%%%%%%%%%%%%%%%%%%%%%%%%%%%%%%%%%%%%%%%%%%%%%%%%%%%

    Let $\mathcal{Q}\subset \PP^4\times \PP^1$ be a pencil of quadrics, i.e., the zero locus of a bihomogeneous polynomial $Q$ of degree $(2,1)$, defined over a field $k$ of characteristic different from $2$. If the projection map $\calQ \to \PP^1$ is generically smooth, then we may naturally associate three objects.  First, we may consider the base locus $X = X_{\mathcal{Q}}\subset \PP^4$ of the pencil of quadrics, i.e., $\cap_{t\in \PP^1}\mathcal{Q}_t$, where $\calQ_t \subset \PP^4$ denotes the fiber over $t \in \PP^1$.  This is a degree $4$ projective surface.  
    Second, we may consider the subscheme $\scrS\subset \PP^1$ parameterizing the singular quadrics in the pencil. If $Q$ is any degree $(2,1)$ form defining $\mathcal{Q}$, then $\scrS$ is given by the vanishing of $\det(M_{Q}$), where $M_Q$ denotes the symmetric matrix corresponding to $Q$ considered as a quadratic form whose coefficients are linear polynomials in the homogeneous coordinate ring of $\PP^1$. Since $\calQ \to \PP^1$ is generically smooth, $\scrS \subset \PP^1$ is a degree $5$ subscheme.
    Third, we may consider the fourfold $\calG = \calG_{\calQ}\to \PP^1$ that parametrizes lines on quadrics in the pencil; the generic fiber of $\calG$ is a Severi-Brauer variety with index dividing $4$ and order dividing $2$~\cite{EKM-QuadraticForms}*{Ex. 85.4}.

    \begin{remark}
        Over a field of characteristic \(2\), \(\det(M_{Q})\) is identically \(0\) since \(M_Q\) is a \(5x5\) skew-symmetric matrix, and so the correspondences between these objects already fails.  Due to this, the assumption that \(k\) has characteristic different from \(2\) will remain in force for the remainder of the paper.
    \end{remark}

    Each of these objects has been well-studied, and their conditions for smoothness are known to be closely related.
    \begin{prop}\label{prop:EquivSmoothness}
        Let $\calQ\subset \PP^4\times \PP^1$ be a pencil of quadrics over a field of characteristic different from \(2\).  Then the following are equivalent:
        \begin{enumerate}
            \item The base locus $X$ is smooth and purely of dimension \(2\), in which case $X$ is a del Pezzo surface of degree $4$;\label{cond:smoothdp4}
            \item The degree $5$ subscheme $\scrS\subset \PP^1$ is reduced;\label{cond:reducedS}
            \item For every \(s\in \scrS\), the fiber \(\calQ_s\) is rank \(4\) and the vertex of \(\calQ_s\) does not lie on any other quadric in the pencil; and\label{cond:verticesdistinct}
            \item The fourfold $\calG$ is smooth, the map $\calG\to \PP^1$ is smooth away from $\scrS$, and above $\scrS$ the fibers are geometrically reducible. \label{cond:calGSmooth}
        \end{enumerate}
    \end{prop}
    \begin{proof}
        The equivalence of conditions~\eqref{cond:smoothdp4},~\eqref{cond:reducedS}, and~\eqref{cond:verticesdistinct} is given by~\cite{Reid}*{Prop. 2.1}. The equivalence of~\eqref{cond:calGSmooth} with any (equivalently all) of the others is given by~\cite{Reid}*{Thm. 1.10}.
    \end{proof}
    
    \begin{defn}
        A pencil of quadrics $\calQ$ over a field of characteristic different from \(2\) satisfies \((\dagger)\) if any of the equivalent conditions in Proposition~\ref{prop:EquivSmoothness} hold.  Given a pencil $\calQ$ satisfying \((\dagger)\), we define $\eps_{\scrS} \in \kk(\scrS)/\kk(\scrS)^{\times2}$ to be the discriminant of a smooth hyperplane section of $\calQ_{\scrS}$; note that the square class of the discriminant does not depend on the choice of hyperplane, nor on the choice of a defining equation for $\calQ_{\scrS}$.
    \end{defn}

    Given a pencil of quadrics satisfying \((\dagger)\), there are even stronger connections among these three objects.
    \begin{prop}\label{prop:ObjectConnections}
        Let $\calQ$ be a pencil of quadrics satisfying \((\dagger)\).  Let $X = X_{\calQ}, \calG = \calG_{\calQ}, $ and $(\scrS, \eps_{\scrS}) = (\scrS_{\calQ}, \eps_{\scrS_{\calQ}})$.  
        \begin{enumerate}
            \item The variety $\calG$ is birational to the symmetric square \(\Sym^2(X)\) of $X$. {Moreover,} \(\calG(k) \neq \emptyset\) if and only if \(X(K) \neq \emptyset\) for some quadratic extension \(K/k\).\label{part:Sym2calGbirational}
            \item\label{prop:connections2} The residues of the Brauer class $[\calG_{\kk(\PP^1)}] \in \Br \kk(\PP^1)$ are 
            \[
                \eps_{\scrS} \in \kk(\scrS)^{\times}/\kk(\scrS)^{\times2} \simeq \HH^1\left(\kk(\scrS), {\tfrac12}\Z/\Z\right) \subset 
                \bigoplus_{t\in (\PP^1)^{(1)}}\HH^1(\kk(t), \Q/\Z).
            \]\label{part:residuesofcalG}
            In particular, $\Norm_{\kk(\scrS)/k}(\eps_{\scrS})\in k^{\times2}$.
            \item Given a pair $(\scrS', \eps_{\scrS'})$ where $\scrS'\subset \PP^1$ is a reduced degree \(5\) subscheme and a class \(\eps_{\scrS'}\in \kk(\scrS')^{\times}/\kk(\scrS')^{\times2}\) of square norm, there exists a unique (up to isomorphism) pencil of quadrics $\calQ$ such that \( (\scrS', \eps_{\scrS'}) = (\scrS_{\calQ}, \eps_{\scrS_{\calQ}}) \).  Thus, for any $t\in \PP^1 - \scrS$, $[\calG_t] \in \Br(\kk(t))$ is determined by $(\scrS, \eps_{\scrS})$.\label{part:ConstructionFromEps}
        \end{enumerate}
    \end{prop}
    \begin{remark}
        The second statement of Part~\eqref{part:residuesofcalG} provides an alternate proof of a proposition by Wittenberg~\cite{Wittenberg}*{Prop. 3.39}.  
    \end{remark}
    \begin{proof}
        \textbf{\eqref{part:Sym2calGbirational}:} Consider a point $(x,x')\in X\times X -\Delta$, where $\Delta$ denotes the diagonal image of $X$, and let $\ell_{\{x,x'\}}$ be the line joining them. For generic $(x,x')$ the line $\ell_{\{x,x'\}}$ is not contained in $X$, in which case we claim that $\ell_{\{x,x'\}}$ lies on a quadric in the pencil containing $X$. This quadric will be unique since {a line that is contained in more than one quadric in the pencil lies on \(X\).}
        To see that $\ell_{\{x,x'\}}$ is contained in some quadric note that the intersections $\calQ_t \cap \ell_{\{x,x'\}}$ determine a nonzero pencil of binary quadrics (i.e., quadrics in $\PP^1$) that all contain \(x\) and \(x'\).  The singular binary quadrics of this pencil are rank at most \(1\) and contain the distinct points $x$ and $x'$ so they must be identically \(0\) on \(\ell_{\{x,x'\}}\). 
        
        Therefore, we have a rational map
        \[
        	f \colon X \times X \dashrightarrow \mathcal{G}\,,\quad (x,x') \mapsto (t_{\{x,x'\}},\ell_{\{x,x'\}})\,,
        \]
        defined on the locus of pairs $(x,x') \in X\times X - \Delta$ such that the line $\ell_{\{x,x'\}}$ is not contained in $X$, where $t_{\{x,x'\}} \in \PP^1$ is such that $\ell_{\{x,x'\}} \subset \calQ_{t_{\{x,x'\}}}$. Noting that a line $\ell \subset \calQ_t$ which is not contained in $X$ intersects $X$ in \(0\)-dimensional scheme of degree $2$ we see that $f$ is {dominant}, generically of degree $2$, and factors through the symmetric square of $X$.  Thus, the induced map $\Sym^2 X \dasharrow \calG$ is birational.

        If $\calG(k) \ne \emptyset$, then the Lang-Nishimura Theorem (see, e.g.,~\cite{Poonen-Qpoints}*{Theorem 3.6.11}) (which applies since $\calG$ is smooth) implies that $\Sym^2(X)(k) \ne \emptyset$ and, consequently, that there is a quadratic point on $X$.  In particular, there is a quadratic extension $K/k$ with $X(K) \ne \emptyset$. 
        Conversely, if $X(K) \ne \emptyset$ for some quadratic extension $K/k$, then $X(K)$ is infinite by \cite{SalbergerSkorobogatov}*{Theorem (0.1)}. The line through any Galois stable pair of distinct points gives a $k$-rational point on $\calG$.

        \textbf{\eqref{part:residuesofcalG}:} Let $t\in \PP^1$.  By~\cite{Reid}*{Thms. 1.2 and 1.8}, the fiber $\calG_t$ is smooth and geometrically irreducible exactly when $\calQ_t$ has rank $5$.  Thus, for all $t\in \PP^1-\scrS$, the class $[\calG_{\kk(\PP^1)}]$ has trivial residue at $t$.  By Proposition~\ref{prop:EquivSmoothness} and assumption \((\dagger)\), if $t\in \scrS$, then $\calQ_t$ has rank $4$. If $\calQ_t$ is rank $4$ and has square discriminant, then by~\cite{Reid}*{Thm. 1.8} the fiber $\calG_t$ is reducible and split over $\kk(t)$.  If $\calQ_t$ is rank $4$ and has nonsquare discriminant, then the same result of Reid says that $\calG_t$ is irreducible and non-split over $\kk(t)$, but becomes split over the quadratic discriminant extension.  Thus, the residue of $[\calG_{\kk(\PP^1)}]$ at $t$ is the discriminant of $\calQ_t$~\cite{Frossard}*{Prop. 2.3}.  By definition of $\eps_{\scrS}$, this gives the first statement.  The second statement now follows from the Faddeev exact sequence for $\Br \kk(\PP^1)$ (see~\cite{GS-csa}*{Thm 6.4.5} or~\eqref{eq:FaddeevCommutativeDiagram}).

        \textbf{\eqref{part:ConstructionFromEps}:} The first statement is a theorem of Flynn~\cite{Flynn} which was expanded upon by Skorobogatov~\cite{SkoDP4}.  The second statement follows from the first together with the Faddeev exact sequence for $\Br(\kk(\PP^1))$ (see~\cite{GS-csa}*{Thm 6.4.5} or~\eqref{eq:FaddeevCommutativeDiagram}).
    \end{proof}
    
    The proceeding proposition together with Theorem~\ref{thm:localquadpts} yields the following.
    \begin{cor}\label{cor:locptsG}
        Assume $k$ is a local field of characteristic not equal to $2$.  For any pencil of quadric threefolds $\calQ\to \PP^1$ satisfying $(\dagger)$, $\calG_{\calQ}(k)\neq\emptyset$.\hfill \qed
    \end{cor}

%%%%%%%%%%%%%%%%%%%%%%%%%%%%%%%%%%%%%%%%%%%%%%%%%%%%%%%%%%%%%%%%%%%%%%%%%%%%%%%%
\subsection{Notation}\label{subsec:notation}
%%%%%%%%%%%%%%%%%%%%%%%%%%%%%%%%%%%%%%%%%%%%%%%%%%%%%%%%%%%%%%%%%%%%%%%%%%%%%%%%
    For a pencil of quadrics that satisfies $(\dagger)$ we will move freely between the objects $\calQ, X = X_{\calQ}, \calG = \calG_{\calQ}, $ and $(\scrS, \eps_{\scrS}) = (\scrS_{\calQ}, \eps_{\scrS_{\calQ}})$. We will assume that $\scrS\subset \A^1 = \PP^1 - \infty$. This can be arranged by an automorphism of $\PP^1$, provided $k$ has at least $5$ elements. We will write $k[T]$ for the coordinate ring of $\A^1$ and let $f(T)$ be the unique monic polynomial whose vanishing defines $\scrS$.
    
    Let $Q_{\A^1} \in k(T)[x_0,\dots,x_4]$ be a quadratic form whose coefficients are linear polynomials in $k[T]$ and whose vanishing defines $\calQ_{\A^1}$ on $\A^1 \subset \PP^1$. While $Q_{\A^1}$ is only defined up to multiplication by an element of $k^\times$, none of our results depend on this choice. For a (possibly reducible) subscheme $\scrT \subset \A^1 = \Spec(k[T])$, the canonical map $k[T] \to \kk(\scrT)$ can be applied to the coefficients of $Q_{\A^1}$ to obtain a quadratic form $Q_\scrT$ over the $k$-algebra $\kk(\scrT)$ whose vanishing defines $\calQ_\scrT = \calQ \times_{\PP^1} \scrT$. In particular, for $a \in k = \A^1(k)$, the form $Q_a$ is obtained by evaluating the coefficients of $Q_{\A^1}$ at $a$. We define $Q_\infty = Q_1 - Q_0$, so that $Q_{\A^1} = Q_0 + TQ_\infty$.
    
    We will write $\theta$ for the image of $T$ in $\kk(\scrS) = k[T]/\langle f(T)\rangle$. For a subscheme $\scrT \subset \scrS$ we use $\eps_\scrT \in \frac{\kk(\scrT)^\times}{\kk(\scrT)^{\times 2}} \subset \frac{\kk(\scrS)^\times}{\kk(\scrS)^{\times 2}}$ to denote the discriminant corresponding to $\calQ_\scrT$. We will use $\N$ to denote any map induced in an obvious way by the norm map $\Norm_{\kk(\scrS)/k} \colon \kk(\scrS) \to k$.  Note that $\Norm_{\kk(\scrT)/k}(\eps_\scrT) = \Norm_{\kk(\scrS)/k}(\eps_\scrT) = \N(\eps_\scrT)$.

%%%%%%%%%%%%%%%%%%%%%%%%%%%%%%%%%%%%%%%%%%%%%%%%%%%%%%%%%%%%%%%%%%%%%%%%%%%%%%%%
    \subsection{Alternate proof of Theorem~\ref{thm:localquadpts} for odd residue characteristic}\label{sec:Alternate}
%%%%%%%%%%%%%%%%%%%%%%%%%%%%%%%%%%%%%%%%%%%%%%%%%%%%%%%%%%%%%%%%%%%%%%%%%%%%%%%%
    We now give an alternate proof of Theorem~\ref{thm:localquadpts} (valid for local fields of odd residue characteristic) which avoids the classificiation of reducible special fibers.
   
    \begin{prop}
   	    Let $X \subset \PP^4_k$ be a smooth complete intersection of two quadrics over a local field $k$ of characteristic not equal to $2$. Then $X$ has index dividing $2$. If the residue characteristic of $k$ is odd, then there is a quadratic extension $K/k$ such that $X$ has a $K$-point.
    \end{prop}
   
    \begin{proof}
        First let us prove that $X$ has a quadratic point assuming that $s \in \scrS(k) \ne \emptyset$. After a change of coordinates on the $\PP^1$ parameterizing the pencil and a change of coordinates on $\PP^4$, we may assume that $s = 0$, that $Q_0 = Q_0(x_0,x_1,x_2,x_3)$, and that $Q_\infty = \tilde{Q}_{\infty}(x_0,x_1,x_2,x_3) + x_4^2$. If $\calQ_0$ contains a smooth $k$-point, then the line joining the vertex of $\calQ_0$ and this point will intersect $X$ in a degree $2$ subscheme, which shows that $X$ has a quadratic point.  Thus, we may restrict to the case that $\calQ_0$ has no smooth $k$-points.
        
        Projection away from the vertex of $\calQ_0 \subset \PP^4_k$ gives a double cover $X \to Y:= \calQ_0 \cap V(x_4)$ onto the quadric surface $Y$.  Since $\calQ_0$ has no smooth $k$-points, $Y(k) = \emptyset$. We will prove that, in this case, the branch curve $C$ of the double cover $X\to Y$ has a quadratic point.  Note that by definition of the double cover, $C = X\cap V(x_4)$ and so is a degree $4$ genus $1$ curve that is the base locus of the pencil of quadric surfaces $\calQ' \to \PP^1$ with $\calQ'_t = \calQ_t \cap V(x_4)$. Moreover, $C$ is a $2$-covering of the degree $2$ genus one curve $C'$ given by the equation $y^2 = \det(M)$ where $M$ is the $4\times4$ symmetric matrix with entries in $\HH^0(\OO_{\PP^1}(1))$ corresponding to a defining equation for $\calQ'$ (see \cite{JNT}). 

        Consider the fiber of $C'\to \PP^1$ above $0$.  By definition of $\calQ'$, this is given by the equation $y^2 = \disc(\calQ_0\cap V(x_4))$.  By assumption, $\calQ_0\cap V(x_4)$ has no $k$-points.  Since there is (up to isomorphism) a unique rank $4$ quadric over the local field $k$ that is anisotropic and it has square discriminant, we conclude that $\disc(\calQ_0\cap V(x_4))$ is a square and so $C'(k)\neq \emptyset$. Consequently, $C' \simeq \Jac(C)$ and so the order of $C$ in $\HH^1(k,\Jac(C))$ divides $2$. By a result of Lichtenbaum \cite{Lichtenbaum}*{Theorems 3 $\&$ 4} it follows that $C$ has a point defined over some quadratic extension of the local field $k$. The aforementioned result of Lichtenbaum is stated for $k$ a $p$-adic field, but the proof works for any local field due to Milne's extension of Tate's local duality results to positive characteristic~\cite{ADT}*{Cor. I.3.4, Rmk. I.3.5, Thm. III.7.8}.
            
        Now we can deduce the statement in the proposition. The scheme $\scrS \subset \PP^1_k$ parameterizing singular quadrics in the pencil has degree $5$, so there is an odd degree extension $k'/k$ such that $\scrS(k') \ne \emptyset$. By what we have shown above, $X$ has a $K$-rational point for some quadratic extension $K/k'$. It follows that $X$ has index at most $2$. If the residue characteristic is odd, then the inclusion $k \subset k'$ induces an isomorphism $k^\times/k^{\times 2} \simeq k'^\times/k'^{\times 2}$, so $K$ contains a quadratic extension $k_2/k$ as an odd index subfield. By the theorems of Amer, Brumer and Springer \cites{Amer,Brumer,Springer},  we have $X(K) \ne \emptyset \Rightarrow X(k_2) \ne \emptyset$, so $X$ has a $k_2$-point.
	\end{proof}

    \begin{rmk}\label{rmk:410}
        The preceding proof can be adapted to give an easy proof that a locally solvable del Pezzo surface of degree $4$ over a global field over a field of characteristic different from \(2\) must have index dividing $2$. Indeed, over some odd degree extension $X$ may be written as a double cover of a quadric surface, which is known to satisfy the Hasse principle. Hence $X$ obtains a rational point over some extension of degree $2m$ with $m$ odd.
    \end{rmk}

%%%%%%%%%%%%%%%%%%%%%%%%%%%%%%%%%%%%%%%%%%%%%%%%%%%%%%%%%%%%%%%%%%%%%%%%%%%%%%%%
%%%%%%%%%%%%%%%%%%%%%%%%%%%%%%%%%%%%%%%%%%%%%%%%%%%%%%%%%%%%%%%%%%%%%%%%%%%%%%%%
\section{Arithmetic of the space of lines on the quadrics in the pencil}\label{sec:calG}
%%%%%%%%%%%%%%%%%%%%%%%%%%%%%%%%%%%%%%%%%%%%%%%%%%%%%%%%%%%%%%%%%%%%%%%%%%%%%%%%
%%%%%%%%%%%%%%%%%%%%%%%%%%%%%%%%%%%%%%%%%%%%%%%%%%%%%%%%%%%%%%%%%%%%%%%%%%%%%%%%

In this section we develop the main tools to prove Theorems~\ref{thm:MainIndThm} and~\ref{thm:MainThm} over global fields of characteristic not equal to $2$. We maintain the notation defined in Section~\ref{subsec:notation}.  Specifically, $\calQ \to \PP^1$ is a pencil of quadrics in $\PP^4_k$ over a field $k$ of characteristic not equal to $2$ which satisfies $(\dagger)$, and we let $X = X_{\calQ}$, $\calG = \calG_{\calQ}$ and $(\eps_{\scrS}, \scrS) = (\eps_{\scrS_{\calQ}}, \scrS_{\calQ})$.

In Section~\ref{sec:BrG}, we compute $\Br(\calG)/\Br_0(\calG)$ and construct explicit representatives in $\Br(\calG)$, denoted by $\beta_\scrT$, which are determined by subsets $\scrT\subset \scrS$ such that $\N(\eps_\scrT)\in k^{\times2}$. In Section~\ref{sec:Clif}, we study the rank $4$ quadrics $\calQ_{\scrT}$ corresponding to subsets $\scrT\subset \scrS$ such that $\N(\eps_\scrT)\in k^{\times2}$.  We use Clifford algebras associated to these rank $4$ quadrics to define constant Brauer classes $\CQ{\scrT}\in \Br(k)$ and we show how these are related to the kernel of the canonical map $\Br(k) \to \Br(X)$. The two constructions come together in Sections~\ref{sec:LocalEv} where we show how the $\CQ{\scrT}$ arise when evaluating $\beta_{\scrT}$ at certain local points of $\calG$ (see Lemmas~\ref{lem:EvBetaSqrtEps} and \ref{lem:betanonconstant}).  Finally, in Section~\ref{sec:AdelicEv}, we deduce consequences for the evaluation of $\beta_{\scrT}$ at adelic points of $\calG$.

%%%%%%%%%%%%%%%%%%%%%%%%%%%%%%%%%%%%%%%%%%%%%%%%%%%%%%%%%%%%%%%%%%%%%%%%%%%%%%%%
\subsection{The Brauer group of \texorpdfstring{$\calG$}{G}}\label{sec:BrG}
%%%%%%%%%%%%%%%%%%%%%%%%%%%%%%%%%%%%%%%%%%%%%%%%%%%%%%%%%%%%%%%%%%%%%%%%%%%%%%%%
	It follows from the Faddeev exact sequence (see \cite{GS-csa}*{Thm 6.4.5}) that the homomorphism
	\[
		\gamma' \colon \kk(\scrS)^\times \ni \eps \mapsto \Cor_{\kk(\scrS)/k}(\eps,T-\theta) \in \Br(\kk(\PP^1))
	\]
	induces an isomorphism
	\begin{equation}\label{eq:gamma}
		\gamma \colon \ker\left(\N\colon \frac{\kk(\scrS)^\times}{\kk(\scrS)^{\times 2}} \to \frac{k^\times}{k^{\times 2}}\right) \simeq \ker\left({\Br(\PP^1-\scrS)}[2] \stackrel{\infty^*}{\longrightarrow} {\Br k}[2]\right)\,,
	\end{equation}
    where $\infty^*$ denotes evaluation of the Brauer class at $\infty\in \PP^1 - \scrS$. 
    Recall that $\N(\eps_{\scrS})\in k^{\times2}$ by Proposition~\ref{prop:ObjectConnections}\eqref{prop:connections2}.

    Define $\beta = \pi^*\gamma \colon \ker\left(\N\colon \frac{\kk(\scrS)^\times}{\kk(\scrS)^{\times 2}} \to \frac{k^\times}{k^{\times 2}}\right) \to \Br(\kk(\calG))$. For $\scrT \subset \scrS$ such that $\N(\eps_{\scrT})\in k^{\times2}$, we set $\beta_\scrT := \beta(\eps_\scrT)$.

    \begin{prop}\label{prop:BrFaddeev}
        The map $\beta$ induces a homomorphism
        \begin{equation}
            \ker \left(\N\colon \bigoplus_{s\in \scrS}\langle \eps_s\rangle \to k^{\times}/k^{\times2}\right) \stackrel{\beta}\to \Br(\calG),
        \end{equation}
        whose image surjects onto $\Br(\calG)/\Br_0(\calG)$. Furthermore, $\beta_{\scrS} = [\calG_{\infty}] \in \Br _0 (\calG)$, and for all $\scrT\subset \scrS$ with $\N(\eps_{\scrT})\in k^{\times2}$ and $\eps_{\scrT}\neq \eps_{\scrS}\in \kk(\scrS)^{\times}/\kk(\scrS)^{\times2}$, we have
        \[
            \beta_{\scrT} \in \Br_0(\calG) \subset\Br (\calG) \;\Longleftrightarrow \;\beta_{\scrT}= 0\in \Br(\calG)\; \Longleftrightarrow \;\eps_{\scrT}\in \kk(\scrT)^{\times2}.
        \]
    \end{prop}

\begin{cor}\label{cor:propsBrG}
    \hfill
	\begin{enumerate}
	       	\item\label{BB01} Every nontrivial element of $\Br(\calG)/\Br_0(\calG)$ is represented by $\beta_{\scrT}$ for some degree $2$ subscheme $\scrT \subset \scrS$ with $\N(\eps_\scrT) \in k^{\times 2}$.     
        \item\label{BrSizeBound} $\Br(\calG)/\Br_0(\calG) \simeq (\Z/2\Z)^n$ for some $n \in \{0,1,2\}$.
        \item\label{BB04} If $\Br(\calG)/\Br_0(\calG)$ is not cyclic, then every degree $2$ subscheme $\scrT\subset \scrS$ with $\N(\eps_\scrT)\in k^{\times2}$ \emph{must} be reducible.  
        \item \label{statement:SingleRatlRank4} Let $s_0\in \scrS(k)$ be such that there exists an $s'\in \scrS(k)$ with $\beta_{\{s_0,s'\}}\in \Br(\calG)-\Br_0(\calG)$.  Then 
          $\{\beta_{\{s_0,s\}} : s\in \scrS(k), \; \N(\eps_{\{s_0,s\}})\in k^{\times2}\}$ generates $\Br(\calG)/{\Br_0(\calG)}.$
       	\item\label{GoodGeneratingSet} There is a collection $\mathbb{T}$ of degree $2$ subschemes of $\scrS$ and an element $\eps \in k^\times$, such that
       		\begin{itemize}
       			\item $\N(\eps_\scrT) \in k^{\times 2}$ for all $\scrT \in \mathbb{T}$;
       			\item $\{ \beta_\scrT \;:\; \scrT \in \mathbb{T} \}$ generates $\Br(\calG)/\Br_0(\calG)$;
       			\item for all $s \in \cup_{\scrT\in\mathbb{T}}\scrT$, the image of $\eps$ in $\kk(s)^\times/\kk(s)^{\times 2}$ is equal to $\eps_s$; and
       			\item for any extension $L/k$ and any $s \in \cup_{\scrT\in\mathbb{T}}\scrT$, $\eps \in \kk(s_L)^{\times 2}$ if and only if $\eps \in \kk(s'_L)^{\times 2}$ for all $s' \in \cup_{\scrT\in\mathbb{T}}\scrT$.
       		\end{itemize}
        \item $\Br(\calG)/\Br_0(\calG) \simeq \HH^1(k,\Pic(\Xbar))$. \label{statement:BrGH1PicX}
        \item\label{BB07} If $k$ is a local or global field, then the injective map $\Br(X)/\Br_0(X) \to \Br(\calG)/\Br_0(\calG)$ given by~Proposition~\ref{prop:sym2}\eqref{it:p1} is an isomorphism.\label{statement:BrGIsomBrX}
        \end{enumerate}
\end{cor}

\begin{proof}[Proof of Corollary~\ref{cor:propsBrG}]
    The proposition implies that \(\beta_{\scrS}\in \Br_0 (\calG)\) and, for any \(\scrT\subset\scrS\) such that \(\N(\eps_{\scrT})\in \kk(\scrT)^{\times2}\), that \(\beta_{\scrT} = \beta_{\scrS - \scrT}\in \Br(\calG)/\Br_0(\calG)\). Since $\scrS$ has degree $5$, it follows that every nontrivial element in $\Br(\calG)/\Br_0(\calG)$ is represented by some $\beta_\scrT$ with $\deg(\scrT) \le 2$. But if $\scrT$ has degree $1$, then $\eps_\scrT = \N(\eps_\scrT) \in k^{\times 2}$ and $\beta_\scrT = 0$. Thus we have~\eqref{BB01}. In particular, if \(\Br(\calG)\neq \Br_0\calG\), then \(\{\deg s: s\in \scrS\}\) must be \(\{3,2\},\;\{3,1,1\},\;\{2,2,1\},\;\{2,1,1,1\},\) or \(\{1,1,1,1,1\}\). Now a straightforward case by case analysis of the possible relations on $\oplus_{s\in\scrS}\langle \eps_s\rangle{\isom \oplus_{s\in \scrS}\Z/2\Z}$ allows one to deduce statements~\eqref{BrSizeBound}--\eqref{statement:SingleRatlRank4}. Given this characterization of $\Br(\calG)/\Br_0(\calG)$ in terms of degree $2$ subschemes $\scrT\subset \scrS$,~\eqref{GoodGeneratingSet} can be established using \cite{VAV}*{Lemma 3.1} for the existence of $\eps \in k^\times$. {For~\eqref{statement:BrGH1PicX}, we observe that~\cite{VAV}*{Proof of Theorem 3.4} gives a description of \(\HH^1(k,\Pic \Xbar)\) in terms of degree two subschemes \(\scrT\subset\scrS\) and the square classes \(\eps_{\scrT}\); comparing this description with~\eqref{GoodGeneratingSet} gives the desired isomorphism.}
    % and~\eqref{statement:BrGH1PicX} follows from~\cite{VAV}*{Proof of Theorem 3.4}. 
    Finally, when $k$ is a local or global field, the injective map $\Br_1(X)/\Br_0(X) \to \HH^1(k,\Pic(\Xbar))$ coming from the Hochschild-Serre spectral sequence \cite{CTS-Brauer}*{Prop. 4.3.2} is an isomorphism, so~\eqref{statement:BrGH1PicX} implies that the injective map $\Br(X)/\Br_0(X) \to \Br(\calG)/\Br_0(\calG)$ from~Proposition~\ref{prop:sym2}\eqref{it:p1} is also surjective.
\end{proof}

\begin{rmk}
    If $\scrT \subset \scrS$ is a degree $2$ subscheme with $\N(\eps_\scrT) \in k^{\times 2}$ such that the quadric $\calQ_\scrT$ has a smooth $\kk(\scrT)$-point, then \cite{VAV}*{Cor. 3.5} yields a rational map $\rho\colon X \dashrightarrow \PP^1$ such that $\rho^*\gamma(\eps_\scrT) \in \Br(X)$. One can show that the image of $\rho^*\gamma(\eps_\scrT)$ under the map $\Br(X)/\Br_0(X) \to \Br(\calG)/\Br_0(\calG)$ given by~Proposition~\ref{prop:sym2}\eqref{it:p1} is equal to the class of $\beta_\scrT$.
\end{rmk}

\begin{proof}[Proof of Proposition~\ref{prop:BrFaddeev}]
Let $\eta\in \PP^1$ be the generic point.  Since $\calG$ is smooth, $\Br(\calG)$ injects into $\Br (\calG_{\eta})$.  Further, by the Hochschild-Serre spectral sequence, we have an exact sequence
\[
    0 \to \Pic(\calG_{\eta}) \to 
    \left(\Pic(\overline{\calG_{\eta}}) \right)^{G_{\kk(\eta)}} \to 
    \Br (\kk(\eta)) \to 
    \ker \left(\Br (\calG_{\eta}) \to \Br (\overline{\calG_{\eta}})\right) \to \HH^1\left(G_{\kk(\eta)}, \Pic (\overline{\calG_{\eta}}) \right).
\]
Since $\calG_{\eta}$ is a Severi-Brauer variety, $\Pic (\overline{\calG_{\eta}})\simeq \Z$ with trivial Galois action, and $\Br (\overline{\calG_{\eta}}) = 0$.  Hence, the exact sequence simplifies to
\begin{equation}\label{eq:SimplifiedHochschildSerre}
    \Z \to 
    \Br (\kk(\eta)) \stackrel{\pi^*}{\to} 
    \Br (\calG_{\eta})  \to 0,
\end{equation}
    where the first map sends $1$ to $[\calG_{\eta}]\in \Br (\kk(\eta))$.  Thus, to determine $\Br (\calG)$, it suffices to determine $\Br (\calG) \cap \pi^*\Br(\kk(\eta))$.

    The projection map $\pi \colon \calG \to \PP^1$ induces the following commutative diagram of exact sequences where the top row is the Faddeev exact sequence \cite{GS-csa}*{Thm 6.4.5}.
    \begin{equation}\label{eq:FaddeevCommutativeDiagram}
        \begin{tikzcd}
        \Br(k) \arrow[hookrightarrow, r] \arrow[d, "\pi^*"] & 
        \Br(\kk(\eta)) \arrow[r, "(\partial_t)"] \arrow[d, two heads, "\pi^*_{\Br}"] & 
        \displaystyle{\bigoplus_{t\in (\PP^1)^{(1)}}} \HH^1(\kk(t), \Q/\Z) \arrow[rr, two heads, "\sum_t \Cor_{\kk(t)/k}"] \arrow[d, "\pi^*_{\HH^1}"] & &
        \HH^1(k, \Q/\Z)\\
        \Br(\calG) \arrow[hookrightarrow,r] & \Br(\calG_{\eta}) \arrow[r, "(\partial_x)"] & \displaystyle{\bigoplus_{t\in \PP^1}}\,
        \displaystyle{\bigoplus_{\substack{x\in \calG^{(1)} \\\pi(x)= t} } }  
        \HH^1(\kk(x), \Q/\Z), 
        \end{tikzcd}
    \end{equation}

    If $t\in \PP^1-\scrS$, then the fiber $\calG_t$ is geometrically irreducible by Proposition \ref{prop:EquivSmoothness} and hence ${\pi^*_{\HH^1}}\colon\HH^1(\kk(t),\Q/\Z) \to \HH^1(\kk(\calG_t), \Q/\Z)$ is an injection.  For $t\in \scrS$, the fiber $\calG_t$ consists of two split components that are conjugate over $\kk(t)(\sqrt{\eps_t})$.
    
    Therefore, for $t\in\scrS$, the kernel of $\pi^*\colon\HH^1(\kk(t),\Q/\Z) \to \HH^1(\kk(\calG_t), \Q/\Z)$ is the $2$-torsion cyclic subgroup corresponding to the extension $\kbar\cap \kk(\calG_t) = {\kk(t)(\sqrt{\eps_t})}$. Moreover, the residues of the kernel of $\pi^*_{\Br}$ are $(\partial_t)\left(\ker (\pi^*_{\Br}\right))= (\eps_t)_{t\in \scrS} = \eps_{\scrS}\in \kk(\scrS)/\kk(\scrS)^{\times2}$. Thus, the commutativity of the above diagram shows that 
    \[
        \ker \pi^*_{\HH^1}\bigcap \ker \sum_t\Cor_{\kk(t)/k} \simeq \ker \left(\N\colon \bigoplus_{t\in \scrS}\langle \eps_t\rangle \to k^{\times}/k^{\times2}\right).
    \]
    In particular, the image under $\beta$ of $\ker \left(\N\colon \bigoplus_{t\in \scrS}\langle \eps_t\rangle \to k^{\times}/k^{\times2}\right)$ is contained inside of $\Br(\calG)$.  Further, since $\pi^*_{\Br}$ is surjective, the image of $\ker \left(\N\colon \bigoplus_{t\in \scrS}\langle \eps_t\rangle \to k^{\times}/k^{\times2}\right)$ under $\beta$ generates $\Br (\calG)/\Br_0(\calG)$.
    
    It remains to understand which subsets $\scrT\subset \scrS$ give rise to $\beta_{\scrT}\in \Br_0(\calG)$.  If $\beta_{\scrT}\in \Br_0(\calG)$, then by definition of $\Br_0(\calG)$ there exists $\calA\in \Br(k)$ such that $\gamma(\eps_{\scrT}) - \calA\in \ker \pi^*_{\Br}$.  By~\eqref{eq:SimplifiedHochschildSerre}, the kernel of $\pi^*_{\Br}$ is generated by $[\calG_{\eta}]$.  Thus, $\gamma(\eps_{\scrT}) = [\calG_{\eta}] + \calA $ or $\gamma(\eps_{\scrT}) = \calA$, where both equalities are in $ \Br (\PP^1 - \scrS)$. The final statement of the proposition follows from these equalities after computing residues and evaluating at $\infty$.
\end{proof}

Recall that $B_{Q_t}$ denotes the bilinear form corresponding to $Q_t$.

   \begin{lemma}\label{lem:piy}
   	Let $f \colon \Sym^2(X) \dasharrow \calG$ be the birational map given in Proposition~\ref{prop:ObjectConnections} {and let \(\{x,x'\}\in \Sym^2(X) - \textup{Indet}(f)\).}  Suppose that $f(\{x,x'\}) = (t,\ell) =: y \in \calG(k)$. Then $\pi(y) = t = [ B_{Q_0}(x,x'):-B_{Q_\infty}(x,x')] \in \PP^1(k)$. If, moreover, $\scrT \subset \scrS$ is such that $\N(\eps_\scrT) \in k^{\times 2}$ and $\pi(y) \not\in \scrT \cup \{ \infty \}$, then
   	\[
   		\beta_\scrT(y) = \Cor_{\kk(\scrT)/k}\left(\eps_\scrT, -\frac{B_{Q_\scrT}(x,x')}{B_{Q_\infty}(x,x')}\right)\,.
	\]
   \end{lemma}
   
   \begin{proof}
   	{Observe that for any point $ax + bx'$ on the line $\ell_{\{x,x'\}}$ through $x$ and $x'$  we have 
   	\[Q_t(ax+bx') = B_{Q_t}(ax+bx',ax+bx') = a^2Q_t(x) + b^2Q_t(x') + 2abB_{Q_t}(x,x') = {2abB_{Q_t}(x,x')} \,.\]
    Therefore, the} line $\ell_{\{x,x'\}}$ is contained in the quadric $\calQ_t$ precisely when $B_{Q_t}(x,x') = 0$. If $B_{Q_0}(x,x') = B_{Q_\infty}(x,x') = 0$, then $\ell_{\{x,x'\}} \subset X$ in which case $f$ is not defined at $\{x,x'\}$. Otherwise, the relation $B_{Q_t}(x,x') := B_{Q_0}(x,x') + t B_{Q_\infty}(x,x') = 0$ shows that \(t= \pi(y)\in \PP^1\) must be equal to \([B_{Q_0}(x,x'):-B_{Q_\infty}(x,x')] \).
   	
   	For the second statement recall that $\beta_\scrT = \pi^*\gamma(\eps_\scrT) = \pi^*\Cor_{\kk(\scrS)/k}(\eps_\scrT,T-\theta)\,,$ where $T$ is the coordinate on $\Spec(k[T]) = \A^1 \subset \PP^1$ and $\theta$ is the image of $T$ in $\kk(\scrS)$. We have
   		$$\pi(y) - \theta = -\frac{B_{Q_0}(x,x') + \theta B_{Q_\infty}(x,x')}{B_{Q_\infty}(x,x')} = -\frac{B_{Q_\scrS}(x,x')}{B_{Q_\infty}(x,x')} \in \kk(\scrS)\,.$$
As $\gamma(\eps_\scrT)$ is unramified away from $\scrT$ we may evaluate at $\pi(y)$ to obtain 
    \[
        \beta_\scrT(y) = \gamma(\eps_\scrT)(\pi(y)) = \Cor_{\kk(\scrS)/k}\left(\eps_\scrT, -\frac{B_{Q_\scrS}(x,x')}{B_{Q_\infty}(x,x')}\right)\,.
    \]
    The projections of $\eps_\scrT \in \bigoplus_{s \in \scrS}\kk(s)^\times/\kk(s)^{\times 2}$ onto the factors corresponding to $s \in \scrS-\scrT$ are trivial. So 
   	\[
   		\beta_\scrT(y) = \Cor_{\kk(\scrS)/k}\left(\eps_\scrT, -\frac{B_{Q_\scrS}(x,x')}{B_{Q_\infty}(x,x')}\right) = \Cor_{\kk(\scrT)/k}\left(\eps_\scrT, -\frac{B_{Q_\scrT}(x,x')}{B_{Q_\infty}(x,x')}\right)\,.\qedhere
	\]
   \end{proof}

    %%%%%%%%%%%%%%%%%%%%%%%%%%%%%%%%%%%%%%%%%%%%%%%%%%%%%%%%%%%%%%%%%%%%%%%%%%%%
    \subsection{Clifford algebras and Brauer classes}\label{sec:Clif}
    %%%%%%%%%%%%%%%%%%%%%%%%%%%%%%%%%%%%%%%%%%%%%%%%%%%%%%%%%%%%%%%%%%%%%%%%%%%%

    For a quadratic form $F$ over a field of characteristic not equal to $2$ we use $\Clif(F)$ to denote the Clifford algebra of the restriction of $F$ to a maximal regular subspace, and $\Clif_0(F)$ to denote the corresponding even subalgebra. By Witt's Theorem~\cite{Lam-QF}*{Chap. I, Theorems 4.2 and 4.3}, these do not depend on the choice of maximal regular subspace. If $F$ has even rank, then $\Clif(F)$ is a central simple algebra, which will be identified with its class in the Brauer group. This extends to quadratic forms over finite \'etale algebras in the natural way{, i.e., factor by factor}. 
    
   In particular, we will consider $\Clif(Q_\scrT) \in \Br(\kk(\scrT))$ where $Q_\scrT$ is a quadratic form defining the quadric $\calQ_\scrT$ corresponding to a subscheme $\scrT \subset \scrS$. This depends on the choice of quadratic form as indicated by the following lemma.

 	\begin{lemma}\label{lem:cQ}
		Let $s \in \scrS$ and $c \in \kk(s)^\times$. Then
		\[
			\Clif(cQ_s) = \Clif(Q_s) + (\eps_s,c) \in \Br(\kk(s))\,.
		\]
	\end{lemma}
	\begin{proof}
		This follows from a short calculation using \cite{Lam-QF}*{Chap. V, Corollary 2.7}.
	\end{proof}
    
     For a rank $4$ quadric $\calQ_s, s\in\scrS$ with $\eps_s\in\kk(s)^{\times2}$, any quadratic form $Q_s$ defining $\calQ_s$ is a constant multiple of the reduced norm form of a quaternion algebra whose class in $\Br(\kk(s))$ is equal to $\Clif(Q_s)$~\cite{EKM-QuadraticForms}*{Prop. 12.4}. The following lemma gives a description of $\Clif(Q_s)$ in cases when $\eps_s\notin \kk(s)^{\times2}$.

    \begin{lemma}\label{lem:ClifOfIsotropic}
        Assume that there exists some $s \in \scrS$ with $\eps_s \notin \kk(s)^{\times 2}$ such that $\calQ_s$ has a smooth $\kk(s)$-point. Let $Q_s$ be a quadratic form whose vanishing definines $\calQ_s$. Then for any $\Gal(\kk(s))$-stable pair $\{x, x'\} \subset \calQ_s(\kbar)$ and any $\kk(s)$-linear form $\ell$ defining a hyperplane tangent to $\calQ_s$ at a smooth point with $\ell(x)\ell(x')\neq 0$ we have the following equality in $\Br(\kk(s))$:
            \[
            \Clif(Q_s) = \left(\eps_s, -\frac{B_{Q_s}(x,x')}{\ell(x)\ell(x')}\right)\,,
            \]
            where $B_{Q_s}$ denotes the bilinear form corresponding to $Q_s$.
    \end{lemma}
    \begin{proof}
        By~\cite{VAV}*{Lemma 2.1}, for any $\ell = \ell_0$ tangent to $\calQ_s$ at a smooth point, the quadric $\calQ_s$ is defined by the vanishing of 
            $Q_s = c(\ell_0\ell_1 - \ell_2^2 + \eps_s \ell_3^2),$
        for some linear forms $\ell_1, \ell_2, \ell_3$ and some $c \in \kk(s)^\times$. In particular, we have $\ell_0(x)\ell_1(x) = \ell_2(x)^2 - \eps_s\ell_3(x)^2$ and similarly for $x'$.  Thus, we may compute:
        \begin{align*}
            -\frac{B_{Q_s}(x,x')}{\ell(x)\ell(x')} &=
            -c \cdot\frac{\ell_0(x)\ell_1(x') +\ell_0(x')\ell_1(x) -2\ell_2(x)\ell_2(x') + 2\eps_s\ell_3(x)\ell_3(x')}{\ell_0(x)\ell_0(x')}\\
            &=  -c\left(\frac{\ell_2(x')^2 - \eps_s\ell_3(x')^2}{\ell_0(x')^2}+ \frac{\ell_2(x)^2 - \eps_s\ell_3(x)^2}{\ell_0(x)^2}-2\frac{\ell_2(x)\ell_2(x')}{\ell_0(x)\ell_0(x')} + 2\eps_s\frac{\ell_3(x)\ell_3(x')}{\ell_0(x)\ell_0(x')}\right) \\
            &=  -c\left[\left(\frac{\ell_2(x)}{\ell_0(x)} - \frac{\ell_2(x')}{\ell_0(x')}\right)^2 - \eps_s \left(\frac{\ell_3(x)}{\ell_0(x)} - \frac{\ell_3(x')}{\ell_0(x')}\right)^2\right],
        \end{align*}
        which shows that $\left(\eps_s, -\frac{B_{Q_s}(x,x')}{\ell(x)\ell(x')}\right) = (\eps_s, -c)$. Thus, it remains to relate the quaternion algebra $(\eps_s, -c)$ to the Clifford algebra of $Q_s$.  By~\cite{Lam-QF}*{Chap. V, Corollary 2.7},
        \[
            \Clif(Q_s) \simeq \Clif(Q_s|_{\langle \ell_0, \ell_1\rangle})\otimes \Clif(c^2\cdot Q_s|_{\langle \ell_2, \ell_3\rangle})\simeq \M_2(k) \otimes (-c, c\eps_s).
        \]
        To complete the proof, we observe that $(-c,c\eps_s) = (-c, \eps_s) = (\eps_s, -c)\in \Br(k)$.
    \end{proof}

	 \begin{defn}\label{def:ClifQT}
		Given $\scrT \subset \scrS$ such that $\N(\eps_\scrT) \in k^{\times 2}$, define 
		\[
			\CQ{\scrT} := \Cor_{\kk(\scrT)/k}(\Clif(Q_\scrT))\in \Br(k). 
		\] 
    \end{defn}
    \begin{remark}
    Even though $\Clif(Q_\scrT)$ may depend on the choice of quadratic form defining the pencil, the condition $\N(\eps_{\scrT}) \in k^{\times2}$ ensures that the class $\CQ{\scrT}$ does not. Indeed, if one computes $\CQ{\scrT}$ using instead a form $cQ_\scrT$ which differs from $Q_\scrT$ by $c \in k^\times$, Lemma~\ref{lem:cQ} shows that the result will differ by $\Cor_{\kk(\scrT)/k}(\eps_\scrT,c) = (\N(\eps_\scrT),c)$, which is trivial whenever $\N(\eps_\scrT)$ is a square.
    \end{remark}

    \begin{lemma}\label{lem:relBr}
        The kernel of the canonical map $\Br(k) \to \Br(X)$ is generated by $$\{ \CQ{s} \;:\; s \in \scrS \text{ such that } \eps_s \in \kk(s)^{\times 2} \}.$$
    \end{lemma}
    \begin{proof}
        By the exact sequence of low degree terms coming from the Hochschild-Serre spectral sequence \cite{CTS-Brauer}*{Prop. 4.3.2}, the kernel of $\Br (k) \to \Br (X)$ is the image of the cokernel $\Pic (X) \to \Pic (\Xbar)^{G_k}$.  By~\cite{VAV}*{Prop. 2.3} (which relies on results from~\cite{KST-dp4}), $\Pic(\Xbar)^{G_k}$ is freely generated by the hyperplane section and, for every $s\in \scrS$ such that $\eps_s\in \kk(s)^{\times2}$, the divisor class $\Norm_{\kk(s)/k}([C_s])$ where $C_s$ is obtained by intersecting $X$ with a plane contained in $\calQ_s$.  Since the hyperplane section is $k$-rational, the cokernel of $\Pic (X)\to \Pic (\Xbar)^{G_k}$ is  generated by
        \[
          \left\{\Norm_{\kk(s)/k}([C_s]) \;:\; s \in \scrS \text{ such that } \eps_s \in \kk(s)^{\times 2} \textup{ and } \calQ_s \textup{ contains no $k$-rational planes}  \right\}.
        \]
	By definition, the image of $[C_s]$ in $\Br(\kk(s))$ is the Severi-Brauer variety whose points parametrize representatives of the class $[C_s]$.  Since $\eps_s$ is a square, by~\cite{CTSko93}*{Thm. 2.5}, {\(\calQ_s\) is a cone over the surface \(Z_s\times Z_s\), where \(Z_s\) is a smooth conic obtained by intersecting \(\calQ_s\) with a general \(2\)-plane.
    Since planes in a fixed ruling on $\calQ_s$ correspond to fibers in a projection \(Z_s\times Z_s\to Z_s\), we deduce that $[C_s]\mapsto Z_s\in \Br(\kk(s))$.} By~\cite{EKM-QuadraticForms}*{Prop. 12.4} we also have that $\Clif(Q_s) = Z_s \in \Br(\kk(s))$. Hence, $\Norm_{\kk(s)/k}([C_s]) = \Cor_{\kk(s)/k}(\Clif(Q_s)) = \mathcal{C}_s$.
    \end{proof}
        
    %%%%%%%%%%%%%%%%%%%%%%%%%%%%%%%%%%%%%%%%%%%%%%%%%%%%%%%%%%%%%%%%%%%%%%%%%%%%
    \subsection{Local evaluation maps}\label{sec:LocalEv}
    %%%%%%%%%%%%%%%%%%%%%%%%%%%%%%%%%%%%%%%%%%%%%%%%%%%%%%%%%%%%%%%%%%%%%%%%%%%%

    \begin{lemma}\label{lem:SquareEpsGivesRatlPt}
        If there exists a degree $2$ subscheme $\scrT\subset \scrS$ such that for all $t\in \scrT$, $\eps_t\in \kk(t)^{\times2}$ and $\calQ_t$ has a smooth ${\kk(t)}$-point, then $X(k)\neq \emptyset$.
    \end{lemma}
    \begin{proof}
        Let $\scrT(\kbar) = \{t_1,t_2\}$.  The assumptions in the lemma imply that there are $k(t_i)$-rational planes contained in $\calQ_{t_i}$.  The intersection of one with $X$ gives a $k(t_i)$-rational conic $C_i$ on $X$.  {If \(t_1\notin \scrT(k)\), then we replace \(C_2\) with the conjugate of \(C_1\).  Thus,} the pair $\{C_1,C_2\}$ are Galois invariant.  As computed in \cite{VAV}*{Proof of Proposition 2.2} we have $C_{1}.C_{2} = 1$.  {(We note that our \(C_2\) may be either \(C_2\) or \(C_2'\) in the notation of~\cite{VAV}, but both have the same intersection number with \(C_1\).)}  Therefore the intersection of these divisors produces a $k$-point on $X$.
    \end{proof}   
   
    \begin{lemma}\label{lem:EvBetaSqrtEps}
    	Assume that $k$ is a local field of characteristic not equal to $2$ and let $\scrT\subset \scrS$ be a degree $2$ subscheme such that $\N(\eps_{\scrT})\in k^{\times2}$.  Then, for any quadratic extension $K/k$ with $\eps_\scrT\in \kk(\scrT_K)^{\times2}$ and $K\neq \kk(\scrT)$, there exists $y \in \calG(k)$ corresponding to a quadratic point $\Spec K \to X$.  Moreover, for such $y$,  
        \[
            \beta_\scrT(y) = 
            \begin{cases}
                \CQ{\scrT}& \textup{if }\eps_{\scrT}\notin{ \kk(\scrT)}^{\times2},\\
                0 &  \textup{if }\eps_{\scrT}\in {\kk(\scrT)}^{\times2}.
            \end{cases}
        \]
    \end{lemma}
    \begin{proof}
    If $X(k) \ne \emptyset$, then for any nontrivial extension $K/k$ we have $X(k) \subsetneq X(K)$ because $k$ is local (see, e.g., \cite{LiuLorenzini}*{Proposition 8.3}). Then any pair of $\Gal(K/k)$-conjugate points on $X$ will give the required $y \in \calG(k)$. Now we prove the first statement in the case where $X(k) =\emptyset$.  Over any local field, there is a unique rank $4$ quadric (up to isomorphism) that fails to have a point, and it has square discriminant. Furthermore, this anisotropic quadric has a point over any quadratic extension of $\kk(t)$.
    
    If $\eps_t\in\kk(t)^{\times2}$ for some $t\in \scrT$ (equivalently, for all $t \in \scrT$ by Corollary~\ref{cor:propsBrG}\eqref{GoodGeneratingSet}), then $\calQ_t$ may not have a smooth $\kk(t)$-point, but it will have a smooth point over any quadratic extension of $\kk(t)$.  If $K/k$ is a quadratic extension different from $\kk(\scrT)/k$, then $\kk(\scrT_K)$ will be a quadratic extension of $\kk(\scrT)$ and hence we may apply Lemma~\ref{lem:SquareEpsGivesRatlPt}.  {Moreover, since \(\eps_{\scrT}\in \kk(\scrT)^{\times2}\), by definition, \(\beta_{\scrT} = 0\in \Br(\calG)\).}
    
    Now we consider the case when $\eps_t\notin\kk(t)^{\times2}$, so \(Q_t\) has nonsquare discriminant, and thus is isotropic.  Hence, Lemma~\ref{lem:SquareEpsGivesRatlPt} gives the existence of $K$-points on $X$ for any $K$ such that $\eps_\scrT\in \kk(\scrT_K)^{\times2}$.

    Now suppose $y$ corresponds to the line joining the $K/k$-conjugate points $x,x' \in X(K)$, with $K$ satisfying the conditions of the lemma. By continuity of the evaluation map, we may reduce to the case where $\pi(y) \ne \infty$, $\pi(y) \not\in \scrT$.
    By Lemma~\ref{lem:piy} we have
        \begin{align*}
            \beta_\scrT(y) &= 
            \left(\eps_{\scrT}, -B_{Q_\scrT}(x,x')/B_{Q_\infty}(x,x')\right)\\
            &= \Cor_{\kk(\scrT)/k}
            \left(\eps_{\scrT}, -B_{Q_\scrT}(x,x')\right) 
            + 
            \left(\N_{\kk(\scrT)/k}(\eps_{\scrT}), {B_{Q_\infty}(x,x')}\right) \\
            &= \Cor_{\kk(\scrT)/k}
            \left(\eps_{\scrT}, -B_{Q_\scrT}(x,x')\right) &
            (\text{since } N(\eps_\scrT) \in k^{\times 2}\,)
            \\            
            &= \Cor_{\kk(\scrT)/k}\left[
            \left(\eps_{\scrT}, \ell_\scrT(x)\ell_\scrT(x')\right)+ \Clif(Q_\scrT)\right] 
            &\text{ (by Lemma~\ref{lem:ClifOfIsotropic})}\\
            & =  \Cor_{\kk(\scrT_K)/k} (\eps_{\scrT}, \ell_{\scrT}(x)) + \Cor_{\kk(\scrT)/k}\Clif(Q_\scrT)\\
            & = \Cor_{\kk(\scrT)/k}\Clif(Q_\scrT) &\text{(since } \eps_\scrT\in \kk(\scrT_K)^{\times2}\,)\\
            &= \CQ{\scrT} & \text{(by Definition~\ref{def:ClifQT}) }.& 
            \hfill \qedhere
        \end{align*}
    \end{proof}
    
 \begin{lemma}\label{lem:SolubleFiberOfG}
        Assume that $k$ is a local field of characteristic not equal to $2$. Suppose $s \in \scrS(k)$ is such that $\calQ_s$ has a smooth $k$-point and let $v_s$ denote the vertex of $\calQ_s$. For any $t\in \A^1(k) -\{s\}$ sufficiently close to $s$, we have
        \[
            \calG_{t}(k) \neq \emptyset \quad \Longleftrightarrow \quad (\eps_{s},t-s) = \Clif(Q_s) + (\eps_s,-Q_\infty(v_s)) \text{ in } \Br(k)\,.
        \]
    \end{lemma}
    \begin{rmk}
    	Note that by Lemma~\ref{lem:cQ}, the sum $\Clif(Q_s) + (\eps_s,-Q_\infty(v_s))$ appearing on the right-hand side above does not depend on the choice of quadratic form defining the pencil.
    \end{rmk}    
    \begin{proof}
 	Since $t \in \A^1(k) - \{s\}$ is sufficiently close to $s$ and $\scrS$ is closed, we have $t\notin \scrS$ and $Q_t$ has rank $5$. So by \cite{EKM-QuadraticForms}*{Ex. 85.4} the Severi-Brauer variety $\calG_t$ and the even Clifford algebra $\Clif_0(Q_t)$ (which is a central simple $k$-algebra) determine the same class in $\Br(k)$. In particular, $\calG_{t}(k)\neq \emptyset$ if and only if $\Clif_0(Q_{t}) = 0$ in $\Br(k)$.
 	
 	Since $X$ is smooth, $Q_t(v_s) \ne 0$. So the quadratic forms $Q_t$ and $Q_{t}|_{\langle v_s\rangle^{\perp}}{\oplus} Q_{t}|_{\langle v_s\rangle}$ are equivalent by \cite{Lam-QF}*{Chap. I, Cor 2.5}. Therefore,
 	\begin{align*}
 		\Clif_0(Q_t) &= \Clif(-Q_{t}(v_s)\cdot Q_{t}|_{\langle v_s\rangle^{\perp}}) & (\text{by~\cite{Lam-QF}*{Chap. V, Cor. 2.9}})\\
 		&= \Clif(Q_t|_{\langle v_s\rangle^{\perp}}) + (\disc(Q_t|_{\langle v_s\rangle^{\perp}}),-Q_t(v_s)) & (\text{by~Lemma~\ref{lem:cQ}}).
 	\end{align*}
 	For $t$ sufficiently close to $s$, the quadratic forms $Q_{t}|_{\langle v_s\rangle^{\perp}}$ and $Q_{s}|_{\langle v_s\rangle^{\perp}}$ will be equivalent. For such $t$, $\Clif(Q_t|_{\langle v_s\rangle^{\perp}}) = \Clif(Q_s) \in \Br(k)$ and $\disc(Q_t|_{\langle v_s\rangle^{\perp}}) = \disc(Q_s) \in k^\times/k^{\times 2}$. Hence
        \[
            \Clif_0(Q_t) = \Clif(Q_s) + (\eps_s,-Q_t(v_s))\,.
        \]
        To complete the proof, we note that $Q_t(v_s) {= (Q_s + (t-s)Q_{\infty})(v_s)} = (t-s)Q_{\infty}(v_s)$.
    \end{proof}

    \begin{lemma}\label{lem:betanonconstant}
    	Assume that $k$ is a local field of characteristic not equal to $2$ and $\scrT\subset\scrS$ is a degree $2$ subscheme with $\N(\eps_{\scrT})\in k^{\times2}$ and $\eps_\scrT\notin \kk(\scrT)^{\times2}$. Then there exists $y \in \calG_\scrT(\kk(\scrT))$ such that $\pi(y) = \scrT$. Moreover, for any such $y$,
    	\[
    		\Cor_{\kk(\scrT)/k}(\beta_\scrT(y)) = \CQ{\scrT} + \left(\eps, -\Delta_{\scrT}\N(Q_{\infty}(v_\scrT))\right)\in \Br(k)\,,
    	\]
    	where $\eps \in k^\times$ is an element whose image in $\kk(\scrT)^\times/\kk(\scrT)^{\times 2}$ represents $\eps_\scrT$, $\Delta_{\scrT}$ is the discriminant of $\kk(\scrT)/k$ (which we take to be $1$ if $\scrT$ is reducible), and $v_\scrT$ is the vertex of $\calQ_\scrT$.
    \end{lemma}

    \begin{proof}
    	By \cite{VAV}*{Lemma 3.1} there exists $\eps\in k^\times$ such that $\eps\cdot \eps_{t}\in \kk(t)^{\times2}$ for all $t\in \scrT$.  Fix a closed point $s \in \scrT$, and let $s'$ be the unique $\kk(s)$ point in $\scrT_{\kk(s)} - \{s\}$. 
        
        Since $\eps_s \not\in \kk(s)^{\times 2}$, $\calQ_s$ is a cone over an isotropic quadric and as such contains smooth $\kk(s)$-points and $\kk(s)$-rational lines (passing through the vertex). Hence $\calG_s(\kk(s))$ is nonempty. By the implicit function theorem, we can find $t \in (\PP^1 - \{s\})(\kk(s))$ arbitrarily close to $s$ such that $\calG_t(\kk(s)) \ne \emptyset$. In addition, by Lemma~\ref{lem:SolubleFiberOfG} and the fact that the evaluation map $\beta \colon \calG(\kk(s)) \to \Br(\kk(s))$ is locally constant and constant on the fibers of $\pi :\calG \to \PP^1$ (because $\beta_\scrT$ is the pullback of a element of $\Br(\kk(\PP^1))$), we may choose such a $t$ sufficiently close to $s$ so that
    	\begin{enumerate}
    		\item $(\eps,t-s) = \Clif(Q_s) + (\eps,-Q_\infty(v_s)) \in \Br (\kk(s))$,
    		\item\label{betanonconstant2} $\beta_\scrT(\calG_s(\kk(s))) = \beta_\scrT(\calG_t(\kk(s))) \in \Br(\kk(s))$, and
    		\item $t-s'$ and $s-s'$ represent the same class in $\kk(s)^{\times 2}$.
    	\end{enumerate}

    	Then for $y_1 \in \calG_s(\kk(s))$ and $y_1' \in \calG_t(\kk(s))$, we have
    	\[
    		\beta_\scrT(y_1) = \beta_\scrT(y_1') = (\eps,(t-s)(t-s')) = \Clif(Q_s) + (\eps,-Q_\infty(v_s)(s-s'))\,{\in \Br(\kk(s))}.
    	\]
    	Suppose $y\colon \Spec(\kk(\scrT)) \to \calG$ is such that $\pi(y) = \scrT$. Then, because $\eps \in k^\times$ we have $\Cor_{\kk(\scrT)/k}(\eps,s-s') = (\eps,\Norm_{\kk(\scrT)/k}(s-s')) = (\eps,(s-s')(s'-s))$. It follows that
    	\begin{align*}
    		\Cor_{\kk(\scrT)/k}(\beta_\scrT(y)) &= \Cor_{\kk(\scrT)/k}\left[ \Clif(Q_{\scrT}) + (\eps, -Q_{\infty}(v_\scrT))+(\eps,s-s')\right]\\
    		&= \Cor_{\kk(\scrT)/k}\left( \Clif(Q_{\scrT})\right)  + (\eps,\N(Q_\infty(v_\scrT))) + (\eps,(s-s')(s' - s))\\
    		&= \CQ{\scrT} + \left(\eps, \N(Q_{\infty}(v_\scrT))\right) + \left(\eps, -\Delta_{\scrT}\right)\,. \hfill \qedhere
    	\end{align*}
    \end{proof}

%%%%%%%%%%%%%%%%%%%%%%%%%%%%%%%%%%%%%%%%%%%%%%%%%%%%%%%%%%%%%%%%%%%%%%%%%%%%%%%%
\subsection{Evaluation of Brauer classes on \texorpdfstring{$\calG(\A_k)$}{G(Ak)}}\label{sec:AdelicEv}
%%%%%%%%%%%%%%%%%%%%%%%%%%%%%%%%%%%%%%%%%%%%%%%%%%%%%%%%%%%%%%%%%%%%%%%%%%%%%%%%

\begin{defn}\label{def:RT}
    Let $k$ be a global field of characteristic not equal to $2$.  Given $\scrT\subset \scrS$ define
    \[
        R_\scrT := \{v\in\Omega_k : \eps_{\scrT_{v}} \in \kk(\scrS_{v})^{\times 2} \text{ and }  \CQ{\scrT_v} \neq 0 \}.
    \]
\end{defn}

\begin{theorem}\label{thm:EvAdelicOnG}
    Assume that $k$ is a global field of characteristic different from $2$.  
    \begin{enumerate}
        \item\label{case:irredT} There exists $(y_v)\in \calG(\A_k)$ such that for all degree $2$ subschemes $\scrT\subset \scrS$ with $\N(\eps_\scrT) \in k^{\times2}$, we have $\sum_{v\in \Omega_k} \inv_v(\beta_\scrT(y_v)) = \frac{\#{R}_\scrT}{2} \in \Q/\Z$.\label{case:ParityObs}
        \item For all $t \in \scrS(k)$ there exists $(y_v)\in \calG(\A_k)$ such that for all degree $2$ subschemes $\scrT\subset \scrS$ with $\N(\eps_\scrT) \in k^{\times2}$ and $t \in \scrT$,
        we have $\sum_{v\in \Omega_k} \inv_v(\beta_\scrT(y_v)) = \frac{\#R_t}{2} \in \Q/\Z$.\label{case:RationalRank4}
    \end{enumerate}
\end{theorem}
\begin{proof}
    \begin{enumerate}[wide]
    \item It suffices to prove the result for $\{ \beta_\scrT \;:\; \scrT \in \mathbb{T} \}$, where $\mathbb{T}$ is a collection of degree $2$ subschemes of $\scrS$ as in Corollary~~\ref{cor:propsBrG}\eqref{GoodGeneratingSet}, with corresponding $\eps \in k^\times$ simultaneously representing the discriminants of all $\scrT \in \mathbb{T}$.
    
    We define an adelic point $(y_v) \in \calG(\A_k)$ as follows. For $v \in \Omega_k$ such that $\eps\in \kk(\scrT_v)^{\times2}$ for some $\scrT \in \mathbb{T}$ (equivalently, for all $\scrT \in \mathbb{T}$ by Corollary~\ref{cor:propsBrG}\eqref{GoodGeneratingSet}), let $y_v \in \calG(k_v)$ be any point (which exists by Corollary~\ref{cor:locptsG}). Note that if $\eps \in \kk(\scrT_v)^{\times 2}$, then $\beta_\scrT \otimes k_v = 0$ by Proposition~\ref{prop:BrFaddeev}.
    For each $v \in \Omega_k$ with $\eps\notin \kk(\scrT_v)^{\times2}$ for some (equivalently all) $\scrT \in \mathbb{T}$, let $y_v\in \calG(k_v)$ be a point corresponding to a $k_v(\sqrt{\eps})$-point on $X$, as provided by Lemma~\ref{lem:EvBetaSqrtEps}.  Note that Lemma~\ref{lem:EvBetaSqrtEps} further implies that for such $y_v$, $\beta_\scrT(y_v) = \CQ{\scrT_v}$ for all $\scrT\in \mathbb{T}$. 
    Thus, for any $\scrT \in \mathbb{T}$ we have
\[
	\sum_{v\in \Omega_k} \inv_v(\beta_\scrT(y_v)) = \sum_{\eps \notin \kk(\scrT_v)^{\times 2}} \inv_v \left(\CQ{\scrT_v}\right) = \sum_{\eps \in \kk(\scrT_v)^{\times 2}} \inv_v \left(\CQ{\scrT_v}\right)=
    \frac{\#{R}_\scrT}{2} \in \Q/\Z\,,
\]
where the penultimate equality follows from quadratic reciprocity.

    \item Let $t\in \scrS(k)$ and set $\eps:=\eps_t$.  If $t$ is not contained in any degree $2$ subschemes $\scrT\subset \scrS$ with $\N(\eps_{\scrT}) \in k^{\times2}$, then we need only show that $\calG(\A_k)\neq\emptyset$, which follows from Corollary~\ref{cor:locptsG}. Thus, we may assume there is some degree $2$ subscheme $\scrT\subset\scrS$ containing $t$ such that $\N(\eps_{\scrT}) \in k^{\times2}$. For any such $\scrT$ we have $\eps_{\scrT} = (\eps,\eps) \in \kk(\scrT)^{\times}/\kk(\scrT)^{\times 2}\simeq k^\times/k^{\times2} \times k^{\times}/k^{\times2}$.  
    
    We define an adelic point $(y_v) \in \calG(\A_k)$ as follows. For $v \in \Omega_k$ such that $\eps \in k_v^{\times 2}$, take $y_v$ to be any point of $\calG(k_v)$ (which exists by Corollary~\ref{cor:locptsG}). For $v \in \Omega_k$ such that $\eps \notin k_v^{\times 2}$ we take $y_v \in \calG(k_v)$ to be any point such that $\pi(y_v) \in \PP^1(k_v)$ is close enough $t$ so that Lemma~\ref{lem:SolubleFiberOfG} applies (note that $\calQ_t$ is a cone over an isotropic quadric surface so the hypothesis of the Lemma~\ref{lem:SolubleFiberOfG} is satisfied) and so that, for all $s \in \scrS(k)-\{t\}$ with $\eps\eps_{s} \in k^{\times 2}$, ($\pi(y_v) - s)$ and $(t-s)$ have the same class in $k_v^{\times}/k_v^{\times 2}$.

Suppose $\scrT = \{s,t\}\subset \scrS(k)$ is such that $\N(\eps_\scrT)\in k^{\times2}$. For $v \in \Omega_k$ such that $\eps \in k_v^{\times 2}$, we have $\inv_v(\beta_\scrT(y_v)) = 0$. For $v \in \Omega_k$ such that $\eps \notin k_v^{\times 2}$ we have 
    \begin{align*}
        \inv_v(\beta_\scrT(y_v)) &= \inv_v\left(\eps,(\pi(y_v)-t)(\pi(y_v)-s)\right)\\
        &= \inv_v\left(\eps,\pi(y_v)-t)\right) + \inv_v\left(\eps,t-s\right)\\
        &= \inv_v(\Clif(Q_{t})) + \inv_v(\eps,-Q_\infty(v_t)) + \inv_v(\eps,t - s) \qquad \text{(By Lemma~\ref{lem:SolubleFiberOfG})}\,.
    \end{align*}
    Since $(\eps,-Q_\infty(v_t)(t-s))$ is an element of $\Br(k)$, its local invariants sum to $0$. Furthermore, $(\eps,-Q_\infty(v_t)(t-s))$ has trivial invariant at all $v \in \Omega_k$ where $\eps\in k_v^{\times2}$.  Thus, 
    \[
        \sum_{v \in \Omega_k}\inv_v(\beta_\scrT(y_v)) = 
        \sum_{\eps \notin k_v^{\times 2}}\inv_v(\beta_\scrT(y_v)) = 
        \sum_{\eps \notin k_v^{\times 2}} \inv_v(\Clif(Q_{t})) = \sum_{\eps \in k_v^{\times 2}} \inv_v(\Clif(Q_{t})).
    \]
    where the last equality follows from the fact that the local invariants of $\Clif(Q_t) \in \Br(k)$ sum to $0$. For $v \in \Omega_k$ such that $\eps_t \in k_v^{\times 2}$ we have $\inv_v(\Clif(Q_t)) = \inv_v(\CQ{t_v})$. Hence,
    \[
    	\sum_{v \in \Omega_k}\inv_v(\beta_\scrT(y_v)) =\sum_{\eps \in k_v^{\times 2}} \inv_v({\CQ{t_v}}) 
         = \frac{\#R_{t}}{2}\,, \in{\Q}/{\Z}\,. \qedhere
    \]
    \end{enumerate}
\end{proof}

The following lemma relates the set $R_{\scrT}$ to the condition given in~\eqref{case-global-n4} of Theorem~\ref{thm:MainThm}.
\begin{lemma}\label{lem:RTandSolubility}
        Let $k$ be a global field of characteristic not equal to $2$ and $\scrT \subset \scrS$ a degree $2$ subscheme such that $\N(\eps_\scrT) \in k^{\times 2}$. Then $v \in R_\scrT$ if and only if there are an odd number of components of $\calQ_{\scrT_v} = \cup_{t_v \in \scrT_v}\calQ_{t_v}$ which have no smooth $\kk(t_v)$-point.
\end{lemma}
\begin{proof}
    Let $v \in \Omega_k$. First suppose that $\eps_{\scrT_v} \not\in \kk(\scrS_v)^{\times 2}$. Then $v \not\in R_\scrT$ by definition. Note also that for all $t_v \in \scrT_v$, $\eps_{t_v} \not\in \kk(t_v)^{\times 2}$ (a priori this must hold for some $t_v \in \scrT_v$; the stronger conclusion holds because $\scrT$ has degree $2$ and $\N(\eps_\scrT) \in k^{\times 2}$). Recall that there is a unique anisotropic quadratic form of rank $4$ over any local field and that it has square discriminant. Hence, when $\eps_{\scrT_v} \not\in \kk(t_v)^{\times 2}$, all components $\calQ_{t_v}$ have smooth $\kk(t_v)$-points.
    
    Now suppose that $\eps_{\scrT_v} \in \kk(\scrS_v)^{\times 2}$. As above $\eps_{t_v} \in \kk(t_v)^{\times 2}$, for each $t_v \in \scrT_v$. Then the rank $4$ quadratic forms $Q_{t_v}$ are equivalent to constant multiples of the norm forms of the quaternion algebras $\Clif(Q_{t_v})$ (see \cite{EKM-QuadraticForms}*{Prop. 12.4}). In particular, $\calQ_{t_v}$ has a smooth $\kk(t_v)$-point if and only if $\Clif(Q_{t_v}) = 0 \in \Br(\kk(t_v))$. The corestriction maps $\Cor_{\kk(t_v)/k_v} : \Br(\kk(t_v)) \to \Br(k_v)$ are isomorphisms, so $\CQ{\scrT_v} = \sum_{t_v \in \scrT_v}\Cor_{\kk(t_v)/k_v}\Clif(Q_{t_v})$ is nonzero if and only if there are an odd number of components of $\calQ_{\scrT_v}$ with no smooth $\kk(t_v)$-points. By definition $v \in R_\scrT$ if and only if $\CQ{\scrT_v} \ne 0$.
\end{proof}

\begin{lemma}\label{lem:BaseChangeR_T}
    Assume that $k$ is a global field of characteristic different from $2$ and suppose $\scrT\subset\scrS$ is irreducible of degree $2$ such that $\N(\eps_\scrT) \in k^{\times 2}$.  For any $t\in \scrT(\kk(\scrT))$, the cardinalities of the sets
     \[
       R_{\scrT}\subset \Omega_k  \quad\textup{and}\quad R_t\subset \Omega_{\kk(\scrT)}
     \]
     have the same parity.
 \end{lemma}
 \begin{proof}
    For a prime $v \in \Omega_k$, we have $\eps_\scrT \in \kk(\scrT_v)^{\times 2}$ if and only if $\eps_t \in \kk(t)_w^{\times 2}$ for all (equivalently some) $w \in \Omega_{\kk(\scrT)}$ with $w \mid v$. For such $v$ we have 
    $$\inv_v(\CQ{\scrT_v}) = \inv_v(\Cor_{\kk(\scrT)/k}(\Clif(Q_{\scrT}))) = \sum_{w \mid v}\inv_w(\Clif(Q_{t})) = \sum_{w \mid v}\inv_w(\calC_{t_w}).$$
    In particular, $\calC_{\scrT_v} \ne 0$ if and only if there are an odd number of primes $w \mid v$ with $\calC_{t_w}\ne 0$.
\end{proof}

%%%%%%%%%%%%%%%%%%%%%%%%%%%%%%%%%%%%%%%%%%%%%%%%%%%%%%%%%%%%%%%%%%%%%%%%%%%%%%%%
%%%%%%%%%%%%%%%%%%%%%%%%%%%%%%%%%%%%%%%%%%%%%%%%%%%%%%%%%%%%%%%%%%%%%%%%%%%%%%%%
\section{Proofs of the Main Theorems}\label{sec:MainProofs}
%%%%%%%%%%%%%%%%%%%%%%%%%%%%%%%%%%%%%%%%%%%%%%%%%%%%%%%%%%%%%%%%%%%%%%%%%%%%%%%%
%%%%%%%%%%%%%%%%%%%%%%%%%%%%%%%%%%%%%%%%%%%%%%%%%%%%%%%%%%%%%%%%%%%%%%%%%%%%%%%%

    %%%%%%%%%%%%%%%%%%%%%%%%%%%%%%%%%%%%%%%%%%%%%%%%%%%%%%%%%%%%%%%%%%%%%%%%%%%%
    %%%%%%%%%%%%%%%%%%%%%%%%%%%%%%%%%%%%%%%%%%%%%%%%%%%%%%%%%%%%%%%%%%%%%%%%%%%%
    \subsection{Corollaries of Theorem~\ref{thm:EvAdelicOnG}}
    %%%%%%%%%%%%%%%%%%%%%%%%%%%%%%%%%%%%%%%%%%%%%%%%%%%%%%%%%%%%%%%%%%%%%%%%%%%%
    %%%%%%%%%%%%%%%%%%%%%%%%%%%%%%%%%%%%%%%%%%%%%%%%%%%%%%%%%%%%%%%%%%%%%%%%%%%%

\begin{cor}\label{cor:NonemptyAdelicBrSetG}
    Assume that $k$ is a global field of characteristic not equal to $2$ and that either of the following conditions hold:
    \begin{enumerate}
        \item\label{cor:part1} Every nontrivial element of $\Br(\calG)/\Br_0(\calG)$ can be represented by $\beta_{\scrT}$ for some degree $2$ subscheme $\scrT\subset \scrS$ such that $\N(\eps_\scrT)\in k^{\times2}$ and $\#{R}_\scrT$ even; or\label{cond:ParityObsVanishes}
        \item\label{cor:part2} Every nontrivial element of $\Br(\calG)/\Br_0(\calG)$ can be represented by $\beta_{\scrT}$ for some degree $2$ subscheme $\scrT{=\{t_1,t_2\}}\subset \scrS(k)$ such that $\N(\eps_\scrT)\in k^{\times2}$.\label{cond:reducibleBr}
    \end{enumerate}
    Then $\calG(\A_k)^{\Br}\neq \emptyset$.
\end{cor}

\begin{proof}
    If condition~\eqref{cond:ParityObsVanishes} holds, then the corollary follows from Theorem~\ref{thm:EvAdelicOnG}\eqref{case:ParityObs}.  Now assume condition~\eqref{cond:reducibleBr} holds and~\eqref{cond:ParityObsVanishes} fails. Then there exists a nontrivial element of $\Br(\calG)$ of the form $\beta_{\{t,t'\}}$ with $t,t'\in \scrS(k)$ such that {$R_{\{t,t'\}}$ has odd cardinality. Note that $R_{\{t,t'\}}$ is the symmetric difference of $R_t$ and $R_{t'}$. }Thus, interchanging $t$ and $t'$ if needed we may assume $R_t$ has even cardinality. Theorem~\ref{thm:EvAdelicOnG}\eqref{case:RationalRank4} then gives an adelic point orthogonal to all $\beta_{\scrT}$ such that $\scrT$ has degree $2$, contains $t$ and $\N(\eps_\scrT) \in k^{\times 2}$. The result follows since Corollary~\ref{cor:propsBrG}\eqref{statement:SingleRatlRank4} shows that such $\beta_\scrT$ generate $\Br(\calG)/\Br_0(\calG)$.
\end{proof}

\begin{remark}\label{rmk:RemainingOpenCases}
    If both conditions of Corollary~\ref{cor:NonemptyAdelicBrSetG} fail, then $\Br(\calG)/\Br_0(\calG) \simeq \Z/2\Z$ by Corollary~\ref{cor:propsBrG} and any $\beta_\scrT$ with $\scrT$ of degree $2$ which represents the nontrivial class must have $\scrT$ irreducible. Thus, $\scrS$ must contain an irreducible degree $2$ subscheme $\scrT$ such that
    \begin{itemize}
        \item $\N(\eps_\scrT) \in k^{\times2}$,
        \item $\eps_{\scrT}\not \in \kk(\scrT)^{\times2}$,
        \item if $\#(\scrS-\scrT)(k) = 3$, then $\eps_t\notin k^{\times2}$ for all $t\in \scrS - \scrT$, and
        \item $\#{R}_{\scrT}$ is odd, which in particular implies that $\calQ_{\scrT}$ has no smooth $\kk(\scrT)$-points.
    \end{itemize}
\end{remark}

\begin{cor}\label{cor:ObsToWeakApproximation}
Assume that $k$ is a global field of characteristic not equal to $2$. Suppose there is a degree $2$ subscheme $\scrT\subset \scrS$ with $\N(\eps_{\scrT})\in k^{\times2}$ such that $R_\scrT$ has odd cardinality. Then $\calG(\A_k)^{\Br} \ne \calG(\A_k)$ and there exists a quadratic extension $K/k$ such that $X_K(\A_K) \ne X_K(\A_K)^{\Br} = \emptyset$. In particular, $\calG$ does not satisfy weak approximation and there exists quadratic extension $K/k$ such that $X_K$ has a Brauer-Manin obstruction to the Hasse principle.
\end{cor}

\begin{proof}
	The first statement follows immediately from Theorem~\ref{thm:EvAdelicOnG}\eqref{case:irredT}. For the second statement we construct $K$ by approximating fixed quadratic extensions of $k_v$ for the primes $v \in S := \{v : X(k_v)=\emptyset \textup{ or } \inv_v\circ\beta_{\scrT}\colon \calG(k_v) \to \Q/\Z \textup{ is nonzero}\}$.  (In particular, by Lemma~\ref{lem:EvBetaSqrtEps} and the definition of $\CQ{\scrT}$, we will approximate $K$ at every prime where $\CQ{\scrT}$ ramifies.)  For such $v$, if $\eps\notin k_v^{\times2}$, then we fix $K_v := k_v(\sqrt{\eps})$.  If $v$ is such that $\eps\in k_v^{\times2}$, then we let $K_v$ be any quadratic extension such that $X(K_v)\neq \emptyset$.  Then  Lemma~\ref{lem:EvBetaSqrtEps} implies that for every $v\in S$, there exists a $y_v\in \calG(k_v)$ corresponding to a quadratic point $\Spec K_v\to X$ and for all such $y_v$, $\beta_{\scrT}(y_v) = \CQ{\scrT_v}$ if $\eps\not\in k_v^{\times2}$ and $\beta_{\scrT}(y_v) =0$ otherwise.  Furthermore, for $v\notin S$ (which necessarily means that $\CQ{\scrT_v} = 0$), our assumptions imply that $X(K_v)\neq \emptyset$ and that $\beta_{\scrT}(y_v) = 0$ for all $y_v\in\calG(k_v)$.  Thus, for all $(y_v)\in \calG(\A_k)$ corresponding to an adelic quadratic point $\Spec(\A_K) \to X$ we have
   \[
     \sum_v \inv_v\beta_{\scrT}(y_v) = \sum_{\eps\notin k_v^{\times2}} \inv_v \beta_{\scrT}(y_v) =  \sum_{\eps\notin k_v^{\times2}} \inv_v \CQ{\scrT_v} = \sum_{\eps\in k_v^{\times2}} \inv_v \CQ{\scrT_v} = \frac{\# R_{\scrT}}2\in \Q/\Z.
   \]
   By Proposition~\ref{prop:sym2}\eqref{it:p2} and Corollary~\ref{cor:propsBrG}\eqref{BB07}, this implies that $X_K(\A_K)^{\Br} = \emptyset$.
\end{proof}

\begin{example}\label{ex:ObsToWeakApproximation}
	Let $\calG \to \PP^1$ be the fibration of Severi-Brauer threefolds corresponding to the pencil containing the quadrics given by the vanishing of the rank $4$ forms
   \begin{align*}
   	Q_0&= x_0x_1 - x_2^2 + \eps x_3^2\,,\, \text{and}\\
   	Q_1&= ax_0^2 + bx_1^2 -abx_2^2 - \eps x_4^2\,
   \end{align*}
   where $a,b,\eps \in k^\times$. Then $\scrT = \{ 0,1 \} \subset \scrS$ is a degree $2$ subscheme with $\eps_\scrT = (\eps,\eps)$. Hence, Corollaries~\ref{cor:propsBrG} and~\ref{cor:NonemptyAdelicBrSetG} imply that $\calG(\A_k)^{\Br}\neq\emptyset$. Note that $\calQ_0$ has smooth $k$-points, so $R_\scrT = R_1 = \{ v \in \Omega_k \;:\; \eps \in k_v^{\times 2} \text{ and } \inv_v(a,b) \ne 0 \}$. Clearly one can choose $a,b,\eps$ so that $R_\scrT$ has odd cardinality (e.g., for $k = \Q$, $a = 3, b = 7, \eps = 2$ we have $R_\scrT = \{ 7 \}$), in which case $\calG$ has a Brauer-Manin obstruction to weak approximation and the base locus $X$ of the pencil is a counterexample to the Hasse principle over some quadratic extension by Corollary~\ref{cor:ObsToWeakApproximation}. 
   
   If $4-ab \in k_v^{\times 2} - k^{\times 2}$ for some prime $v \in R_\scrT$ (which holds for the values indicated above), then there exists no quadratic extension $K/k$ such that $X_K$ is everywhere locally solvable and $\Br(X_K) = \Br_0(X_K)$. To see this first observe that $4-ab = \eps_t$ is the discriminant of the rank $4$ quadric $Q_{t} = \frac{1}{1-ab}(Q_1-abQ_0)$ (here $t = 1/(1-ab) \in \scrS(k)$). Now note that if a prime $v \in R_\scrT$ splits in a quadratic extension $K$, then $X_K$ is not locally solvable because $\calQ_1$ has no smooth $K_w$-points for the primes $w \mid v$. On the other hand, Proposition~\ref{prop:BrFaddeev} shows that $\beta_\scrT \otimes K \in \Br(X_K)$ lies in the subgroup $\Br_0(X_K)$ if and only if $\eps \in K^{\times 2}$ (in which case $K = k(\sqrt{\eps})$ and all primes of $R_\scrT$ split in $K$) or $\eps_{\scrS-\scrT} \in \kk(\scrS_K)^{\times 2}$ (in which case $K= k(\sqrt{4-ab})$ and so some prime of $R_\scrT$ splits in $K$ by assumption). We conclude that if $K/k$ is a quadratic extension such that $\beta_\scrT \otimes K \in \Br_0(X_K)$, then $X_K(\A_K)= \emptyset$.
\end{example}

\begin{cor}\label{cor:ExistenceOf0cycleOfDegree1}
    Assume that $k$ is a global field of characteristic not equal to $2$.  There is an adelic $0$-cycle of degree $1$ on $\calG$ orthogonal to $\Br(\calG)$.
\end{cor}

\begin{proof}
   We may assume that $\calG(\A_k)^{\Br} = \emptyset$ (for otherwise the Corollary holds immediately) and hence, that the hypothesis of Corollary~\ref{cor:NonemptyAdelicBrSetG} fails.  As explained in Remark~\ref{rmk:RemainingOpenCases}, this implies that there is an irreducible degree $2$ subscheme $\scrT\subset\scrS$ such that $\N(\eps_\scrT)\in k^{\times2}$, $\eps_{\scrT}\notin \kk(\scrT)^{\times2}$ and $R_{\scrT}$ has odd cardinality. By Corollary~\ref{cor:propsBrG}\eqref{BB04}, the existence of such an irreducible $\scrT$ implies that $\Br(\calG)/\Br_0(\calG)$ has order $2$. Moreover, if $t \in \scrT(\kk(\scrT))$ then, by Lemma~\ref{lem:BaseChangeR_T}, the set $R_t \subset \Omega_{\kk(\scrT)}$ has odd cardinality. Thus, by Theorem~\ref{thm:EvAdelicOnG} applied over $\kk(\scrT)$ we obtain an effective adelic $0$-cycle of degree $2$ over $k$, denote it by $(z_v)$, such that $\sum_{v \in \Omega_k}\inv_v(\beta_\scrT(z_v)) = 1/2$. If $(y_v) \in \calG(\A_k)$ is any adelic point (which exists by Corollary~\ref{cor:locptsG}), then $(z_v-y_v)$ is an adelic $0$-cycle of degree $1$ and, since $\calG(\A_k)^{\Br} = \calG(\A_k)^{\beta_\scrT} = \emptyset$, we have $\sum_{v \in \Omega_k}\inv_v(\beta_\scrT(z_v-y_v)) = \sum_{v \in \Omega_k}\inv_v(\beta_\scrT(z_v)) - \sum_{v \in \Omega_k}\inv_v(\beta_\scrT(y_v)) = 1/2 - 1/2 = 0$.
\end{proof}

\begin{remark}
	In the cases not already covered by Corollary~\ref{cor:NonemptyAdelicBrSetG}, the proof above hinges on constructing an adelic $0$-cycle of degree $2$ on $\mathcal{G}$ that is not orthogonal to the Brauer group. Lemma~\ref{lem:betanonconstant} can be used to give an alternative construction of such a $0$-cycle. See Section~\ref{sec:opencases}.
\end{remark}

%%%%%%%%%%%%%%%%%%%%%%%%%%%%%%%%%%%%%%%%%%%%%%%%%%%%%%%%%%%%%%%%%%%%%%%%%%%%%%%%
\subsection{Proof of Theorem~\ref{thm:MainIndThm}}\label{sec:ProofOfMainIndThm}
Let $X' \subset \PP^n_k$ be a smooth complete intersection of two quadrics over $k$. By Bertini's theorem the intersection of $X'$ with a suitable linear subspace will yield a smooth complete intersection of two quadrics $X \subset \PP^4_k$. If $k$ is a local field, then the result follows from Theorem~\ref{thm:localquadpts}. It remains to consider the case that $k$ is a number field. By Corollary~\ref{cor:ExistenceOf0cycleOfDegree1}, $\calG$ has an adelic $0$-cycle of degree $1$ orthogonal to the Brauer group. Since $\calG$ is a pencil of Severi-Brauer varieties,~\cite{CTSD}*{Theorem 5.1} shows that $\calG$ has a $0$-cycle of degree $1$. By Proposition~\ref{prop:ObjectConnections} this gives a $0$-cycle of degree $2$ on $X$.

%%%%%%%%%%%%%%%%%%%%%%%%%%%%%%%%%%%%%%%%%%%%%%%%%%%%%%%%%%%%%%%%%%%%%%%%%%%%%%%%
\subsection{Proof of Theorem~\ref{thm:MainThm}}

Let $X'\subset \PP^n_k$ be a smooth complete intersection of two quadrics over $k$. As noted above, $X'$ contains a smooth of two quadrics in $\PP^4_k$. Thus, Theorem~\ref{thm:localquadpts} implies that $X'$ contains a quadratic point when $k$ is local. This proves Theorem~\ref{thm:MainThm}\eqref{case-local}. Similarly, Theorem~\ref{thm:MainThm}\eqref{case-globalchar2} follows from Proposition~\ref{prop:char2}.

Now assume $k$ is global of characteristic not equal to $2$. Theorem~\ref{thm:localquadpts} implies that there is a quadratic extension $K/k$ such that $X'_K$ is everywhere locally solvable. If $k$ is a global function field and $n \ge 5$, then $X'_K$ satisfies the Hasse principle by \cite{Tian}. This proves Theorem~\ref{thm:MainThm}\eqref{case-globalfunction-n5}.

We claim that (under the hypotheses of the theorem) $X'$ contains a smooth del Pezzo surface $X$ of degree $4$ such that the corresponding Severi-Brauer pencil $\calG$ has $\calG(\A_k)^{\Br}\ne \emptyset$. If $n \ge 5$, then by \cite{Wittenberg}*{Section 3.5} the intersection of $X'$ with an appropriate linear subspace is a smooth del Pezzo surface $X$ of degree $4$ with $\Br(X) = \Br_0(X)$. Corollary~\ref{cor:propsBrG}\eqref{statement:BrGIsomBrX} implies that the corresponding $\calG$ has $\Br(\calG)/\Br_0(\calG) = 0$, so $\calG(\A_k)^{\Br}\ne \emptyset$ by Corollary~\ref{cor:NonemptyAdelicBrSetG}. When $n = 4$, $X'= X$ is itself a smooth del Pezzo surface of degree $4$. {If all of the nontrivial elements of $\Br(\calG)/\Br_0(\calG)$ can be represented by some $\beta_\scrT$ with $\scrT$ reducible, then Corollary~\ref{cor:NonemptyAdelicBrSetG}\eqref{cor:part1} implies that $\calG(\A_k)^{\Br} \ne \emptyset$. Otherwise, by Corollary~\ref{cor:propsBrG}, the order of $\Br(\calG)/\Br_0(\calG)$ divides $2$ and any nontrivial element can be represented by $\beta_\scrT$ with $\scrT$ irreducible and $N(\eps_\scrT) \in k^{\times 2}$. Any such element determines a quadratic extension $L = \kk(\scrT)$ and the assumption in case~\eqref{case-global-n4} of the theorem is that the geometric components of $\calQ_\scrT$ (which are defined over $L$) each fail to have smooth local points at an even number of primes of $L$. By Lemma~\ref{lem:RTandSolubility} this implies that $R_\scrT$ has even cardinality and so $\calG(\A_k)^{\Br} \ne \emptyset$ by Corollary~\ref{cor:NonemptyAdelicBrSetG}\eqref{cor:part2}.}

If $k$ is a number field for which Schinzel's hypothesis holds, then it is a result of Serre that the Brauer-Manin obstruction is the only obstruction to the Hasse principle for fibrations of Severi-Brauer varieties (Serre's result is unpublished, but a more general result~\cite{CTSD}*{Theorem 4.2} implies this result of Serre). In this case we obtain a $k$-point on $\calG$ and, consequently by Proposition~\ref{prop:ObjectConnections}, a quadratic point on $X$. To prove the result {assuming $k$ satisfies $(\star)$} it is enough to find a quadratic extension $K/k$ such that $X_K(\A_K)^{\Br} \ne \emptyset$. The existence of such a $K$ follows from~Proposition~\ref{prop:sym2}\eqref{it:p4}, since as noted in Corollary~\ref{cor:propsBrG}\eqref{statement:BrGIsomBrX} the map $\Br(X)/\Br_0(X) \to \Br(\calG)/\Br_0(\calG)$ given by~Proposition~\ref{prop:sym2}\eqref{it:p1} is an isomorphism.

%%%%%%%%%%%%%%%%%%%%%%%%%%%%%%%%%%%%%%%%%%%%%%%%%%%%%%%%%%%%%%%%%%%%%%%%%%%%%%%%
%%%%%%%%%%%%%%%%%%%%%%%%%%%%%%%%%%%%%%%%%%%%%%%%%%%%%%%%%%%%%%%%%%%%%%%%%%%%%%%%
\section{Complements and Remarks}
%%%%%%%%%%%%%%%%%%%%%%%%%%%%%%%%%%%%%%%%%%%%%%%%%%%%%%%%%%%%%%%%%%%%%%%%%%%%%%%%
%%%%%%%%%%%%%%%%%%%%%%%%%%%%%%%%%%%%%%%%%%%%%%%%%%%%%%%%%%%%%%%%%%%%%%%%%%%%%%%%

%%%%%%%%%%%%%%%%%%%%%%%%%%%%%%%%%%%%%%%%%%%%%%%%%%%%%%%%%%%%%%%%%%%%%%%%%%%%
\subsection{Remarks on the cases not covered by Theorem~\ref{thm:MainThm}}\label{sec:opencases}
%%%%%%%%%%%%%%%%%%%%%%%%%%%%%%%%%%%%%%%%%%%%%%%%%%%%%%%%%%%%%%%%%%%%%%%%%%%%%

Suppose $X$ is a del Pezzo surface of degree $4$ over a global field $k$ of characteristic not equal to $2$ with corresponding Severi-Brauer pencil $\calG$ such that the conditions of Corollary~\ref{cor:NonemptyAdelicBrSetG} are not satisfied. As noted in Remark~\ref{rmk:RemainingOpenCases} this implies that $\Br(\calG)/\Br_0(\calG)$ is cyclic of order $2$, with the nontrivial class represented by $\beta_\scrT$ for an irreducible subscheme $\scrT \subset \scrS$ of degree $2$ with $\N(\eps_\scrT) \in k^{\times2}$ for which $\#R_{\scrT}$ is odd. By Corollary~\ref{cor:ObsToWeakApproximation}, $\beta_\scrT$ obstructs weak approximation and so $\calG(\A_k)^{\Br} \ne \emptyset$ if and only if there exists a prime $v \in \Omega_k$ such that the evaluation map $\beta_\scrT:\calG(k_v) \to \Br(k_v)$ is not constant.

Let $\CQ{\scrT}' := \CQ{\scrT} + \left(\eps, -\Delta_{\scrT}\N(Q_{\infty}(v_\scrT))\right) \in \Br(k)$ be the class from Lemma~\ref{lem:betanonconstant} and define
\[
	R'_\scrT := \{ v \in \Omega_k \;:\; \eps_\scrT \notin\kk(\scrS_v)^{\times 2} \text{ and } \inv_v(\CQ{\scrT}') \ne 0 \}\,.
\]
Since $R_\scrT$ has odd cardinality, so too must $R'_\scrT$. In particular, $R'_\scrT$ is nonempty.  If $\scrT_v$ is reducible for a prime $v \in R'_\scrT$, then Lemma~\ref{lem:betanonconstant} shows that the evaluation map $\beta_\scrT \colon \calG(k_v) \to \Br(k_v)$ is nonconstant and so $\calG(\A_k)^{\Br} = \calG(\A_k)^{\beta_\scrT} \ne \emptyset$. If $\scrT_v$ is irreducible at $v \in R'_\scrT$, then Lemma~\ref{lem:betanonconstant} shows that $\beta_\scrT \otimes \kk(\scrT_v) \colon \calG(\kk(\scrT_v)) \to \Br(\kk(\scrT_v))$ is nonconstant. Indeed the lemma gives a $\kk(\scrT_v)$-point where $\beta_\scrT\otimes \kk(\scrT_v)$ takes the nonzero value $\CQ{\scrT_v}'$, but $\beta_\scrT \otimes \kk(\scrT_v)$ takes the value $0$ at any elements in the subset $\calG(k_v) \subset \calG(\kk(\scrT_v))$. Unfortunately, this is not enough to conclude that $\calG(\A_k)^{\Br}\ne \emptyset$ because $\beta_\scrT \colon \calG(k_v) \to \Br(k_v)$ can still be constant. Using the lemma below one can check that this occurs at $v = 5$ for the pencil of quadrics defined by
\[
    Q_{0} = -55x_1^2 + 2x_1x_2 + x_3^2 + 5x_4^2 \quad \textup{and}\quad
    Q_{\infty} = 33x_0^2 - 5x_1^2 - x_2^2 + 10x_3x_4\,.
\]
We note, however, that in this example (and in all others with $R_\scrT' \ne \emptyset$ that we have considered) there is some prime $w\in \Omega_k$ (in this case $w = 2$) for which the evaluation map $\beta_\scrT \colon  \calG(k_w) \to \Br(k_w)$ is not constant and, hence, $\calG(\A_k)^{\Br} \ne \emptyset$.

\begin{lemma}\label{lem:constantG}
	If $v \in R'_\scrT$ is such that $\scrT_v$ is irreducible, $k_v$ has odd residue characteristic and $X(\kk(\scrT_v)) = \emptyset$, then $\beta_\scrT\colon \calG(k_v) \to \Br(k_v)$ is constant.
\end{lemma}

\begin{proof}
	Suppose $X(\kk(\scrT_v)) = \emptyset$ and let $y \in \calG(k_v)$. Then $y$ corresponds to a quadratic point $\Spec(K) \to X$, with $K/k_v$ a quadratic field extension such that $K \ne \kk(\scrT_v)$. Since $k_v$ has odd residue characteristic, $\kk(\scrT_L)$ is the compositum of all quadratic extensions of $k_v$. In particular, it must contain a square root of $\eps_\scrT$ (since $\eps_\scrT \in k_v^\times\kk(\scrT_v)^{\times 2}$). Therefore, Lemma~\ref{lem:EvBetaSqrtEps} applies, and its conclusion shows that $\beta_\scrT(y)$ does not depend on $y$.
\end{proof}

In contrast, the following lemma shows that for $X$ (in place of $\calG$) nonconstancy of an evaluation map over an extension of $k_v$ does imply nonconstancy over $k_v$. 

    \begin{lemma}\label{lem:dp4ConstantEval}    
        Let $X$ be a del Pezzo surface of degree $4$ over a local field $k$ of characteristic not equal to $2$ such that $X(k) \neq \emptyset$.  If $\alpha\in \Br(X)$ is such that $\inv_k\circ \alpha \colon X(k) \to \Q/\Z$ is constant, then for all finite extensions $K/k$, $\inv_K\circ\alpha_K\colon X(K) \to \Q/\Z$ is constant and equal to $[K:k](\inv_k\circ \alpha)$.
    \end{lemma}
    \begin{remark}
        In the case that $k_v$ has odd residue characteristic, $\scrT_v$ is irreducible, $\eps_{\scrT_v}\in k_v^{\times2}$, and $\inv_v({\calC}_{\scrT}')\ne 0$, Lemma~\ref{lem:dp4ConstantEval} can be used to prove the converse of Lemma~\ref{lem:constantG}. Namely, if $\beta_\scrT$ is constant on $k_v$-points, then $X(\kk(\scrT_v))$ must be empty.
    \end{remark}
    \begin{proof}
    	Let $P \in X(k)$. By \cite{SalbergerSkorobogatov}*{Lemma 4.4} (which follows from~\cite{CTCoray}*{Theorem C}), every $0$-cycle of degree $0$ on $X$ is linearly equivalent to $Q - P$ for some $Q \in X(k)$. Therefore, for any closed point $R$ on $X$, there is some $Q \in X(k)$ such that $R \sim Q + (\deg(R)-1)P$.  Since evaluation of Brauer classes factors through rational equivalence and by assumption $\alpha(P) = \alpha(Q)$, we see that $\inv_K\circ\alpha_K= [K:k](\inv_k\circ \alpha)$ for any extension $K/k$.  
    \end{proof}
    
    \begin{rmks}\label{rmk:CTP2}\hfill
    	\begin{enumerate}
    	\item\label{rmkCTP2:1} The result of \cite{CTCoray} used in the proof above shows that every $0$-cycle of degree $1$ on a conic bundle with $5$ or fewer degenerate fibers is rationally equivalent to a rational point. The example mentioned just before Lemma~\ref{lem:constantG} shows that this does not extend to more general Severi-Brauer bundles (at least over a $p$-adic fields). Indeed, evaluation of Brauer classes factors through rational equivalence and in the example there is a Brauer class on $\calG$ which is nonconstant on $0$-cycles of degree $1$, but is constant on rational points. {An example of a Severi-Brauer bundle (in fact a conic bundle) with a \(0\)-cycle of degree \(1\) but no rational point was constructed by Colliot-Th\'el\`ene and Coray~\cite{CTCoray}*{Section 5}; in this example the conic bundle has \(6\) singular fibers.}
    	\item\label{rmkCTP2:2} If $X/k$ is a del Pezzo surface of degree $4$ over a number field which is a counterexample to the Hasse principle explained by the Brauer-Manin obstruction, then as shown in \cite{CTP}*{Section 3.5} there exists $\alpha \in \Br(X)$ such that $X(\A_k)^\alpha = \emptyset$ (a priori multiple elements of $\Br(X)$ might be required to give the obstruction). An immediate consequence of Lemma~\ref{lem:dp4ConstantEval} is that over any odd degree extension $K/k$ the same Brauer class will give an obstruction, i.e., $X_K(\A_K)^{\alpha_K} = \emptyset$. This answers a question posed in \cite{CTP}*{Remark 3, p. 95}. In particular, this shows that the conjecture that all failures of the Hasse principle for del Pezzo surfaces of degree $4$ are explained by the Brauer-Manin obstruction is compatible with the theorems of Amer, Brumer and Springer \cites{Amer,Brumer,Springer} which imply that an intersection of two quadrics with index $1$ has a rational point.
    	\end{enumerate}
    	
    \end{rmks}
 
%%%%%%%%%%%%%%%%%%%%%%%%%%%%%%%%%%%%%%%%%%%%%%%%%%%%%%%%%%%%%%%%%%%%%%%%%%%%%%%%
\subsection{A degree \texorpdfstring{$4$}{4} del Pezzo surface with obstructions only over odd degree extensions}\label{sec:BSDexample}
%%%%%%%%%%%%%%%%%%%%%%%%%%%%%%%%%%%%%%%%%%%%%%%%%%%%%%%%%%%%%%%%%%%%%%%%%%%%%%%%
	
	\begin{prop}
		Let $X/\Q$ be the del Pezzo surface of degree $4$ given by the vanishing of
		\[
		Q_0 = (x_0+x_1)(x_0+2x_1) - x_2^2  + 5x_4^2\,,\quad \text{and} \quad Q_{1}= 2(x_0x_1 - x_2^2 + 5x_3^2)\,.
		\]
		For any finite extension $K/\Q$ we have $X_K(\A_K)^{\Br} = \emptyset$ if and only if $[K:\Q]$ is odd.
	\end{prop}
	
	\begin{proof}
		This surface was considered by Birch and Swinnerton-Dyer \cite{BSD} who showed that $X$ is a counterexample to the Hasse principle explained by the Brauer-Manin obstruction. It follows from Lemma~\ref{lem:dp4ConstantEval} that for any $K$ with $[K:\Q]$ odd, $X_K$ is also a counterexample to the Hasse principle explained by the Brauer-Manin obstruction. 
		
		Since $X$ is locally solvable over $\Q$, $\Br(X)/\Br_0(X)$ is generated by the image of $\Br(X)[2]$. 
		The singular quadrics in the pencil lie above  $\scrS(\Qbar) = \{ 0,\pm1,\frac{\pm4\sqrt{2} + 5}{7} \} \subset \PP^1$ and the corresponding discriminants satisfy $\eps_0 = \eps_1 = 5$, $\eps_{-1} = -1$ and $\N(\eps_{(\pm4\sqrt{2} + 5)/7}) = -1$. For any $K/\Q$ linearly disjoint from $k_1 = \Q(\sqrt{-1},\sqrt{2},\sqrt{5})$, the restriction map induces an isomorphism $\Br(X)/\Br_0(X) \simeq \Br(X_K)/\Br_0(X_K)$ and so $X_K(\A_K)^{\Br} \ne \emptyset$ by Lemma~\ref{lem:CorCommutesWithEval}\eqref{part:ConstEvalLeadsToNoObstructions}. On the other hand, if $K/\Q$ is not linearly disjoint from $k_1$, we can check directly that $X(K) \ne \emptyset$. Indeed, $K$ must contain $\Q(\sqrt{d})$ for some $d \in \{ -1,\pm 2,\pm 5, \pm 10\}$. Over these quadratic fields one can exhibit points:
	\[
		(1:1:1:0:\sqrt{-1})\,,\,
		(1:-2:2\sqrt{2}:\sqrt{2}:1) \,,\,
		(4:9:6:0:5\sqrt{-2}) \,,\,
		(0:0:\sqrt{5}:1:1)\,,\,
	\]
	\[
		(5:0:0:0:\sqrt{-5}) \,,\,
		(2\sqrt{10}:-\sqrt{10}:0:2:0)\,,\,
		(0:\sqrt{-10}:0:0:2)\,.\qedhere
	\]
	\end{proof}

%%%%%%%%%%%%%%%%%%%%%%%%%%%%%%%%%%%%%%%%%%%%%%%%%%%%%%%%%%%%%%%%%%%%%%%%%%%%%%%%
\subsection{A degree \texorpdfstring{$4$}{4} del Pezzo surface with index \texorpdfstring{\(4\)}{4}}\label{sec:ind4}
%%%%%%%%%%%%%%%%%%%%%%%%%%%%%%%%%%%%%%%%%%%%%%%%%%%%%%%%%%%%%%%%%%%%%%%%%%%%%%%%

\begin{thm}\label{thm:ind4}
	There exists a del Pezzo surface $X$ of degree $4$ over a field $k$ of characteristic $0$ such that $X$ has index $4$.
\end{thm}

\begin{proof}
	Let $k_0$ be an algebraically closed field of characteristic $0$. For $i = 1,\dots,2g$, set $k_i := k_{i-1}((t_i))$ and set $k := k_{2g}$. By a result of Lang and Tate \cite{LangTate}*{p. 678}, if $A/k_0$ is an abelian variety of dimension $g$ and $n$ is an integer, then there exists a torsor under $A_k = A \times_{k_0} \Spec(k)$ of period $n$ and index $n^{2g}$. In particular, if $C/k_0$ is any genus $2$ curve, then there exists a torsor under the Jacobian $J = \Jac(C_k)$ of $C_k$ of period $2$ and index $16$. Since $C$ is defined over the algebraically closed field $k_0$, it has a rational Weierstrass point over $k$. As observed by Flynn \cite{Flynn}, and worked out in detail by Skorobogatov \cite{SkoDP4}, if $J_\lambda$ is a $2$-covering $\pi_\lambda:J_\lambda \to J$ (i.e., a twist of $[2]:J \to J$ corresponding to $\lambda \in \HH^1(k,J[2])$), then there are morphisms
\[
	J_\lambda \leftarrow \tilde{J}_\lambda \rightarrow Z_\lambda \rightarrow X_\lambda\,,
\]
where $\tilde{J}_\lambda \to J_\lambda$ is the blow up of $J_\lambda$ at $\pi_\lambda^{-1}(0_J)$, $Z_\lambda$ is the desingularized Kummer variety associated to $J_\lambda$ and $Z_\lambda \to X_\lambda$ is a double cover of a del Pezzo surface of degree $4$. In particular, there is a degree $4$ morphism $\tilde{J}_\lambda \to X_\lambda$. So the index of $X_\lambda$ is at least $\operatorname{index}(J_\lambda)/4$, which will equal $4$ for suitable choice of $\lambda$ by the aforementioned result of Lang and Tate.
\end{proof}

\begin{theorem}\label{thm:PI}
	Suppose $k$ is a number field and $Y$ is a torsor of period $2$ under the Jacobian of a genus $2$ curve over $k$ with a rational Weierstrass point. The index of $Y$ divides $8$.
\end{theorem}

\begin{proof}
	As in the proof of the previous theorem, the index of $Y$ divides $4 \operatorname{index}(X)$ for some del Pezzo surface $X$ of degree $4$. The result follows from Theorem~\ref{thm:MainIndThm}.
\end{proof}

\begin{rmks}\label{rmk:PI}\hfill
	\begin{enumerate}
		\item The conclusion of Theorem~\ref{thm:PI} was known to hold by work of Clark \cite{Clark}*{Theorems 2 and 3} when $k$ is a $p$-adic field and when $k$ is a number field and $Y$ is locally solvable.
		
        \item Arguing as in the proof of the theorem we see that the Kummer variety $Z_\lambda$ has index dividing $4$ when $k$ is a local or global field. This is lower than one would expect, given that $Z_\lambda$ is an intersection of $3$ quadrics in $\PP^5_k$.
		\item The result of Lang-Tate quoted in the proof above shows that over general fields of characteristic $0$, there are examples where $Z_\lambda$ and $Y$ have index $8$ and $16$, respectively.
		
        \item In response to a preliminary report on this work by the authors, John Ottem suggested the following alternate proof of Theorem~\ref{thm:ind4} {which gives an example over the \(C_3\) field \(\kk(\PP^3_{\C})\)}.  Let \(D_1,D_2\subset \PP^3\times \PP^4\) be two general $(2,2)$ divisors over $\C$, and let \(Y = D_1\cap D_2\).  Then, by the Lefschetz hyperplane theorem (applied twice), restriction gives an isomorphism $\HH^4(\PP^3\times \PP^4,\Z) \stackrel{\sim}{\to}\HH^4(D_1, \Z)\stackrel{\sim}{\to} \HH^4(Y, \Z)$.  Note that the generic fiber of the first projection is a del Pezzo surface of degree $4$ over $\kk(\PP^3)$.  Hence any threefold $V\subset Y$ can be expressed as $aH_1^2 + bH_1H_2 + cH_2^2$, where $H_i$ denotes the pullback of $\OO(1)$ under the projection $\pi_i$.  Then the degree of $V\to \PP^3$ is given by
		\[
            V.H_1^3 = V.H_1^3.X = (a H_1^2+b H_1H_2+cH_2^2). H_1^3.(2H_1+2H_2)^2, 
        \]
        which must be divisible by $4$.  Thus $Y_{\kk(\PP^3)}$ has index $4$. {Note that to apply the Lefschetz hyperplane theorem, we need \(\dim D_i> 5\), so this argument does not extend to \(\kk(\PP^2_\C)\).}
        \label{item:Ottem} 
        
        This construction suggested by Ottem generalizes to arbitrary complete intersections.  Namely, given a sequence of degrees $(d_1, \dots, d_r)$ and an ambient dimension $N$, one can consider an intersection of general $(d_1, d_1), (d_2, d_2), \dots, (d_r,d_r)$ hypersurfaces in $\PP^M\times \PP^N$.  If $M > N-r$, then the same argument as above yields a $(d_1, \dots, d_r)$ smooth complete intersection $Y\subset \PP^N_{\kk(\PP^M)}$ with index $d_1d_2\cdots d_r$.

        \item\label{rmk:Wittenberg} After viewing an early draft of this paper, Olivier Wittenberg shared a correspondence of his from 2013~\cite{Wittenberg-PersonalCommunication} that provides yet another construction that proves Theorem~\ref{thm:ind4}.  {Let \(k\) be any field of characteristic different \(2\) such that there exists a quadric surface \(Q\) with no \(k\)-points that remains pointless after a quadratic extension \(k'/k\). Wittenberg's construction gives an example over the field \(k((t))\).
        
        Let \(f\) be a general rank \(2\) quadric in \(\PP^4\) that splits over \(k'\).  Then for a general quadric \(g\), the intersection \(Q\cap V(f + tg)\) is a smooth del Pezzo surface of degree \(4\) that has index \(4\) over \(k((t))\).  Indeed, the smooth locus of the special fiber has index \(4\) by construction, so (for general enough \(g\)), the general fiber must also have index \(4\).  This construction of Wittenberg extends to give complete intersections of $n$ quadrics with index $2^n$ (over fields of larger transcendence degree).} \label{item:Wittenberg}

	\end{enumerate}
\end{rmks}

 %%%%%%%%%%%%%%%%%%%%%%%%%%%%%%%%%%%%%%%%%%%%%%%%%%%%%%%%%%%%%%%%%%%%%%%%%%%%
\subsection{The index of a degree \texorpdfstring{$d$}{d} del Pezzo surface}\label{sec:dP}
%%%%%%%%%%%%%%%%%%%%%%%%%%%%%%%%%%%%%%%%%%%%%%%%%%%%%%%%%%%%%%%%%%%%%%%%%%%%
    
    The following table gives sharp upper bounds for the indices of degree $d$ del Pezzo surfaces over local fields, number fields and arbitrary fields of characteristic $0$. The entries in the column $d = 4$ are a consequence of the results in this paper, while for $d \ne 4$, they can be deduced fairly easily from known results as described below.
    	
    	\begin{center}
    \begin{tabular}{|r||c|c|c|l|c|l|c|c|c|}\hline
    $d$ & $9$& $8$ & $7$ & $6$ & $5$ & $4$ & $3$ & $2$ & $1$\\
    \hline\hline
    $k$      arbitrary & $3$ & $4$ & $1$ & $6$       & $1$ & $4$ [Thm.~\ref{thm:ind4}]& $3$ & $2$ & $1$ \\
    \hline
    $k$ a number field & $3$ & $2$ & $1$ & $6$       & $1$ & $2$ [Thm.~\ref{thm:MainIndThm}] & $3$ & $2$  & $1$ \\
    \hline
    $k$ a local field  & $3$ & $2$ & $1$ & $2$ or $3$ & $1$ & $2$ [Thm.~\ref{thm:MainIndThm}]& $3$ & $2$ & $1$ \\
    \hline
    \end{tabular}
    \end{center}

	When $d = 9$, $Y$ is a Severi-Brauer surface and so the index of $Y$ divides $3$ and examples of index $3$ exist whenever $\Br(k)$ contains an element of order $3$. 
	
	 When $d = 8$, $Y = \Res_{L/k}(C)$ is the restriction of scalars of a conic $C/L$ defined over a degree $2$ \'etale algebra $L/k$~\cite{Poonen-Qpoints}*{Prop. 9.4.12}. Since the conic has a point over some quadratic extension $L'/L$, the index of $Y$ divides $4$ and over general fields there are examples with index $4$. Over local and global fields however, the index must divide $2$. Indeed, in this case $C$ will have a point over a quadratic extension $L'/L$ of the form $L' = k' \otimes_k L$ for some quadratic extension $k'/k$. The universal property of restriction of scalars then gives $Y(k') \ne \emptyset$, showing that the index divides $2$.
	
	When $d = 7$, $Y(k) \ne \emptyset$ over any field $k$ and so the index is always equal to $1$. The same applies to $d = 1,5$ (see, e.g.,~\cite{Poonen-Qpoints}*{Thm 9.4.8 and Section 9.4.11}).
	
	For $d = 6$, $Y$ is determined by a $\Gal(L/k)$-stable triple of geometric points on a Severi-Brauer surface $S/L$ over a quadratic \'etale algebra $L/k$ such that if $S \not\simeq \PP^2_L$ then the class of $S$ in the Brauer group does not lie in the image of $\Br(k) \to \Br(L)$~\cite{Corn}. If $k$ is a local field and $L$ is a quadratic field extension, then the map $\Br(k) \to \Br(L)$ is an isomorphism, so either $S = \PP^2_L$ (in which case $Y$ has index dividing $2$) or $L = k\times k$ in which case the index of $Y$ divides $3$. One can construct examples of index $6$ over number fields, by arranging to have index $2$ at one completion and index $3$ at another.
	
	For $d = 3$ and $k$ local, index $1$ implies the existence of a $k$-rational point~\cite{Coray}, and so a cubic surface without points over some local field has index $3$. This gives examples of index $3$ over number fields as well. 
	
	For $d = 2$, index $2$ examples can be obtained by blowing up a degree $4$ del Pezzo surface of index $2$ at a quadratic point. By Theorems~\ref{thm:MainIndThm} and~\ref{thm:MainThm}, any del Pezzo surface of degree $4$ without points over a local field gives such an example. The surface considered in Section~\ref{sec:BSDexample} gives an example over a number field.

%%%%%%%%%%%%%%%%%%%%%%%%%%%%%%%%%%%%%%%%%%%%%%%%%%%%%%%%%%%%%%%%%%%%%%%%%%%%%%%%
%%%%%%%%%%%%%%%%%%%%%%%%%%%%%%%%%%%%%%%%%%%%%%%%%%%%%%%%%%%%%%%%%%%%%%%%%%%%%%%%

\end{document}